\documentclass[11pt,letterpaper]{amsart}%
\usepackage{amsmath}%
\usepackage{amssymb}%
\usepackage{amscd}%
\usepackage{color}%
\usepackage{graphicx}%
\usepackage{mathrsfs}%
\usepackage[all]{xy}%
\usepackage{newtxtext,newtxmath}%
\usepackage[all, warning]{onlyamsmath}
\usepackage{amsrefs}
\usepackage{bbm}
\usepackage{scalerel}

\def\ol#1{\overline{#1}}
\def\wt#1{\widetilde{#1}}
\theoremstyle{plain}
    \newtheorem{theorem}{Theorem}[section]
    \newtheorem{theorem*}{Theorem}
    \newtheorem{proposition}[theorem]{Proposition}
    \newtheorem{proposition-definition}[theorem]{Proposition-Definition}
    \newtheorem{lemma}[theorem]{Lemma}
    \newtheorem{corollary}[theorem]{Corollary}
      
      \newtheorem{conjecture}[theorem]{Conjecture}
\theoremstyle{definition}
    \newtheorem{definition}[theorem]{Definition}
    
    \newtheorem{example}[theorem]{Example}
    \newtheorem{remark}[theorem]{Remark}
    \newtheorem{hypothesis}[theorem]{Hypothesis}
    \newtheorem{condition}[theorem]{Condition}
\def\Alphabet{A,B,C,D,E,F,G,H,I,J,K,L,M,N,O,P,Q,R,S,T,U,V,W,X,Y,Z}
\def\alphabet{a,b,c,d,e,f,g,h,i,j,k,l,m,n,o,p,q,r,s,t,u,v,w,x,y,z}
\def\endpiece{xxx}
\def\makeAlphabet[#1]{\expandafter\makeA#1,xxx,}
\def\makealphabet[#1]{\expandafter\makea#1,xxx,}
\def\makeA#1,{\def\temp{#1}\ifx\temp\endpiece\else%
\mkbb{#1}\mkfrak{#1}\mkbf{#1}\mkcal{#1}\mkscr{#1}\mkbs{#1}\expandafter\makeA\fi}%
\def\makea#1,{\def\temp{#1}\ifx\temp\endpiece\else\mkfrak{#1}\mkbf{#1}\mkbs{#1}\expandafter\makea\fi}%
\def\mkbb#1{\expandafter\def\csname bb#1\endcsname{\mathbb{#1}}}
\def\mkfrak#1{\expandafter\def\csname fr#1\endcsname{\mathfrak{#1}}}
\def\mkbf#1{\expandafter\def\csname b#1\endcsname{\mathbf{#1}}}
\def\mkcal#1{\expandafter\def\csname c#1\endcsname{\mathcal{#1}}}
\def\mkscr#1{\expandafter\def\csname s#1\endcsname{\mathscr{#1}}}
\def\mkbs#1{\expandafter\def\csname bs#1\endcsname{{\boldsymbol{#1}}}}
\def\makeop[#1]{\xmakeop#1,xxx,}
\def\mkop#1{\expandafter\def\csname #1\endcsname{{\mathrm{#1}}}} %
\def\xmakeop#1,{\def\temp{#1}\ifx\temp\endpiece\else\mkop{#1}\expandafter\xmakeop\fi}%
\def\makeup[#1]{\xmakeup#1,xxx,}
\def\mkup#1{\expandafter\def\csname #1\endcsname{{\mathrm{#1}\,}}} %
\def\xmakeup#1,{\def\temp{#1}\ifx\temp\endpiece\else\mkup{#1}\expandafter\xmakeup\fi}%
\makeAlphabet[\Alphabet]
\makealphabet[\alphabet]
\makeop[Alt,End,Ext,Hom,Sym,Tor,dR,Li,Cone,res,Res,id,Gr,MHS,Re,Vec,Rep,Rex,bil,GL,Lie,pr,Gd,OF,Dec,Tr,Cl,alt,sgn,For,Gal,VMHS,MHM,ad,sm,Ind,Pro,Hdg,abs,Ho,MHC,AHC]
\makeup[Spec,Ker,Coker,Im,Proj,Tot]

\def\cLog{\cL\!\operatorname{og}}

\def\sLog{\sL\!\!\operatorname{og}}
\def\bbLog{\bbL\!\operatorname{og}}
\def\bsalpha{{\boldsymbol{\alpha}}}
\def\bsbeta{{\boldsymbol{\beta}}}

\def\absv#1{|#1|}
\def\pol{\boldsymbol{\operatorname{pol}}}
\def\bone{\mathbbm{1}}

\def\pair#1{\langle#1\rangle}
\def\bra#1{{\langle #1\rangle}}
\def\al{\alpha}

\def\bsal{\bsalpha}

\def\T{{\bbT}}

\def\b{\bullet}
\def\a{\bsalpha}
\def\Ft{F^\times_+}

\newcommand{\isomto}{\stackrel{\cong}{\to}}
\def\Nekovar{Nekov\'a\v r }
\def\hyp{\text{-}}

\def\Hdg{{}}
\def\ve{\varepsilon}

\newcommand{\dlog}{d\!\log}
\newcommand{\lra}{\longrightarrow}
\newcommand{\eps}{\varepsilon}
\renewcommand{\det}{\operatorname{\mathrm{det}}}

\newcommand{\Nr}{\mathrm{N}}

\DeclareMathOperator*{\bigboxtimes}{\scalerel*{\boxtimes}{\sum}}

\newcommand{\hooklongrightarrow}{\lhook\joinrel\longrightarrow}

\begin{document}

\title[Polylogarithm and the Shintani Generating Class]{The Hodge Realization of the 
Polylogarithm and the Shintani Generating Class for Totally Real Fields}
\author[Bannai]{Kenichi Bannai$^{*\diamond}$}
\email{bannai@math.keio.ac.jp}
\address{${}^*$Department of Mathematics, Faculty of Science and Technology, Keio University, 3-14-1 Hiyoshi, Kouhoku-ku, Yokohama 223-8522, Japan}
\address{${}^\diamond$Mathematical Science Team, RIKEN Center for Advanced Intelligence Project (AIP),1-4-1 Nihonbashi, Chuo-ku, Tokyo 103-0027, Japan}
\address{${}^\circ$Department of Mathematics, Cooperative Faculty of Education, Gunma University,
Maebashi, Gunma 371-8510, Japan}

\author[Bekki]{Hohto Bekki$^{*}$}
\author[Hagihara]{Kei Hagihara$^{*}$}
\author[Ohshita]{Tatsuya Ohshita$^\circ$}
\author[Yamada]{Kazuki Yamada$^*$}
\author[Yamamoto]{Shuji Yamamoto$^{*}$}

\date{\today}
\date{\today\quad(Version $2.0$)}
\begin{abstract}
	In this article, 
	we construct the Hodge realization of the \emph{polylogarithm class} in the equivariant Deligne--Beilinson cohomology 
	of a certain algebraic torus associated to a totally real field. 
	We then prove that the de Rham realization of this polylogarithm gives the \textit{Shintani generating class}, 
	a cohomology class generating the values of the Lerch zeta functions of the totally real field at nonpositive integers. 
	Inspired by this result, we give a conjecture concerning the specialization 
	of this polylogarithm class at torsion points, 
	and discuss its relation to the Beilinson conjecture for Hecke characters of totally real fields.
\end{abstract}
\thanks{This research is supported by KAKENHI JP18H05233, JP20J01008, 
JP20K14295 and JP21K18577.
The topic of research initiated from the KiPAS program FY2014--2018 of the Faculty of Science and Technology at Keio University.}
\maketitle 
\setcounter{tocdepth}{1}
\setcounter{section}{0}

\tableofcontents 

%
%
%
%
%
\section{Introduction}
%
%
%
%
%

The \textit{Beilinson conjecture} for algebraic varieties and more general 
motives defined over algebraic number fields 
is a conjecture relating the special values of Hasse-Weil $L$-functions 
to the determinant of the regulator map from motivic cohomology to Deligne--Beilinson cohomology. 
This conjecture is a vast generalization of the Gross conjecture for special values of Artin $L$-functions \cite{Gro79}, 
based on the work by Borel \cite{Bor77} on special values of Dedeking zeta functions. 
The general strategy for the proof of this conjecture is to construct explicit elements in motivic cohomology 
whose image by the regulator map is amenable to calculation. 
Beilinson originally proved the Beilinson conjecture for the noncritical values of Dirichlet $L$-functions 
by constructing the \textit{cyclotomic elements} in the motivic cohomology, or equivalently $K$-groups, 
associated to the spectrum of cyclotomic fields \cites{Bei84,Neu88}.  
Beilinson and Deligne provided a universal construction of the cyclotomic elements 
in terms of the \textit{motivic polylogarithm}, a class constructed 
in the motivic cohomology of the projective line minus three points. 
This allows for a simplified and conceptual proof of the Beilinson conjecture for 
the Dirichlet $L$-function, by reducing the complicated 
calculation of the regulator to the calculation of the Hodge realization of the polylogarithm class 
\cites{BD94,Del89,HW98}. 
Although perhaps not written explicitly in literature, 
similar strategy may be applied to the proof of the Beilinson conjecture for Hecke character
associated to imaginary quadratic fields \cite{Den97}, originally proved by Deninger \cites{Den89,Den90},
using the elliptic polylogarithm constructed by Beilinson and Levin \cite{BL94}.

In this article, we prospect a similar strategy for Hecke $L$-functions
of totally real fields, using an equivariant version of the polylogarithm for certain algebraic
tori associated to totally real fields. 
Let $F$ be a totally real field and $g\coloneqq[F\colon\bbQ]$.  
We denote by $\frI$ the set of nonzero fractional ideals of $F$,
and by $\Ft$ the multiplicative group of totally positive elements of $F$.
For any $\fra\in\frI$, we let
$
	\T^\fra\coloneqq\Hom(\fra,\bbG_m),
$
which is a $g$-dimensional affine group scheme over $\bbQ$. 
If we let 
\begin{align*}
	U&\coloneqq\coprod_{\fra\in\frI}U^\fra
\end{align*}
for $U^\fra\coloneqq\T^\fra\setminus\{1\}$, then $U$ has a natural action of $\Ft$. 
We regard the quotient stack $U/\Ft$ as a generalization of $\bbP^1\setminus\{0,1,\infty\}$; 
indeed, when $F=\bbQ$, we have $U/\Ft\cong\bbG_m\setminus\{1\}=\bbP^1\setminus\{0,1,\infty\}$. 

The polylogarithm classes of general commutative group schemes were constructed by Huber and Kings \cite{HK18}. 
We consider an equivariant version of their construction in the case of $U/\Ft$. 
More precisely, we construct an element of the equivariant Deligne--Beilinson cohomology $H^{2g-1}_{\sD}(U/\Ft,\bbLog)$ 
with coefficients in the logarithm sheaf $\bbLog$, which is a certain equivariant 
pro-unipotent variation of mixed $\bbR$-Hodge structures on $U$. 
Since the general theory of equivariant variations of mixed $\bbR$-Hodge structures 
and equivariant Deligne-Beilinson cohomology
are not yet sufficiently developed,
we give a definition of equivariant Deligne-Beilinson cohomology $H^{m}_{\sD}(U/\Ft,\bbLog)$
specific to our case,
using the logarithmic Dolbeault complex
 (see Propositions \ref{prop: construction} and \ref{prop: A11}).
 
 We define the polylogarithm class in equivariant Deligne-Beilinson cohomology as follows.

\begin{definition}[= Definition \ref{def: polylog class}]\label{def: polylog}
	Let $\Cl^+_F$ denote the narrow ideal class group of $F$. 
	We define the polylogarithm class 
	\[
		\pol\in H^{2g-1}_{\sD}(U/\Ft,\bbLog)
	\]
	to be the class in the equivariant Deligne--Beilinson cohomology of $U$ with coefficients in $\bbLog$
	that maps to $(1,\ldots,1)\in \bigoplus_{\Cl^+_F}\bbR$ through the isomorphism
	$H^{2g-1}_{\sD}(U/\Ft,\bbLog)\cong\bigoplus_{\Cl^+_F}\bbR$ given in Proposition \ref{prop: key}. 
\end{definition}

When $F=\bbQ$, the polylogarithm class $\pol$ coincides with the Hodge realization of the polylogarithm class 
constructed by Beilinson and Deligne on the projective line minus three points.

Previously in \cites{BHY19,BHY20}, we constructed the Shintani generating class,
which is a certain equivariant cohomology class 
on $U$ generating the values of the Lerch zeta function of $F$ at nonpositive integers.
Our main theorem relates our polylogarithm class with the Shintani generating class.

\begin{theorem}[= Theorem \ref{thm: main}]\label{thm: intro}
	Under a natural inclusion
	\[
		H^{2g-1}_{\sD}(U/\Ft,\bbLog)\hookrightarrow H^{2g-1}_\dR(U/\Ft,\bbLog\otimes\sO_U), 
	\]
	the polylogarithm class $\pol$ maps to the de Rham Shintani class $\cS$, 
	constructed from the Shintani generating class.
\end{theorem}

Theorem \ref{thm: intro} indicates that our polylogarithm class is related to 
the critical values of Hecke $L$-functions, 
since these values are written in terms of the Lerch zeta values at nonpositive integers. 
We further expect that our polylogarithm class is related to the Lerch zeta values at positive integers, 
and hence to the \emph{noncritical} Hecke $L$-values. 
Let $\xi$ be a nontrivial torsion point in $U$, 
and let $\xi\Delta$ be the $\Delta$ orbit of $\xi$ with respect to the action of $\Delta\coloneqq\cO_{F+}^\times$.
Assuming further the existence of a conjectural \textit{equivariant plectic Hodge theory}, 
fitting in with the formalism of the category of mixed plectic $\bbR$-Hodge 
structures $\MHS^{\boxtimes I}_{\bbR}$ defined by \Nekovar and Scholl \cite{NS16},
we may define the specialization
\begin{equation}\label{eq: specialization}
	i_{\xi\Delta}^*\colon H^{2g-1}_{\sD}(U/\Ft,\bbLog)
	\rightarrow H^{2g-1}_{\sD^I}(\xi\Delta/\Delta,i_{\xi\Delta}^*\bbLog)
	\cong\prod_{n=1}^\infty(2\pi i)^{(n-1)g}\bbR
\end{equation}
with respect to the equivariant morphism 
$i_{\xi\Delta}\colon\xi\Delta\rightarrow U$. 
We propose the following: 

\begin{conjecture}[= Conjecture \ref{conj: i^*pol}]\label{conj}
	Let $\xi$ be a torsion point in $U$, and we let $\xi\Delta$ be the
	orbit of $\xi$ in $U$. Then 
	$i_{\xi\Delta}^*\pol\in H^{2g-1}_{\sD^I}(\xi\Delta/\Delta,i_{\xi\Delta}^*\bbLog)
	=\prod_{n=1}^\infty(2\pi i)^{(n-1)g}\bbR$ 
	satisfies
	\[
		i_{\xi\Delta}^*\pol=(d_F^{1/2}\cL^\infty(\xi\Delta,n))_{n=1}^\infty\in\prod_{n=1}^\infty(2\pi i)^{(n-1)g}\bbR. 
	\]
	Here $d_F$ denotes the discriminant of $F$ and 
	\[
		\cL^\infty(\xi\Delta,n)
		\coloneqq\begin{cases}
		     \Re \cL(\xi\Delta,n) & \text{if $(n-1)g$ is even}, \\
		     i\,\Im \cL(\xi\Delta,n) & \text{if $(n-1)g$ is odd}, 
		\end{cases}
    \]
	where $\cL(\xi\Delta,s)$ is the Lerch zeta function corresponding to the point $\xi$
	(see Definition \ref{def: Lerch}).
\end{conjecture}

In a subsequent research, we will also consider a syntomic version of this conjecture.

If $F=\bbQ$, then the Lerch zeta function for a nontrivial torsion point $\xi\in\bbG_m(\bbC)$ is given as
$\cL(\xi,s)=\sum_{n=1}^\infty \xi^n n^{-s}$, 
hence we have $\cL(\xi,k)=\Li_k(\xi)$, where $\Li_k(\xi)$ is the value at $\xi$ of the $k$-th polylogarithm function
given by the series $\Li_k(t)\coloneqq\sum_{n=1}^\infty t^n n^{-k}$. 
Therefore Conjecture \ref{conj} is a direct generalization
of the result of Beilinson and Deligne on the calculation of the specialization of $\pol$ at 
nontrivial roots of unity. Recall that this calculation and the motivicity of the polylogarithm class 
lead to a conceptual proof of the Beilinson conjecture for the Dirichlet $L$-functions. 
In \S\ref{subsec: plectic roadmap}, we will explore how Conjecture \ref{conj} and some kind of motivicity of $\pol$
imply the Beilinson conjecture for the Hecke $L$-function of totally real fields. 

We remark that Beilinson, Kings and Levin \cite{BKL18} 
introduced the topological polylogarithm of totally real fields 
and investigated its relation to the critical Hecke $L$-values. 
It seems interesting to clarify the precise relation between their topological polylogarithm 
and our polylogarithm class $\pol$. 

The precise contents of this article are as follows.
In \S\ref{sec: log}, we give the definition of equivariant variation of mixed Hodge structures, 
and discuss their cohomology theory under a basic assumption (Condition \ref{cond: RGamma_Hdg}). 
In \S\ref{sec: log sheaf}, we introduce the logarithm sheaf $\bbLog$, 
which is an equivariant variation of mixed Hodge structures on $\bbT=\coprod \bbT^\fra$. 
In \S\ref{sec: polylog}, we construct the polylogarithm class $\pol$ in the equivariant 
Deligne--Beilinson cohomology of $U$ with coefficients in $\bbLog$. 
The construction of this cohomology group itself is postponed to the Appendix. 
In \S\ref{sec: shintani}, we recall the construction of the Shintani generating class $\cG$, 
and introduce its modification $\cS$, the de Rham Shintani class. 
Then we give a precise relation between $\pol$ and $\cS$ (Theorem \ref{thm: intro}). 
In \S\ref{sec: plectic}, assuming the existence of a suitable formalism of plectic Hodge theory, 
we construct the plectic polylogarithm in the equivariant plectic Deligne--Beilinson cohomology
$H^{2g-1}_{\sD^I}(U/\Ft,\bbLog)$.  In fact, we prove that this cohomology is isomorphic 
to the (non-plectic) equivariant Deligne--Beilinson cohomology $H^{2g-1}_{\sD}(U/\Ft,\bbLog)$, and the 
plectic polylogarithm coincides with the polylogarithm of Definition \ref{def: polylog}
through this isomorphism (see Remark \ref{rem: same}).
We then state Conjecture \ref{conj} concerning the specialization of the polylogarithm class
and discuss its relation to the Beilinson conjecture for the Hecke $L$-function of totally real fields. 

\subsection{Notation and Convention}
In this article, we will consider certain kinds of ``spaces'' (e.g., schemes or complex manifolds) 
and ``sheaves'' on them (e.g., $\sO$-modules, local systems or variations of mixed $\bbR$-Hodge structures). 
In general, the pullback functor on the category of such sheaves along a morphism $f$ of such spaces is 
denoted by $f^*$, as opposed to $f^{-1}$ which denotes the pullback of the underlying abelian sheaves. 
For example, when $f\colon X\to Y$ is a morphism of schemes, we have 
$f^*\sF=f^{-1}\sF\otimes_{f^{-1}\sO_Y}\sO_X$ for an $\sO_Y$-module $\sF$. 

We often regard a smooth algebraic variety $X$ over $\bbC$ as a complex manifold, 
and denote by $\Omega^\b_X$ the complex of sheaves of holomorphic differential forms on $X$.

The symbol $\bbN$ denotes the set of non-negative integers: $\bbN=\{0,1,2,\ldots\}$.

%
%
%
%
%
\section{Equivariant Variations of Mixed Hodge Structures}\label{sec: log}
%
%
%
%
%

In this section, we briefly review the theory of variations of mixed Hodge structures 
and introduce its equivariant version under a certain hypothetical condition (Condition \ref{cond: RGamma_Hdg}). 
In particular, we give a construction of equivariant Deligne--Beilinson cohomology with coefficients 
assuming this condition.  Since the general theory of equivariant variations of mixed $\bbR$-Hodge structures
is not yet fully developed, the above condition has not yet been verified.  In this article, motivated by the above
consideration, we will use a definition of equivariant Deligne-Beilinson cohomology specific to our case
which we give in the Appendix (Definition \ref{def: A10}, see also Proposition \ref{prop: construction}).

%
%
\subsection{Variations of Mixed $\bbR$-Hodge Structures}\label{subsec: variations}
%
%

In this subsection, we review the basics of the category of variations of 
mixed $\bbR$-Hodge structures.
See for example \cite{BZ90}, or \cite{BHY18}*{\S2} for details.

\begin{definition}\label{def: MHS}
    Let $V=(V,W_\bullet,F^\bullet)$ be a triple consisting of a finite dimensional $\bbR$-vector space $V$, 
    a finite ascending filtration $W_\bullet$ on $V$ by $\bbR$-linear subspaces, and a finite descending 
    filtration $F^\bullet$ on $V_\bbC\coloneqq V\otimes\bbC$ by $\bbC$-linear subspaces.
    We say that $V$ is a \textit{mixed $\bbR$-Hodge structure} if 
    \[
    	\Gr^W_nV_\bbC=\bigoplus_{p+q=n}(F^p\cap\bar F^q)\Gr^W_nV_\bbC
    \]
    holds for each $n\in\bbZ$, where $\Gr^W_nV_\bbC\coloneqq W_nV_\bbC/W_{n-1}V_\bbC$ and 
    the filtrations $F^\b$ and $\bar F^\b$ on $\Gr^W_nV_\bbC$ are the filtrations induced by $F^\b$ on $V_\bbC$
    and its complex conjugate $\bar F^\b$.
    We denote by $\MHS_{\bbR}$ the category of mixed $\bbR$-Hodge structures, whose morphisms are $\bbR$-linear
    maps of the underlying $\bbR$-vector spaces compatible with the filtrations $W_\bullet$ and $F^\bullet$.
    We say that $V$ is \emph{pure of weight $n$} if $V=\Gr^W_nV$.
\end{definition}

\begin{example}\label{ex: Tate object}
    For each $k\in\bbZ$, the \emph{Tate object} $\bbR(k)$ in $\MHS_\bbR$ is defined as follows. 
    We set $\bbR(k)\coloneqq (2\pi i)^k\bbR$ as an $\bbR$-vector space, and define the filtrations by 
    \[W_{-2k-1}\bbR(k)\coloneqq 0,\quad W_{-2k}\bbR(k)\coloneqq\bbR(k),\quad 
    F^{-k}\bbR(k)_\bbC\coloneqq\bbR(k)_\bbC,\quad F^{-k+1}\bbR(k)_\bbC\coloneqq 0.\]
    This is a pure $\bbR$-Hodge structure of weight $-2k$. 
\end{example}

\begin{definition}
	Let $X$ be a complex manifold. 
	A \textit{variation of mixed $\bbR$-Hodge structures} on $X$ is 
	a triple $\bbV = (\bbV, W_\bullet, F^\bullet)$ consisting of an $\bbR$-local system $\bbV$ on $X$,
	a finite ascending filtration $W_\bullet$ of $\bbV$ by $\bbR$-local subsystems, and 
	a finite descending filtration $F^\bullet$ on $\cV \coloneqq \bbV\otimes\sO_X$ 
	by locally free $\sO_X$-submodules, satisfying the following conditions. 
	\begin{enumerate}
	    \item The Griffiths transversality
    	\[
	    	\nabla(F^p \cV) \subset F^{p-1}\cV\otimes\Omega^1_{X} 
	    \]
	    holds, where $\nabla$ is the integrable connection on $\cV$ defined by $\nabla\coloneqq 1\otimes d$. 
	    \item For any point $\xi\in X$ giving an inclusion $i_\xi: \{\xi\} \hookrightarrow X$, 
	    the triple $i^*_\xi \bbV \coloneqq( i^*_\xi \bbV, W_\bullet, F^\bullet)$ is a mixed 
	    $\bbR$-Hodge structure. Here $W_\bullet$ is the filtration on $i^*_\xi\bbV$
	    given by $W_n (i^*_\xi \bbV) \coloneqq i^*_\xi (W_n \bbV)$ for any $n\in\bbZ$, and 
	    $F^\bullet$ is the filtration on $i^*_\xi\bbV \otimes \bbC\cong i_\xi^*\cV$ 
	    given by $F^p(i^*_\xi \bbV \otimes\bbC) \coloneqq i^*_\xi (F^p \cV)$ for any $p\in\bbZ$.
	\end{enumerate}
\end{definition}

Later, we will need the notion of variation of mixed $\bbR$-Hodge structures 
on a (not necessarily finite) disjoint union of smooth algebraic varieties over $\bbC$, 
which is obviously defined. 

Any mixed $\bbR$-Hodge structure $V$ defines a variation of mixed $\bbR$-Hodge structures on $X$, 
by taking $\bbV$ to be the constant $\bbR$-local system on $X$ defined by $V$ with filtrations $W_\b$ 
and $F^\b$ induced by that on $V$. 
We call such $\bbV$ a \emph{constant} variation of mixed $\bbR$-Hodge structures. 

\begin{definition}
	Let $\bbV = (\bbV, W_\bullet, F^\bullet)$ be a variation of mixed $\bbR$-Hodge structures on $X$.  
	We say that $\bbV$ is \textit{unipotent}, if $\Gr^W_n \bbV \coloneqq (\Gr^W_n \bbV, F^\bullet)$ is 
	a constant variation of pure $\bbR$-Hodge structure of weight $n$,  
	where the filtration $F^\bullet$ on $\Gr^W_n \cV = \Gr^W_n \bbV \otimes \sO_X$ is 
	induced by $F^\bullet$ on $\cV=\bbV \otimes \sO_X$.
\end{definition}

We next define the admissibility of unipotent variation of mixed $\bbR$-Hodge structures. 
Let $X$ be a smooth algebraic variety over $\bbC$ and 
$X \hookrightarrow \ol X$ be a smooth compactification of $X$ such that 
the complement $D\coloneqq\ol X\setminus X$ is a normal crossing divisor. 
Let $\bbV = (\bbV, W_\bullet, F^\bullet)$ be a variation of mixed $\bbR$-Hodge structures on $X$. 
We let $\cV \coloneqq \bbV \otimes \sO_X$, with integrable connection $\nabla\coloneqq1\otimes d$. 
By \cite{Del70}*{Proposition 5.2}, there exists 
a canonical extension $\cV(D)$ of $\cV$, which is a certain locally free $\sO_{\ol X}$-module of finite type 
with integrable logarithmic connection along $D$ (denoted again by $\nabla$). 
Since $W_n\cV \coloneqq W_n\bbV\otimes \sO_X$ for $n\in\bbZ$ is an $\sO_X$-submodule of $\cV$ 
compatible with the connection, 
the canonical extension of $W_\bullet$ on $\cV$ defines a finite ascending 
filtration $W_\bullet$ on $\cV(D)$. 

\begin{definition}
	We say that the filtration $F^\bullet$ on $\cV$ \emph{extends} to $\cV(D)$ when 
	there exists a finite descending filtration $\wt{F}^\bullet$ on $\cV(D)$ 
	consisting of locally free $\sO_{\ol X}$-modules satisfying the following conditions: 
	\begin{enumerate}
	    \item The restriction of $\wt{F}^\bullet$ to $X$ coincides with $F^\bullet$.
	    \item The Griffiths transversality with respect to $\nabla$ holds. 
	    \item The graded quotients $\Gr_{\wt{F}}^p\Gr^W_n\cV(D)$ of the filtration on $\Gr^W_n\cV(D)$ 
	    induced by $\wt{F}^\bullet$ are locally free. 
	\end{enumerate}
	By abuse of notation, we often write $F^\bullet$ for such an extension $\wt{F}^\bullet$. 
\end{definition}

Kashiwara \cite{Kas86}*{\S1.9} introduced the notion of admissibility of a variation of mixed $\bbR$-Hodge structures. 
In the unipotent case, this is equivalent to the notion of goodness defined by Hain--Zucker, 
which may be described as follows (see \cite{HZ87}*{(1.5)} and the remark after it). 

\begin{proposition}\label{prop: admissible}
	A unipotent variation of mixed $\bbR$-Hodge structures $\bbV$ on $X$ is admissible 
	if and only if it satisfies the following conditions.
	\begin{enumerate}
		\item Let $\cV(D)$ be the canonical extension of $\cV$ to $\ol X$.  
		The filtration $F^\bullet$ of $\cV$ extends to a filtration $F^\bullet$ of $\cV(D)$
		compatible with $W_\bullet$. 
		\item For any component $D_i$ of $D$, the residue $\Res_{D_i}$ 
		of $\nabla$ along $D_i$ defined in \cite{Del70}*{(3.8.3)} satisfies 
		$\Res_{D_i}\bigl(W_n\cV(D)|_{D_i}\bigr) 
		\subset W_{n-2} \cV(D)|_{D_i}$ for any $n\in\bbZ$. 
	\end{enumerate}
\end{proposition}

The cohomology with coefficients in $\bbV$ is known to have a mixed $\bbR$-Hodge structure \cite{HZ87}*{(8.6)}. 
More precisely, we need an object of the derived category of $\MHS_\bbR$ whose cohomology gives $H^m(X,\bbV)$. 
Such an object can be systematically constructed by using Saito's theory of mixed Hodge modules as follows. 

\begin{definition}\label{def: RGamma_Hdg}
Let $\bbV$ be an admissible unipotent variation of mixed $\bbR$-Hodge structures on $X$. 
Regarding $\bbV$ as a mixed $\bbR$-Hodge module, we set 
\[R\Gamma_{\mathrm{Hdg}}(X,\bbV)\coloneqq \pi_*\bbV\in D^b(\MHS_\bbR), \]
where $\pi$ denotes the projection of $X$ onto a point (see \cite{Sa90}*{Theorems 0.2 and 0.1}). 
Note that, if we forget the mixed $\bbR$-Hodge structure, 
it coincides (in the derived category) with the complex $R\Gamma(X,\bbV)$ 
which gives the usual sheaf cohomology. 
\end{definition}

%
%
\subsection{Equivariant Sheaves and Their Cohomology}\label{subsec: equivariant sheaves}
%
%

We next introduce equivariant sheaves and their cohomology. 
Let $G$ be a group with identity $e$, and $X$ a topological space with a right $G$-action. 
We denote by  $\pair{u}\colon X \rightarrow X$ the action of $u\in G$, so that 
$\pair{uv} = \pair{v}\circ\pair{u}$ for any $u,v\in G$ holds. 

\begin{definition}\label{def: equivariant sheaf}
A \textit{$G$-equivariant sheaf} $\sF$ on $X$ is an abelian sheaf on $X$ equipped with 
an isomorphism $\iota_u\colon\pair{u}^{-1}\sF\isomto\sF$ for each $u\in G$, 
such that $\iota_e=\id$ and $\iota_u\circ\pair{u}^{-1}\iota_v=\iota_{uv}$. 
\end{definition}

The \textit{equivariant global section} of a $G$-equivariant sheaf $\sF$ is defined as 
\[
	\Gamma(X/G,\sF)\coloneqq\Hom_{\bbZ[G]}(\bbZ,\Gamma(X,\sF))=\Gamma(X,\sF)^G.
\]

\begin{definition}\label{def: equivariant}
	For any integer $m$, we define the equivariant cohomology $H^m(X/G,\sF)$ 
	of a $G$-equivariant sheaf $\sF$ to be the group obtained by applying to $\sF$ the 
	$m$-th right derived functor of $\Gamma(X/G,-)$.
\end{definition}

Note that there is a spectral sequence 
\[E_2^{p,q}=H^p(G,H^q(X,\sF))\Longrightarrow H^{p+q}(X/G,\sF). \]

Let $\pi\colon G\rightarrow H$ be a homomorphism between groups.
For topological spaces $X$ and $Y$ with actions of $G$ and $H$, respectively, 
we say that a continuous map $f\colon X\rightarrow Y$ is \textit{equivariant}, 
if the actions of $G$ and $H$ are compatible with $f$ through $\pi$. 
Then, for an $H$-equivariant sheaf $\sF$ on $Y$, $f^{-1}\sF$ naturally has a structure of 
a $G$-equivariant sheaf on $X$, and $f$ induces the pullback homomorphism on equivariant cohomology
\[
	f^{-1}\colon H^m(Y/H,\sF)\rightarrow H^m(X/G, f^{-1}\sF).
\]

When $X$ is a scheme, the notion of a \emph{$G$-equivariant $\sO_X$-module} is defined in a way similar 
to Definition \ref{def: equivariant sheaf}, with $\pair{u}^{-1}$ replaced by $\pair{u}^*$. 
Then the equivariant cohomology $H^m(X/G,\sF)$ is defined to be that of the underlying $G$-equivariant abelian sheaf. 
Moreover, for an equivariant morphism $f\colon X\to Y$, we have the pullback homomorphism 
\[
	f^*\colon H^m(Y/H,\sF)\rightarrow H^m(X/G, f^*\sF).  
\]

Now let $X$ be a smooth complex variety with a right $G$-action. 
Again, the notion of a \emph{$G$-equivariant variation of mixed $\bbR$-Hodge structures} is defined similarly, 
and the equivariant cohomology is defined to be that of the underlying $G$-equivariant local system. 

Let $\bbV$ be an admissible $G$-equivariant variation of mixed $\bbR$-Hodge structures on $X$, 
where the admissibility is defined by forgetting the $G$-action. 
As in the non-equivariant case, we want to equip the equivariant cohomology $H^m(X/G,\bbV)$ 
with a mixed $\bbR$-Hodge structure. For that purpose, we first consider the commutative diagram 
\[\xymatrix{
D^b(G\hyp\MHS_\bbR)\ar[d]\ar[r] & D^b(\MHS_\bbR)\ar[d]\\
D^b(G\hyp\Vec_\bbR)\ar[r] & D^b(\Vec_\bbR)
}\]
of forgetful functors, where $\Vec_\bbR$ denotes the category of $\bbR$-vector spaces and 
$G\hyp\MHS_\bbR$ (resp.\ $G\hyp\Vec_\bbR$) denotes the category of mixed $\bbR$-Hodge structures 
(resp.\ $\bbR$-vector spaces) with $G$-action. 
Forgetting $G$-equivariant structure, we have objects $R\Gamma_{\mathrm{Hdg}}(X,\bbV)$ of $D^b(\MHS_\bbR)$ and $R\Gamma(X,\bbV)$ of $D^b(\Vec_\bbR)$ by the construction of Definition \ref{def: RGamma_Hdg}.
Moreover $R\Gamma(X,\bbV)$ naturally lifts to an object of $D^b(G\hyp\Vec_\bbR)$, which is denoted by the same symbol 
(to see this, for instance, note that the Godement resolution inherits the $G$-action). 
Given these data, the following condition makes sense: 

\begin{condition}\label{cond: RGamma_Hdg}
There exists a lift of $R\Gamma_{\mathrm{Hdg}}(X,\bbV)$ in $D^b(G\hyp\MHS_\bbR)$, 
which is denoted by the same symbol, and it coincides with $R\Gamma(X,\bbV)$ in $D^b(G\hyp\Vec_\bbR)$ 
if we forget mixed $\bbR$-Hodge structure. 
\end{condition}


Let $R\Gamma(X/G,\bbV)$ denote the object of $D^b(\Vec_\bbR)$ which yields the $G$-equivariant sheaf cohomology. 
Assuming Condition \ref{cond: RGamma_Hdg}, we can construct its natural lift $R\Gamma_\Hdg(X/G,\bbV)$ 
with mixed $\bbR$-Hodge structure, as follows. 

\begin{lemma}\label{lem: G-invariant}
There is a commutative diagram
\[\xymatrix{
D^+(G\hyp\Pro(\MHS_\bbR))\ar[d]\ar[rr]^-{R\Gamma(G,-)} && D^+(\Pro(\MHS_\bbR))\ar[d]\\
D^+(G\hyp\Pro(\Vec_\bbR))\ar[rr]^-{R\Gamma(G,-)} && D^+(\Pro(\Vec_\bbR)), 
}\]
where the vertical arrows are the forgetful functors and $\Gamma(G,-)$ denotes the functor taking the $G$-invariant part. 
Moreover, if $G$ is a free abelian group of finite rank, there is a commutative diagram 
\begin{equation}\label{eq: forget MHS}
\xymatrix{
D^b(G\hyp\MHS_\bbR)\ar[d]\ar[rr]^-{R\Gamma(G,-)} && D^b(\MHS_\bbR)\ar[d]\\
D^b(G\hyp\Vec_\bbR)\ar[rr]^-{R\Gamma(G,-)} && D^b(\Vec_\bbR). 
}
\end{equation}
\end{lemma}
\begin{proof}
This is a special case of \cite{BurgosWildeshaus}*{Proposition 3.17}. 
\end{proof}

\begin{definition}
Under Condition \ref{cond: RGamma_Hdg}, we define 
\[R\Gamma_{\mathrm{Hdg}}(X/G,\bbV):=R\Gamma(G,R\Gamma_{\mathrm{Hdg}}(X,\bbV))\in D^+(\Pro(\MHS_\bbR)).\]
If $G$ is a free abelian group of finite rank, we have $R\Gamma_{\mathrm{Hdg}}(X/G,\bbV)\in D^b(\MHS_\bbR)$. 
\end{definition}

By Lemma \ref{lem: G-invariant} and the equality $R\Gamma(X/G,\bbV)=R\Gamma(G,R\Gamma(X,\bbV))$, 
we see that $R\Gamma_{\mathrm{Hdg}}(X/G,\bbV)$ is a lift of $R\Gamma(X/G,\bbV)$. 
Thus we can equip the equivariant cohomology $H^m(X/G,\bbV)$ with a (pro-)mixed $\bbR$-Hodge structure. 

In what follows, we assume that $G$ is a free abelian group of finite rank. 
Then we have the following: 

\begin{lemma}
There is a commutative diagram of functors 
\begin{equation}\label{eq: RGamma_D}
\xymatrix{
D^b(G\hyp\MHS_\bbR)\ar[d]_-{R\Hom_{\MHS_\bbR}(\bbR(0),-)}\ar[rr]^-{R\Gamma(G,-)} 
&& D^b(\MHS_\bbR)\ar[d]^-{R\Hom_{\MHS_\bbR}(\bbR(0),-)}\\
D^b(G\hyp\Vec_\bbR)\ar[rr]^-{R\Gamma(G,-)} && D^b(\Vec_\bbR). 
}
\end{equation}
\end{lemma}
\begin{proof}
Note first that there is a diagram, commutative or not, of derived functors 
\begin{equation}\label{eq: RGamma_D Ind}
\xymatrix{
D^+(\Ind(G\hyp\MHS_\bbR))\ar[d]_-{R\Hom_{\MHS_\bbR}(\bbR(0),-)}\ar[rr]^-{R\Gamma(G,-)} 
&& D^+(\Ind(\MHS_\bbR))\ar[d]^-{R\Hom_{\MHS_\bbR}(\bbR(0),-)}\\
D^+(\Ind(G\hyp\Vec_\bbR))\ar[rr]^-{R\Gamma(G,-)} && D^+(\Ind(\Vec_\bbR)). 
}
\end{equation}
The diagram \eqref{eq: RGamma_D} is obtained by restriction to full subcategories, cf.~\cite{Hub93}*{Proposition 2.2}. 

The commutativity of \eqref{eq: RGamma_D Ind} follows from the fact that both compositions are equal to 
the derived functor of 
\[G\hyp\MHS_\bbR\longrightarrow \Vec_\bbR;\ 
V\longmapsto \Hom_{\MHS_\bbR}(\bbR(0),V)^G=\Hom_{\MHS_\bbR}(\bbR(0),V^G). \]
To prove these equalities, it suffices to note that the functors $\Gamma(G,-)$ and $\Hom_{\MHS_\bbR}(\bbR(0),-)$ 
preserve the injectives in ind-categories since they admit exact left adjoint functors. 
\end{proof}

\begin{definition}\label{def: EDBC}
Under Condition \ref{cond: RGamma_Hdg} (and the assumption that $G$ is free abelian of finite rank), we define 
\begin{align*}
R\Gamma_{\sD}(X,\bbV)&:=R\Hom_{\MHS_\bbR}(\bbR(0),R\Gamma_{\mathrm{Hdg}}(X,\bbV))\in D^b(G\hyp\Vec_\bbR),\\
R\Gamma_{\sD}(X/G,\bbV)&:=R\Hom_{\MHS_\bbR}(\bbR(0),R\Gamma_{\mathrm{Hdg}}(X/G,\bbV))\\
&=R\Gamma(G,R\Gamma_{\sD}(X,\bbV))\in D^b(\Vec_\bbR).  
\end{align*}
The cohomology of $R\Gamma_{\sD}(X/G,\bbV)$ is called the \emph{equivariant Deligne--Beilinson cohomology} 
with coefficients in $\bbV$ and denoted by $H^m_{\sD}(X/G,\bbV)$. 
\end{definition}

By construction, we have spectral sequences
\begin{align*}
	&E_2^{p,q}=H^p(G,H^q_{\Hdg}(X,\bbV))\Rightarrow H^{p+q}_{\Hdg}(X/G,\bbV),\\
	&E_2^{p,q}=\Ext^p_{\MHS_{\bbR}}(\bbR(0),H^q_{\Hdg}(X,\bbV))\Rightarrow H^{p+q}_{\sD}(X,\bbV),\\
	&E_2^{p,q}=\Ext^p_{\MHS_{\bbR}}(\bbR(0),H^q_{\Hdg}(X/G,\bbV))\Rightarrow H^{p+q}_{\sD}(X/G,\bbV). 
\end{align*}

From the next section, we mainly focus on specific triples $(X,G,\bbV)=(U^\fra,\Delta,\bbLog_\fra^N)$, 
for which we are currently unable to prove Condition \ref{cond: RGamma_Hdg}. 
However, it is possible to construct a complex 
which will play the role of the complex $R\Gamma_{\mathrm{Hdg}}(U^\fra/\Delta,\bbLog_\fra^N)$, 
and this is sufficient for our purpose. 
We will give such a construction in the Appendix.

\section{The Logarithm Sheaf}\label{sec: log sheaf}
Let $F$ be a totally real field of degree $g$ and $\frI$ the set of nonzero fractional ideals of $F$. 
In this section, we first introduce a projective system $\bbLog_\fra=(\bbLog_\fra^N)_{N\ge 0}$ of 
variations of mixed $\bbR$-Hodge structures on the character torus $\bbT^\fra$ for each $\fra\in\frI$. 
The family $(\bbLog_\fra)_{\fra\in\frI}$ gives a projective system of variations of 
mixed $\bbR$-Hodge structures $\bbLog$ on the union $\bbT=\coprod_{\fra\in\frI}\bbT^\fra$ of tori, 
which admits a natural $F^\times_+$-equivariant structure. 

\subsection{The Algebraic Torus $\bbT^\fra$ and Its Cohomology}\label{subsec: H_fra}
 
For $\fra\in\frI$, let $\bbT^\fra\coloneqq\Hom(\fra,\bbG_m)$ be the algebraic torus of characters of $\fra$. 
Then we have $\T^{\fra}=\Spec\bbQ[t^\alpha\mid\alpha\in\fra]$, where $t^\alpha$ for $\alpha\in\fra$ are parameters
satisfying $t^\alpha t^{\alpha'}=t^{\alpha+\alpha'}$ for any $\alpha,\alpha'\in\fra$.
If we fix a basis $\alpha_1,\ldots,\alpha_g\in\fra$ of the $\bbZ$-module $\fra$, 
then we have $\bbT^\fra\cong\Spec\bbQ[t^{\pm\alpha_1},\ldots,t^{\pm\alpha_g}]$, 
hence $\bbT^\fra$ is an affine group scheme of dimension $g$ over $\bbQ$.

Let $H_\fra$ denote the dual of the mixed $\bbR$-Hodge structure $H^1(\bbT^\fra(\bbC),\bbR)$. 
By definition, $H_\fra=H^1(\T^\fra(\bbC),\bbR)^\vee$ as an $\bbR$-vector space, and 
we have a canonical isomorphism $H_{\fra,\bbC}\coloneqq H_\fra\otimes\bbC\cong H^1_\dR(\bbT^\fra/\bbC)^\vee$. 
If we take a $\bbZ$-basis $\bsalpha=(\alpha_1,\ldots,\alpha_g)$ of $\fra$, then
we have an isomorphism $\bbT^\fra\cong \mathbb{G}_m^g$, which induces an isomorphism
\[H^1_{\dR}(\T^\fra/\bbC)\cong \bigoplus_{i=1}^g\bbC\,\dlog t^{\alpha_i}.\]
We denote by $u_{\alpha_1},\ldots,u_{\alpha_g}\in H_{\fra,\bbC}$ 
the dual basis of $\dlog t^{\alpha_1},\ldots,\dlog t^{\alpha_g}$.
Then $\gamma_{\alpha_i}\coloneqq 2\pi i\,u_{\alpha_i}$ form a basis of $H_\fra$. 
The weight and Hodge filtrations on $H_\fra$ are given by 
\[W_{-3}H_\fra=0,\quad W_{-2}H_\fra=H_\fra,\quad F^{-1}H_{\fra,\bbC}=H_{\fra,\bbC},\quad F^0H_{\fra,\bbC}=0. \]

Next, we give another basis of $H_\fra$ independent of the choice of $\bbZ$-basis $\bsalpha$. 
Recall that we have a canonical isomorphism 
\[
	F\otimes\bbR\cong\bbR^I\coloneqq \prod_{\tau\in I}\bbR,  \qquad 
	\alpha\otimes 1 \mapsto (\alpha^\tau),
\]
where $I\coloneqq\Hom(F,\bbR)$ and we let $\alpha^\tau\coloneqq\tau(\alpha)$ for $\tau\in I$. 
Let us imagine that, for any $\fra\in\frI$, 
there were functions $t_\tau=t_{\tau,\fra}$ on $\T^\fra$ for $\tau\in I$ 
such that $t^\alpha$ for all $\alpha\in \fra$ were represented as 
\[t^\alpha=\prod_{\tau\in I} (t_\tau)^{\alpha^\tau}. \]
Then we would have identities of differential forms 
$\dlog t^\alpha=\sum_{\tau\in I}\alpha^\tau \dlog t_\tau$. 
Although such functions $t_\tau$ do not actually exist globally, 
we may \emph{define} the differential forms $\dlog t_\tau=\dlog t_{\tau,\fra}$ 
as the solutions of the linear equations 
\begin{equation}\label{eq: base change}
    \dlog t^{\alpha_i}=\sum_{\tau\in I}\alpha_i^\tau \dlog t_\tau 
\end{equation}
for some $\bbZ$-basis $(\alpha_1,\ldots,\alpha_g)$ of $\fra$. 
These solutions are independent of the choice of the $\bbZ$-basis, 
and form a basis of $H^1_\dR(\bbT^\fra/\bbC)=H_{\fra,\bbC}^\vee$. 
We write $(u_\tau)_{\tau\in I}$ for the basis of $H_{\fra,\bbC}$ dual to $(\dlog t_\tau)_{\tau\in I}$. 
Moreover, if we set $\gamma_\tau\coloneqq 2\pi i\,u_\tau$,
then $(\gamma_\tau)_{\tau\in I}$ is a basis of $H_\fra$. 

\subsection{Construction of $\cLog_\fra^N$}\label{ssec:2.3.2}
Fix $\fra\in\frI$ and an integer $N\ge 0$. 
Let $\cH_\fra\coloneqq \sO_{\T^\fra}\otimes H_{\fra,\bbC}$ be the free $\sO_{\T^\fra}$-module 
associated with the $\bbC$-vector space $H_{\fra,\bbC}$. 
Then we set 
\begin{equation}
    \cLog^N_\fra\coloneqq \prod_{k=0}^N \Sym^k_{\sO_{\T^\fra}} \cH_\fra. 
\end{equation}
We also define the filtrations $W_\b$ and $F^\b$ on $\cLog^N_\fra$ by 
\begin{equation}\label{eq: WF}
	W_{-2m}\cLog^N_\fra \coloneqq W_{-2m+1}\cLog^N_\fra \coloneqq
	\prod_{m\leq k\leq N}  \Sym^k_{\sO_{\T^\fra}} \cH_\fra, \quad 
	F^{-p} \cLog^N_\fra\coloneqq\prod_{0\le k\le p} \Sym^k_{\sO_{\T^\fra}} \cH_\fra.
\end{equation}

Next we define an integrable connection. 
Let $\eta_\fra$ be the element of $H_\fra\otimes H_\fra^\vee$ corresponding to 
$\id\colon H_\fra\to H_\fra$ under the canonical isomorphism $H_\fra\otimes H_\fra^\vee\cong\End(H_\fra)$. 
We denote the corresponding element of $H_{\fra,\bbC}\otimes H_{\fra,\bbC}^\vee$ 
by the same symbol $\eta_\fra$
Moreover, by the lifting $H_{\fra,\bbC}^\vee=H^1_\dR(\T^\fra/\bbC)\to\Omega^1(\T^\fra)$ given by 
$[\dlog t^\alpha]\mapsto \dlog t^\alpha$, 
we lift $\eta_\fra$ to a global section of $\cH_\fra\otimes\Omega^1_{\T^\fra}$, 
which is also denoted by the same symbol $\eta_\fra$. 
Then the connection 
\begin{equation}\label{eq: connection0}
\nabla\colon\cLog^N_\fra\to\cLog^N_\fra\otimes\Omega^1_{\T^\fra}
\end{equation}
is defined on $\Sym^k H_{\fra,\bbC}$ ($\subset\Sym^k_{\sO_{\T^\fra}} \cH_\fra$) to be zero for $k=N$, 
and to be the composite 
\begin{equation}
    \Sym^k H_{\fra,\bbC} 
    \to \Sym^k H_{\fra,\bbC} \otimes H_{\fra,\bbC}\otimes\Omega^1_{\T^\fra}
    \to \Sym^{k+1} H_{\fra,\bbC} \otimes\Omega^1_{\T^\fra}
\end{equation}
for $0\le k\le N-1$, where the first map is the tensor product with $\eta_\fra$ 
and the second is induced by the natural product map 
$\Sym^k H_\fra \otimes H_\fra\to \Sym^{k+1} H_\fra$.

Now we describe the connection $\nabla$ on $\cLog^N_\fra$ explicitly in terms of 
a basis $\bsalpha=(\al_1,\ldots,\al_g)$ of $\fra$. 
Recall that $u_{\alpha_1},\ldots,u_{\alpha_g}\in H_{\fra,\bbC}$ denotes the dual basis of 
$\dlog t^{\alpha_1},\ldots,\dlog t^{\alpha_g}\in H_{\fra,\bbC}^\vee$. 
Thus we have 
\[
	\eta_\fra= \sum_{i=1}^g u_{\al_i}\otimes\dlog t^{\al_i}.
\]
For $\bsk = (k_i) \in \bbN^g$, we put 
$u^\bsk_\bsal\coloneqq u_{\alpha_1}^{k_1}\cdots u_{\alpha_g}^{k_g}\in \Sym^{\absv{\bsk}}H_\fra$,  
where $\absv{\bsk}\coloneqq \sum_{i=1}^g k_i$. 
Then the $u^\bsk_\bsal$ for $|\bsk|\leq N$ form a free basis of $\cLog^N_\fra$:
\begin{equation}\label{eq: cLog}
	\cLog^N_\fra=\prod_{|\bsk|\leq N} \sO_{\T^\fra}\, u^\bsk_\bsal.
\end{equation}
The connection $\nabla$ on $\cLog^N_\fra$ is given by
\begin{equation}\label{eq: connection}
	\nabla(f\,u_\bsal^\bsk)=u_\bsal^\bsk\otimes df
	+\sum_{i=1}^g f\,u_\bsal^{\bsk+1_i}\otimes \dlog t^{\alpha_i}
\end{equation}
for any section $f$ of $\sO_{\T^\fra}$, where $1_i$ denotes the element in $\bbN^g$ 
whose $i$-th component is $1$ and the other components are $0$,
and $u_\bsal^\bsk$ is regarded as zero in $\cLog^N_\fra$ if $|\bsk|>N$.

\subsection{Construction of $\bbLog_\fra^N$}\label{ssec:2.3.3}
Note that the integrable connection $\nabla$ on $\cLog_\fra^N$ gives a $\bbC$-local system 
$\bbLog_{\fra,\bbC}^N$ of horizontal sections. Let us define its $\bbR$-structure $\bbLog_\fra^N$ as follows. 

As before, let $\bsalpha=(\al_1,\ldots,\al_g)$ be a $\bbZ$-basis of $\fra$, and 
$(d\log t^{\alpha_i})$ and $(u_{\alpha_i})$ the associated basis of $H_{\fra,\bbC}^\vee$ and $H_{\fra,\bbC}$, 
respectively. 
We fix local branches of $\log t^{\alpha_i}$. 
Then we define the local sections $\gamma_\bsalpha^\bsk$ of $\cLog_\fra^N$ for $\bsk=(k_i)\in\bbN^g$ by 
\begin{equation}\label{eq: transformation}
	 \gamma_\bsalpha^\bsk \coloneqq (2\pi i)^{\absv{\bsk}}
	 \exp\biggl(\sum_{i=1}^g (-\log t^{\alpha_i}) u_{\alpha_i}\biggr)\cdot u_\bsalpha^\bsk
	 =(2\pi i)^{\absv{\bsk}}\sum_{\bsn=(n_i)\in\bbN^g}\prod_{i=1}^g
	 \frac{(-\log t^{\alpha_i})^{n_i}}{n_i!} u_\bsalpha^{\bsk+\bsn}. 
\end{equation}
Note that $\gamma_\bsalpha^\bsk$ is \emph{not} equal to the product 
$\gamma_{\alpha_1}^{k_1}\cdots\gamma_{\alpha_g}^{k_g}$ in general, 
where $(\gamma_{\alpha_i})$ is the basis of $H_\fra$ defined by $\gamma_{\alpha_i}\coloneqq 2\pi i\,u_{\alpha_i}$. 

\begin{lemma}\label{lem: gamma_bsal}
The local section $\gamma_\bsalpha^\bsk$ is horizontal with respect to $\nabla$, 
and the family $(\gamma_\bsalpha^\bsk)_{\absv{\bsk}\le N}$ is a local basis of 
the $\bbC$-local system $\bbLog_{\fra,\bbC}^N$ of horizontal sections of $\cLog_\fra^N$. 
Moreover, their $\bbR$-span 
\[
	\bbLog^N_\fra\coloneqq\prod_{\bsk\in\bbN^g, \absv{\bsk}\leq N}\bbR\,\gamma_\bsalpha^\bsk 
\] 
gives a well-defined global $\bbR$-structure of $\bbLog_{\fra,\bbC}^N$. 
\end{lemma}
\begin{proof}
We have 
\begin{multline*}
	\nabla(\gamma_\bsal^\bsk)=(2\pi i)^{\absv{\bsk}}\exp\biggl(\sum_{i=1}^g(-\log t^{\alpha_i})u_{\alpha_i}\biggr) 
	\cdot \sum_{j=1}^g  u^{\bsk+1_j}_\bsal\otimes \dlog t^{\alpha_j}\\
	-(2\pi i)^{\absv{\bsk}}\sum_{j=1}^g\exp\biggl(\sum_{i=1}^g(-\log t^{\alpha_i})u_{\alpha_i}\biggr)
	u_{\alpha_j}\cdot u^\bsk_\bsal\otimes \dlog t^{\alpha_j}=0,
\end{multline*}
hence $\gamma_\bsalpha^\bsk$ is horizontal. 
On the other hand, a $\bbC$-linear relation $\sum_{\bsk}c_\bsk \gamma_\bsalpha^\bsk=0$ implies 
\[\sum_{\bsk} (2\pi i)^{\absv{\bsk}} c_\bsk u_\bsalpha^\bsk 
=\exp\biggl(\sum_{i=1}^g(\log t^{\alpha_i})u_{\alpha_i}\biggr) \sum_{\bsk} c_\bsk \gamma_\bsalpha^\bsk=0, \]
hence all $c_\bsk$ must be zero. Thus the sections $\gamma_\bsalpha^\bsk$ are linearly independent over $\bbC$. 
Since the dimension of the space of horizontal sections is equal to the rank of the $\sO_{\bbT^\fra}$-module 
$\cLog_\fra^N$, i.e., the number of $\bsk\in\bbN^g$ such that $\absv{\bsk}\le N$, 
the sections $\gamma_\bsalpha^\bsk$ form a basis of the space of horizontal sections. 

To check that the $\bbR$-span is independent of the branches of $\log t^{\alpha_i}$, 
we take any $j\in\{1,\ldots,g\}$ and replace $\log t^{\alpha_j}$ by $\log t^{\alpha_j}-2\pi i$. 
If $\gamma_\bsalpha^{\prime\bsk}$ denotes the corresponding horizontal section, we have 
\[
\gamma_\bsalpha^{\prime\bsk}=\exp(2\pi i\,u_{\alpha_j})\cdot \gamma_\bsalpha^\bsk
=\sum_{n=0}^\infty\frac{(2\pi i)^n}{n!}u_{\alpha_j}^n\cdot \gamma_\bsalpha^\bsk
=\sum_{n=0}^\infty \frac{1}{n!}\gamma_\bsalpha^{\bsk+n\cdot 1_j}, 
\]
which shows that the $\bbR$-span of $\gamma_\bsalpha^{\prime\bsk}$ is contained 
in the $\bbR$-span of $\gamma_\bsalpha^{\bsk}$. 
Since the $\bbR$-spans have the same dimension, they are equal. 

Next, let $\bsbeta=(\beta_1,\ldots,\beta_g)$ be another $\bbZ$-basis of $\fra$ and 
$(c_{ij})$ the matrix such that $\beta_i=\sum_{j=1}^g c_{ij}\alpha_j$. 
Then, taking the local branches of $\log t^{\beta_i}$ and $\log t^{\alpha_j}$ so that 
$\log t^{\beta_i}=\sum_{j=1}^g c_{ij}\log t^{\alpha_j}$, we have 
\[\sum_{i=1}^g(-\log t^{\beta_i})u_{\beta_i}=\sum_{i=1}^g(-\log t^{\alpha_i})u_{\alpha_i}, \]
hence the operator $\exp\bigl(\sum_{i=1}^g(-\log t^{\alpha_i})u_{\alpha_i}\bigr)$ does not depend on the basis. 
Since $(2\pi i)^{\absv{\bsk}}u_\bsbeta^\bsk$ and $(2\pi i)^{\absv{\bsk}}u_\bsalpha^\bsk$ 
span the same $\bbR$-vector space $\prod_{\absv{\bsk}\le N}\Sym^{\absv{\bsk}}H_\fra$, 
we see that $\gamma_\bsbeta^\bsk$ and $\gamma_\bsalpha^\bsk$ also span the same $\bbR$-vector space. 
\end{proof}

We define the weight filtration on $\bbLog^N_\fra$ by 
\[
	W_{-2m}\bbLog^N_\fra\coloneqq W_{-2m+1}\bbLog^N_\fra\coloneqq\prod_{m\leq\absv\bsk\leq N} \bbR\,\gamma_\bsalpha^\bsk,
\]
which is compatible with the weight filtration $W_\b$ on $\cLog^N_\fra$ through the natural isomorphism 
\[
	\bbLog^N_\fra\otimes\sO_{\T^\fra} \cong \cLog^N_\fra.
\]
Together with the filtration $F^\b$ on $\cLog_\fra^N$, the triple $\bbLog_\fra^N=(\bbLog_\fra^N,W_\b,F^\b)$ 
forms an admissible unipotent variation of mixed $\bbR$-Hodge structures on $\bbT^\fra$, as is shown below. 

By mapping the sections $\gamma^\bsk_\bsal$ of $\bbLog^N_\fra$ to $\prod_{i=1}^g\gamma_{\alpha_i}^{k_i}$, 
we obtain an isomorphism 
\begin{equation}\label{eq: Gr^W}
	\Gr^W_{-2k} \bbLog^N_\fra \isomto \Sym^k H_\fra
\end{equation}
of $\bbR$-local systems on $\T^\fra$, compatible with the filtrations $F^\b$,
where we view $\Sym^k H_\fra$ as a constant variation of pure $\bbR$-Hodge structures on $\T^\fra$.
Thus the triple $\bbLog^N_\fra=(\bbLog^N_\fra, W_\b, F^\b)$ is 
a \textit{unipotent} variation of mixed $\bbR$-Hodge structures.  

\begin{proposition}\label{prop: bbLog_fra^N}
$\bbLog^N_\fra$ is an admissible unipotent variation of mixed $\bbR$-Hodge structures on $\bbT^\fra$. 
\end{proposition}
\begin{proof}
We will use Proposition \ref{prop: admissible} to show the admissibility. 
Fix a $\bbZ$-basis $\bsal=(\al_1,\ldots,\al_g)$ of $\fra$, which induces an isomorphism $\T^\fra\cong(\bbG_m)^g$. 
Then we have an inclusion $\T^\fra\hookrightarrow X_\bsal\coloneqq(\bbP^1)^g$ 
with complement $D_\bsal\coloneqq X_\bsal \setminus \T^\fra$.
We denote by $\Omega^\b_{X_\bsal}(D_\bsal)$ the complex of differentials on $X_\bsal$ 
with logarithmic singularities along $D_\bsal$.
Note that $\cLog^N_\fra$ extends naturally to a coherent $\sO_{X_\bsal}$-module
\[
	\cLog_\fra^N(D_\bsal)\coloneqq\prod_{\absv\bsk\leq N} \sO_{X_\bsal} \gamma^\bsk_\bsal,
\]
and \eqref{eq: connection} defines a connection 
\[
	\nabla\colon\cLog^N_\fra(D_\bsal) \rightarrow \cLog_\fra^N(D_\bsal) \otimes \Omega^1_{X_\bsal}(D_\bsal)
\] 
with logarithmic singularities along $D_\bsal$. 
Then $\cLog^N_\fra(D_\bsal)$ is the canonical extension of $\cLog^N_\fra$ to $X_\bsal$, 
and the conditions of Proposition \ref{prop: admissible} can be verified from the construction. 
\end{proof}

\subsection{$\Ft$-Equivariant Structure}\label{ssec:2.3.4}
For any $\fra\in\frI$ and $x\in F^\times_+$, the multiplication by $x$ gives an isomorphism $\fra\isomto x\fra$ 
and induces an isomorphism $\pair{x}\colon\bbT^{x\fra}\isomto\bbT^\fra$. 
These maps define an action of $F^\times_+$ on the disjoint union $\bbT\coloneqq\coprod_{\fra\in\frI}\bbT^\fra$. 

Since $\bbT$ is a disjoint union of varieties, 
the notion of ($F^\times_+$-equivariant) variation of mixed $\bbR$-Hodge structures on $\bbT$ can be naturally defined. 
In fact, a variation of mixed $\bbR$-Hodge structures $\bbV$ on $\bbT$ corresponds to 
a family $(\bbV_\fra)_{\fra\in\frI}$ of variations of mixed $\bbR$-Hodge structures on $\bbT^\fra$, 
and an $F^\times_+$-equivariant structure on $\bbV$ corresponds to a compatible family of isomorphisms 
$\iota_x\colon\pair{x}^*\bbV_\fra\isomto\bbV_{x\fra}$. 

The family $(H_\fra)_{\fra\in\frI}$, regarded as a family of constant variations of pure $\bbR$-Hodge structures 
on $\bbT^\fra$, defines a variation of pure $\bbR$-Hodge structures on $\bbT$, which is denoted by $\bbH$. 
For $\fra\in\frI$ and $x\in\Ft$, the isomorphism $\pair{x}\colon \T^{x\fra}\isomto \T^\fra$ induces 
an isomorphism $H_{x\fra}\isomto H_\fra$. We regard its inverse as the map 
$\iota_x\colon\pair{x}^*H_\fra\isomto H_{x\fra}$ of constant variations, 
which gives the $\Ft$-equivariant structure on $\bbH$. 
Explicitly, for a basis $\bsal=(\alpha_1,\ldots,\alpha_g)$ of $\fra$, 
we have $\pair{x}^*\dlog t^{\alpha_i}=\dlog t^{x\alpha_i}$ and hence $\iota_x(u_{\alpha_i})=u_{x\alpha_i}$. 

In fact, $\bbH$ is constant on the whole of $\bbT$. 
We set $\bbR(\bone)\coloneqq\bigoplus_{\tau\in I}\bbR(1_\tau)$, where each $\bbR(1_\tau)$ is 
a copy of the Tate object $\bbR(1)$ in $\MHS_{\bbR}$. 
Recall that $\bbR(1)=2\pi i\,\bbR$ as an $\bbR$-vector space (see Example \ref{ex: Tate object}). 
We denote by $(e_\tau)_\tau$ the basis of $\bbR(\bone)$ consisting of vectors $e_\tau$ 
whose $\tau$-component is $2\pi i\in\bbR(1_\tau)$ and others are zero. 
The associated constant variation of pure $\bbR$-Hodge structures on $\bbT$ is denoted 
by the same symbol $\bbR(\bone)$. 
Moreover, we equip $\bbR(\bone)$ with an $\Ft$-equivariant structure so that 
$x\in\Ft$ acts on $\bbR(1_\tau)$ as the multiplication by $(x^\tau)^{-1}$. 
Recall that the basis $(\gamma_\tau)$ of $H_\fra$ is defined to be 
$2\pi i$ times the basis $(u_\tau)$ of $H_{\fra,\bbC}$, the latter of which is 
dual to the basis $(\dlog t_\tau)$ of $H_{\fra,\bbC}^\vee$ determined by \eqref{eq: base change} 
(see \S\ref{subsec: H_fra}). 
Here we add the subscript $\fra$ as $\gamma_{\tau,\fra}$ to indicate the dependence on $\fra$. 
Using this basis, we define the $\bbR$-linear map 
\begin{equation}\label{eq: R(1) to H_fra}
    \bbR(\bone)\lra H_\fra;\ e_\tau\longmapsto \gamma_{\tau,\fra}. 
\end{equation}
This is an isomorphishm of pure $\bbR$-Hodge structures, and one can view it 
as an isomorphism of constant variations of pure $\bbR$-Hodge structures on $\T^\fra$. 
Moreover, by collecting these isomorphisms for all $\fra\in\frI$, 
we obtain an isomorphism $\bbR(\bone)\isomto\bbH$ of variations of pure $\bbR$-Hodge structures on $\bbT$. 

\begin{lemma}\label{lem: H}
The isomorphism $\bbR(\bone)\isomto\bbH$ constructed above is $\Ft$-equivariant. 
\end{lemma}
\begin{proof}
We first observe that 
\[\sum_{\tau\in I}(x\alpha)^\tau\dlog t_{\tau,x\fra}=\dlog t^{x\alpha}
=\pair{x}^*\dlog t^{\alpha}=\sum_{\tau\in I}\alpha^\tau\pair{x}^*\dlog t_{\tau,\fra} \]
holds for all $\alpha\in\fra$. 
From this equation, we obtain that $\pair{x}^*\dlog t_{\tau,\fra}=x^\tau\dlog t_{\tau,x\fra}$ 
and hence $\iota_x(\gamma_{\tau,\fra})=(x^\tau)^{-1}\gamma_{\tau,x\fra}$. 
This shows the desired equivariance. 
\end{proof}

Next we consider the associated $\sO_\T$-module $\cLog^N$ corresponding to 
the collection $(\cLog_\fra^N)_{\fra\in\frI}$ of $\sO_{\T^\fra}$-modules. 
We have 
\[\cLog^N=\prod_{k=0}^N\Sym^k_{\sO_\T}\cH, \]
where $\cH\coloneqq\bbH\otimes\sO_\T$ denotes the free $\sO_\T$-module corresponding to the local system $\bbH$. 
Then the $\Ft$-equivariant structure on $\bbH$ induces an $\Ft$-equivariant structure on $\cLog^N$. 
Explicitly, for a basis $\bsal=(\alpha_1,\ldots,\alpha_g)$ of $\fra$, the map 
\begin{equation}\label{eq: equiv}
\iota_x\colon \pair{x}^*\cLog^N_\fra \isomto \cLog^N_{x\fra} 
\end{equation}
is given by $\iota_x(u_\bsal^\bsk)=u_{x\bsal}^\bsk$. 
Then it follows from the definition \eqref{eq: transformation} that 
$\iota_x(\gamma_\bsal^\bsk)=\gamma_{x\bsal}^\bsk$. 
Therefore, $\iota_x$ induces an isomorphism $\pair{x}^*\bbLog_\fra^N\isomto\bbLog_{x\fra}^N$ of 
variations of mixed $\bbR$-Hodge structures. In other words, we have obtained an $\Ft$-equivariant structure 
on the variation of mixed $\bbR$-Hodge structures $\bbLog^N$ on $\bbT$. 

Note that \eqref{eq: Gr^W} and Lemma \ref{lem: H} gives the isomorphisms 
\[\Gr^W_{-2k}\bbLog^N\cong \Sym^k\bbH\cong \Sym^k\bbR(\bone) \]
of variations of pure $\bbR$-Hodge structures on $\bbT$, which shows that $\bbLog^N$ is unipotent. 
To describe it explicitly, we put 
$u_\fra^\bsk\coloneqq \prod_{\tau\in I} u_{\tau,\fra}^{k_\tau}\in\Sym^{\absv{\bsk}}H_{\fra,\bbC}$ 
for $\bsk=(k_\tau)\in\bbN^I$, and set 
\[
\gamma_\fra^\bsk\coloneqq
(2\pi i)^{\absv{\bsk}}\exp\biggl(\sum_{\tau\in I}(-\log t_{\tau,\fra})u_{\tau,\fra}\biggr)\cdot u_\fra^\bsk. 
\]
Here $\log t_{\tau,\fra}$ is a local section of $\sO_{\T^\fra}$ 
determined by the equation 
\[\log t^{\alpha_i}=\sum_{\tau\in I}\alpha_i^\tau\log t_{\tau,\fra}\] 
for some basis $\bsal=(\alpha_1,\ldots,\alpha_g)$ of $\fra$ and local branches of $\log t^{\alpha_i}$. 
Then one can verify that the sections $\gamma_\fra^\bsk$ form a local basis of $\bbLog_\fra^N$: 
\[\bbLog_\fra^N=\prod_{\absv{\bsk}\le N}\bbR\,\gamma_\fra^\bsk. \]

\begin{lemma}
The map $\Gr^W_{-2k}\bbLog^N\to\Sym^k\bbR(\bone)$ defined on $\T^\fra$ by 
\[\gamma_\fra^\bsk\longmapsto e^\bsk\coloneqq \prod_{\tau\in I}e_\tau^{k_\tau}
\qquad (\bsk\in\bbN^I,\ \absv{\bsk}=k)\]
is an isomorphism of $\Ft$-equivariant variations of pure $\bbR$-Hodge structures on $\T$. 
\end{lemma}
\begin{proof}
It is clear from the construction that this is an isomorphism of variations of pure $\bbR$-Hodge structures.
To check the equivariance, note that 
the identities $\pair{x}^*\log t_{\tau,\fra}=x^\tau\log t_{\tau,x\fra}$ and 
$\iota_x(u_{\tau,\fra})=(x^\tau)^{-1}u_{\tau,x\fra}$ imply 
$\iota_x(\gamma_\fra^\bsk)=x^{-\bsk}\gamma_{x\fra}^\bsk$, 
where $x^{-\bsk}\coloneqq \prod_{\tau\in I}(x^\tau)^{-k_\tau}$. 
This is compatible with the fact that $x$ acts on $e^\bsk\in\Sym^k\bbR(\bone)$ 
as the multiplication by $x^{-\bsk}$. 
\end{proof}

\begin{proposition}
For each integer $N>0$, $\bbLog^N$ is an $\Ft$-equivariant admissible unipotent variation of 
mixed $\bbR$-Hodge structures on $\T$. Moreover, there is a natural exact sequence 
\begin{equation}\label{eq: SES}
	0 \rightarrow \Sym^{N} \bbR(\bone) \rightarrow \bbLog^{N} \rightarrow \bbLog^{N-1} \rightarrow 0 
\end{equation}
of $\Ft$-equivariant variations of mixed $\bbR$-Hodge structures. 
\end{proposition}
\begin{proof}
We have already shown that $\bbLog^N$ is an $\Ft$-equivariant unipotent variation of 
mixed $\bbR$-Hodge structures on $\T$. The admissibility follows from Proposition \ref{prop: bbLog_fra^N}. 
To prove the latter statement, note that $\gamma_\fra^\bsk = (2\pi i)^N u_\fra^\bsk$ gives 
a global section of $\bbLog_\fra^N$ when $\absv{\bsk}=N$. 
Hence we can define an injection $\Sym^{N}\bbR(\bone)\to\bbLog^{N}$ by 
$e^\bsk\mapsto\gamma_\fra^\bsk$ on $\T^\fra$.  
This is a morphism of $\Ft$-equivariant variations of mixed $\bbR$-Hodge structures 
and its cokernel is naturally isomorphic to $\cLog^{N-1}$. 
Hence we obtain the exact sequence \eqref{eq: SES}. 
\end{proof}

\begin{definition}\label{def: log}
	We define the \emph{logarithm sheaf} $\bbLog$ to be the projective system $\bbLog=(\bbLog^N)_{N\geq0}$ 
	of $\Ft$-equivariant admissible unipotent variations of mixed $\bbR$-Hodge structures $\bbLog^N$ on $\T$ 
	with respect to the projection $\bbLog^{N}\to\bbLog^{N-1}$ in \eqref{eq: SES}. 
\end{definition}

\subsection{The Splitting Principle}\label{ssec: splitting}
This principle asserts that the logarithm sheaf splits into the direct product of $\Sym^k\bbR(\bone)$ 
at each torsion point. In our situation, this is formulated as follows. 

Let $\xi\in\bbT^\fra$ be a torsion point. 
Then the pull-back $i_\xi^*\bbLog^N$ of the logarithm sheaf $\bbLog^N$ 
along the inclusion $i_\xi\colon\{\xi\}\hookrightarrow\T$, 
in other words, the fiber of $\bbLog^N$ at the point $\xi$
inherits a mixed $\bbR$-Hodge structure. 
Moreover, for $x\in\Ft$, the map $\pair{x}\colon\T^{x\fra}\isomto\T^\fra$ induces an isormorphism 
\[\pair{x}^*\colon i_\xi^*\bbLog^N\isomto i_{\xi'}^*\bbLog^N\]
of mixed $\bbR$-Hodge structures, where we put $\xi'=\pair{x}^{-1}\xi$. 

\begin{proposition}\label{prop: splitting principle}
For each torsion point $\xi\in\T^\fra$, we have an isomorphism 
\begin{equation}\label{eq: splitting}
	i^*_\xi \bbLog^N\cong\prod_{k=0}^N \Sym^k \bbR(\bone)
\end{equation}
of mixed $\bbR$-Hodge structures. In addition, through this isomorphism and 
the one corresponding to $\xi'=\pair{x}^{-1}\xi$, 
the map $\pair{x}^*\colon i_\xi^*\bbLog^N\isomto i_{\xi'}^*\bbLog^N$ is identified 
with the action of $x$ on $\prod_{k=0}^N \Sym^k \bbR(\bone)$. 
In particular, the isomorphism \eqref{eq: splitting} is $\Delta_\xi$-equivariant, 
where $\Delta_\xi\subset\Ft$ denotes the isotropy subgroup of $\xi$. 
\end{proposition}
\begin{proof}
Choose a basis $\bsal=(\alpha_1,\ldots,\alpha_g)$ of $\fra$ and 
branches of $\log t^{\alpha_i}$. 
Since $\xi$ is a torsion point of $\T^\fra$, the values of $\log t^{\alpha_i}$ at $\xi$ belong to $2\pi i\bbQ$. 
Then the formula 
\[(2\pi i)^{\absv{\bsk}}u_\bsal^\bsk
=\exp\Biggl(\sum_{i=1}^g(\log t^{\alpha_i})u_{\alpha_i}\Biggr)\gamma_\bsal^\bsk
=\sum_{\bsn=(n_i)\in\bbN^g}\prod_{i=1}^g\frac{1}{n_i!}\biggl(\frac{\log t^{\alpha_i}}{2\pi i}\biggr)^{n_i}
\gamma_\bsal^{\bsk+\bsn}\]
shows that $(2\pi i)^{\absv{\bsk}}u_\bsal^\bsk$ evaluated at $\xi$ belongs to $i_\xi^*\bbLog^N$. 
Hence we obtain a canonical isomorphism 
\[i_\xi^*\bbLog^N=\prod_{0\le\absv{\bsk}\le N}\bbR\cdot (2\pi i)^{\absv{\bsk}}u_\bsal^\bsk
=\prod_{k=0}^N\Sym^k H_\fra, \]
which is clearly compatible with $\Ft$-action. 
Thus the isomorphism \eqref{eq: R(1) to H_fra} of $\bbR(\bone)$ and $H_\fra$ gives 
the desired isomorphism \eqref{eq: splitting}, and the $\Ft$-equivariance follows from 
Lemma \ref{lem: H}. 
\end{proof}

%
%
%
%
%
\section{The Polylogarithm Class}\label{sec: polylog}
%
%
%
%
%

In this section, we calculate the cohomology of the logarithm sheaf $\bbLog$,
and construct the polylogarithm class in the equivariant Deligne--Beilinson cohomology.

\subsection{Cohomology of the Logarithm Sheaf}\label{subsec: cohomology}
%

In this subsection, we investigate the cohomologies of $\T$ and its open subset 
$U\coloneqq\coprod_{\fra\in\frI}(\T^\fra\setminus\{1\})$ with coefficients in $\bbLog^N$, 
and their $\Ft$-equivariant structures. 
This amounts to considering the cohomologies of $\T^\fra$ and $U^\fra\coloneqq\T^\fra\setminus\{1\}$ 
for $\fra\in\frI$, and the effects of $\pair{x}^*$ on them for $x\in\Ft$. 
First note that, ignoring the equivariant structure, we have the following result. 

\begin{proposition}\label{prop: calculation}
	Let $\fra\in\frI$.
	For integers $m$ and $N>0$, the natural map
	\[
		H^m(\T^\fra, \bbLog^N_\fra) \rightarrow H^m(\T^\fra, \bbLog^{N-1}_\fra)
	\]
	is the zero map if $m\neq g$, and is an isomorphism if $m=g$.
	Moreover, we have 
	\[
		H^g(\T^\fra,\bbLog^N_\fra)\xrightarrow\cong\cdots\xrightarrow\cong H^g(\T^\fra,\bbLog^0_\fra)\cong\bbR(-g), 
	\]
    where the last isomorphism is given by 
    \[
    H^g(\T^\fra,\bbLog^0_\fra)=H^g(\T^\fra,\bbR)=\bigwedge^g H^1(\T^\fra,\bbR)=\bigwedge^g H_\fra^\vee 
    \xrightarrow{\eqref{eq: R(1) to H_fra}^\vee} \bigwedge^g \bbR(-\bone) \cong \bbR(-g). 
    \]
	In particular, we have 
	\[
		H^m(\T^\fra,\bbLog_\fra)\coloneqq\varprojlim_N H^m(\T^\fra,\bbLog^N_\fra)
		\cong\begin{cases}
			\bbR(-g)   &  m=g\\
			0 & m\neq g.
		\end{cases}
	\]
\end{proposition}

\begin{proof}
	If we fix a $\bbZ$-basis $\bsal=(\al_1,\ldots,\al_g)$ of $\fra$, then we have an isomorphism
	\[
		\T^\fra=\Spec\bbC[t^{\pm\alpha_1},\ldots,t^{\pm\alpha_g}]\cong\bbG_m^g.
	\]
	Then $\bbLog^N_\fra$ corresponds to the usual logarithm sheaf on $\bbG_m^g$,
	and our assertion
	is a special case of the calculation by Huber and Kings \cite{HK18}*{Corollary 7.1.6}
	for a general commutative group scheme given as an extension of an abelian variety by a torus. 
	An explicit proof for the case of $\bbG_m^g$ is also given in \cite{BHY18}*{Lemma 3.3}.
\end{proof}

\begin{remark}\label{rem: weights}
	By Proposition \ref{prop: calculation}, the mixed $\bbR$-Hodge structure 
	$H^{g}(\bbT^\fra,\bbLog^N_\fra)\cong\bbR(-g)$ is pure of weight $2g$.
    On the other hand, for $0\leq m<g$, the short exact sequence \eqref{eq: SES} gives rise to an \emph{injection} 
    \begin{equation}\label{eq: SES injection}
        H^{m}(\bbT^\fra,\bbLog^{N}_\fra)\to H^{m+1}(\bbT^\fra,\Sym^{N+1}\bbR(\bone)), 
    \end{equation}
    since the map $H^{m}(\bbT^\fra,\bbLog^{N+1}_\fra)\rightarrow H^{m}(\bbT^\fra,\bbLog^{N}_\fra)$ is zero 
    by Proposition \ref{prop: calculation}. 
	Noting that 
	\begin{align*}
	    H^{m+1}(\bbT^\fra,\Sym^{N+1}\bbR(\bone))&=H^{m+1}(\bbT^\fra, \bbR)\otimes \Sym^{N+1}\bbR(\bone)
	    =\bigwedge^{m+1} H^{1}(\bbT^\fra,\bbR)\otimes \Sym^{N+1}\bbR(\bone)\\
	    &=\bigwedge^{m+1} \bbR(-\bone)\otimes \Sym^{N+1}\bbR(\bone), 
	\end{align*}
	we see that the mixed $\bbR$-Hodge structure on $H^{m}(\bbT^\fra,\bbLog^N_\fra)$ is 
	pure of weight $2(m-N)$ for $0\leq m<g$.
\end{remark}

Next we consider the $\Ft$-action. What we will need later is the following: 

\begin{lemma}\label{lem: divides}
\begin{enumerate}
\item For $x\in\Ft$, the diagram 
\[\xymatrix{
     H^g(\T^\fra,\bbLog_\fra^N) \ar[r]^-{\cong} \ar[d]_{\pair{x}^*} & \bbR(-g) \ar[d]^{\Nr(x)} \\
     H^g(\T^{x\fra},\bbLog_{x\fra}^N) \ar[r]^-{\cong} & \bbR(-g)
}\]
is commutative, where $\Nr(x)$ denotes the norm of $x$ and 
the horizontal isomorphisms are those given in Proposition \ref{prop: calculation}. 
In particular, the action of $\Delta\coloneqq\cO_{F+}^\times$ on $H^g(\T^\fra,\bbLog_\fra^N)$ is trivial. 
\item Let $0\le m<g$. If $g|N$, the $\Delta$-invariant part of $H^{m}(\bbT^\fra,\bbLog^{N}_\fra)$ is given by 
\[H^m(\T^\fra,\bbLog_\fra^N)^\Delta\cong\begin{cases}
    \bbR(N) & (m=0),\\
    0 & (0<m<g). 
\end{cases}\]
\end{enumerate}
\end{lemma}

\begin{proof}
    The assertion (1) is easily verified from the construction. 
    To prove (2), we note that the $\Delta$-invariant part of 
	\[H^{m+1}(\bbT^\fra,\Sym^{N+1}\bbR(\bone))=\bigwedge^{m+1}\bbR(-\bone)\otimes\Sym^{N+1}\bbR(\bone)\]
    is nontrivial if and only if $g|(N-m)$, which holds only when $m=0$ since we assume $g|N$. 
    From this and the injection \eqref{eq: SES injection}, we obtain our assertion for $0<m<g$. 
	For $m=0$, we have $H^{0}(\bbT^\fra,\bbLog^{N}_\fra)\cong\Sym^N\bbR(\bone)$, which proves the assertion.
\end{proof}

Now we turn to the open set $U^\fra\coloneqq\bbT^\fra\setminus\{1\}$ of $\T^\fra$. 
The following Proposition is the equivariant version of \cite{BHY18}*{Proposition 3.4}, 
which calculates the cohomology of $\bbG_m^g\setminus\{1\}$. 

\begin{proposition}\label{prop: U}
    Let $N\ge 0$ be an integer. 
    \begin{enumerate}
    \item When $g>1$, we have 
    \[
	    H^m(U^\fra,\bbLog^N_\fra)
	    \cong\begin{cases}
	    	H^m(\bbT^\fra,\bbLog^N_\fra) & 0\leq m\leq g,\\
	    	\bigl(\prod_{k=0}^N\Sym^k\bbR(\bone)\bigr)(-g) & m=2g-1,\\
    		0& \text{otherwise}, 
	    \end{cases}
    \]
    and in particular, 
	\[
		H^m(U^\fra,\bbLog_\fra)\coloneqq 
		\varprojlim_N H^m(U^\fra,\bbLog^N_\fra)
		\cong\begin{cases}
			\bbR(-g) & m=g,\\
			\bigl(\prod_{k=0}^\infty\Sym^k\bbR(\bone)\bigr)(-g) & m=2g-1,\\
			0 & \text{otherwise}.
		\end{cases}
	\]
	\item When $g=1$, we have 
    \[
	    H^m(U^\fra,\bbLog^N_\fra)
	    \cong\begin{cases}
	    	H^0(\bbT^\fra,\bbLog^N_\fra) & m=0,\\
	    	H^1(\bbT^\fra,\bbLog^N_\fra)\oplus\bigl(\prod_{k=0}^N\Sym^k\bbR(\bone)\bigr)(-1) & m=1,\\
    		0& \text{otherwise}, 
	    \end{cases}
    \]
    and in particular, 	
	\[
		H^{m}(U^\fra, \bbLog_\fra)
		\cong\begin{cases}
			\bbR(-1)\oplus\prod_{k=0}^\infty \bbR(k-1) & m=1,\\
			0  & \text{otherwise}.
		\end{cases}
	\]
	\end{enumerate}
	In both cases, these isomorphisms are $\Ft$-equivariant. 
\end{proposition}

\begin{proof}
    First assume that $g>1$. 
    Let $i_\fra\colon Z^\fra\hookrightarrow\bbT^\fra$ denote the inclusion of 
    $Z^\fra\coloneqq\bbT^\fra\setminus U^\fra=\{1\}$ into $\T^\fra$. 
    Then we have the localization exact sequence 
    \begin{equation}\label{eq: localization bbLog}
	    \cdots\rightarrow H^m(\bbT^\fra, \bbLog^N_\fra) \rightarrow  H^m(U^\fra, \bbLog^N_\fra) 
	    \to  H^{m+1-2g}(Z^\fra, i^*_\fra\bbLog^N_\fra)(-g) \rightarrow \cdots 
    \end{equation}
    by \cite{Sa90}*{(4.4.1)} and the relation $i^!=\mathbf{D}i^*\mathbf{D}$ 
    where $\mathbf{D}$ denotes the dual functor. 
    From this, we obtain the isomorphism 
    \begin{equation}\label{eq: calc2}
	    H^m(U^\fra,\bbLog^N_\fra)
	    \cong\begin{cases}
	    	H^m(\bbT^\fra,\bbLog^N_\fra) & 0\leq m\leq g,\\
	    	H^{0}(Z^\fra, i^*_\fra\bbLog^N_\fra)(-g) & m=2g-1,\\
    		0& \text{otherwise}.  
	    \end{cases}
    \end{equation}
    Hence we get the result from Proposition \ref{prop: splitting principle} 
    and Proposition \ref{prop: calculation}. 
    
	When $g=1$, the above computation works for $m\ne 1$. For $m=1$, we should consider a short exact sequence 
	\begin{equation}\label{eq:g=1 short exact seq}
	    0\lra H^1(\T^\fra,\bbLog_\fra^N)\lra H^1(U^\fra,\bbLog_\fra^N)
	    \lra H^0(Z^\fra,i_\fra^*\bbLog_\fra^N)(-1)\lra 0. 
	\end{equation}
	By embedding $\bbT^\fra\cong\bbG_m$ into $\bbA^1$ and considering the residue at $0$, 
	we obtain a chain of maps 
	\[H^1(U^\fra,\bbLog_\fra^N)\lra H^1(U^\fra,\bbLog_\fra^0)\lra \bbR(-1) 
	\overset{\cong}{\longleftarrow} H^1(\T^\fra,\bbLog_\fra^0)
	\overset{\cong}{\longleftarrow} H^1(\T^\fra,\bbLog_\fra^N),\]
    which gives a splitting of \eqref{eq:g=1 short exact seq}. 
	
	Since these constructions are compatible with $\pair{x}\colon\T^{x\fra}\to\T^\fra$, 
	the assertion follows. 
\end{proof}

%
%
\subsection{Equivariant Cohomology of the Logarithm Sheaf}\label{subsec: equivariant cohomology of Log}
%
%

We next calculate the equivariant cohomology of the logarithm sheaf on $U$. 
If we fix a finite set of fractional ideals $\frC$ of $F$ representing the narrow class group $\Cl^+_F$,
then we have a natural isomorphism of $\bbR$-vector spaces 
\begin{equation}\label{eq: decomposition5}
	H^m(U/\Ft,\bbLog^N)\cong\prod_{\fra\in\frC}H^m(U^\fra/\Delta,\bbLog^N_\fra). 
\end{equation}
Thus it is equivalent to consider the cohomology groups on $U/\Ft$ 
and on $U^\fra/\Delta$ for each $\fra\in\frI$. 

Although we have not established Condition \ref{cond: RGamma_Hdg}, 
we can show the following Proposition, which we will prove in the Appendix 
(Lemmas \ref{lem: A8}, \ref{lem: A9} and Proposition \ref{prop: A11}). 

\begin{proposition}\label{prop: construction}
For $\fra\in\frI$ and an integer $N\ge 0$, there exists an absolute $\bbR$-Hodge complex 
$R\Gamma_{\mathrm{Hdg}}(U^\fra/\Delta,\bbLog_\fra^N)$ together with an isomorphism 
\[H^m(U^\fra/\Delta,\bbLog_\fra^N)\cong H^m(R\Gamma_{\mathrm{Hdg}}(U^\fra/\Delta,\bbLog_\fra^N))\]
of $\bbR$-vector spaces. Moreover, we have the following: 
\begin{enumerate}
\item 
If we equip $H^m(U^\fra/\Delta,\bbLog_\fra^N)$ 
with a mixed $\bbR$-Hodge structure via the above isomorphism, 
we have a spectral sequence 
\[E^{p,q}_2=H^p(\Delta, H^q(U^\fra,\bbLog^N_\fra))\Rightarrow H^{p+q}(U^\fra/\Delta,\bbLog^N_\fra)\]
in the category of mixed $\bbR$-Hodge structures. 
\item 
If we define the equivariant Deligne-Beilinson cohomology by 
\[H_{\sD}^m(U^\fra/\Delta,\bbLog_\fra^N)\coloneqq 
\Ext^m_{\MHS_\bbR}\bigl(\bbR(0),R\Gamma_{\mathrm{Hdg}}(U^\fra/\Delta,\bbLog_\fra^N)\bigr),\]
we have a spectral sequence 
\[E_2^{p,q}=\Ext^p_{\MHS_{\bbR}}\bigl(\bbR(0), H^{q}(U^\fra/\Delta,\bbLog_\fra^N)\bigr)
\Rightarrow H^{p+q}_{\sD}(U^\fra/\Delta,\bbLog_\fra^N). \]
\end{enumerate}
\end{proposition}

We equip $H^m(U/\Ft,\bbLog^N)$ with a mixed $\bbR$-Hodge structure 
via the isomorphism \eqref{eq: decomposition5}, and 
define the equivariant Deligne-Beilinson cohomology as
\[
	H_{\sD}^m(U/\Ft,\bbLog^N)\coloneqq\prod_{\fra\in\frC}H_{\sD}^m(U^\fra/\Delta,\bbLog^N_\fra). 
\]

In order to calculate the group cohomology of $\Delta \cong \bbZ^{g-1}$,
we will first define a free resolution of $\bbZ$ as a module over the ring of Laurent polynomials. 
Although the following Lemma is a special case of the well-known fact on the Koszul complex, 
we give a proof for convenience of the reader. 

\begin{lemma}\label{lemma: resolution}
Let $x_1,\ldots,x_n$ be indeterminates and $A^{(j)}\coloneqq \bbZ[x_j^{\pm 1}]$ 
the ring of Laurent polynomials in $x_j$ for $j=1,\ldots,n$. We also set 
$A\coloneqq A^{(1)}\otimes_\bbZ\cdots\otimes_\bbZ A^{(n)}=\bbZ[x_1^{\pm 1},\ldots,x_n^{\pm 1}]$, 
and consider the chain complex of $A$-modules 
\[K_\b\coloneqq K^{(1)}_\b\otimes_\bbZ\cdots\otimes_\bbZ K^{(n)}_\b, \]
where $K^{(j)}_\b\coloneqq [A^{(j)}\xrightarrow{x_j-1} A^{(j)}]$ in degrees $1$ and $0$. 
Then the complex $K_\b$ with the augmentation map $K_0\to \bbZ$ sending all monomials to $1$ 
gives a free resolution of $\bbZ$ considered as an $A$-module on which each $x_j$ acts trivially. 
\end{lemma}
\begin{proof}
Note that $K_\b$ is a complex of free $A$-modules. Indeed, $K_m$ is non-zero only when $0\le m\le n$, 
and then $K_m\cong A^{\binom{n}{m}}$. We prove the exactness by induction on $n$. 
When $n=0$, we have $A=\bbZ$ and $K_\b=\bbZ$ by convention, hence the assertion is trivial. 

For $n>0$, set $A'\coloneqq A^{(1)}\otimes_\bbZ\cdots\otimes_\bbZ A^{(n-1)}$ and 
$K'_\b\coloneqq K^{(1)}_\b\otimes_\bbZ\cdots\otimes_\bbZ K^{(n-1)}_\b$. 
Then we see from the construction that $K'_\b\otimes_\bbZ K^{(n)}_\b=K_\b$, which shows that 
\begin{equation}\label{eq: Koszul cone}
    \Cone\Bigl(K'_\b\otimes_\bbZ A^{(n)}\xrightarrow{x_n-1}K'_\b\otimes_\bbZ A^{(n)}\Bigr)\cong K_\b. 
\end{equation}
On the other hand, from the short exact sequence of modules 
\[0\lra A^{(n)}\xrightarrow{x_n-1} A^{(n)} \lra \bbZ \lra 0,\]
we obtain the short exact sequence of complexes 
\[0\lra K'_\b\otimes_\bbZ A^{(n)} \xrightarrow{x_n-1} K'_\b\otimes_\bbZ A^{(n)} \lra K'_\b \lra 0. \]
This gives a quasi-isomorphism 
\[\Cone\Bigl(K'_\b\otimes_\bbZ A^{(n)}\xrightarrow{x_n-1}K'_\b\otimes_\bbZ A^{(n)}\Bigr)\lra K'_\b. \]
Thus we have a quasi-isomorphism $K_\b\to K'_\b$, and the assertion follows from the induction hypothesis. 
\end{proof}

We may use the complex $K_\bullet$ for $n=g-1$ to calculate the group cohomology of a $\Delta$-module. 

\begin{proposition}\label{prop: complex} 
    For any $\Delta$-module $M$ and $m\ge g$, we have $H^m(\Delta,M)=0$. 
    Moreover, assume that $M$ is an $\bbR$-vector space on which $\Delta$ acts diagonalizablly, i.e., 
    $M$ is a direct sum of subspaces where each element of $\Delta$ acts as a scalar multiplication. 
    Then we have
	\begin{equation}\label{eq: trivial}
		H^m(\Delta,M)\cong (M^\Delta)^{\oplus\binom{g-1}{m}}. 
	\end{equation}
\end{proposition}

\begin{proof}
If we fix a generator $\eps_1,\ldots,\eps_{g-1}$ of $\Delta\cong\bbZ^{g-1}$, 
the group ring $\bbZ[\Delta]$ can be identified with the Laurent polynomial ring 
$A=\bbZ[x_1^{\pm 1},\ldots,x_{g-1}^{\pm 1}]$ of $g-1$ variables. 
By applying Lemma \ref{lemma: resolution} for $n=g-1$, we obtain a free resolution $K_\b$ 
of the trivial $\Delta$-module $\bbZ$. Hence we have $H^m(\Delta,M)\cong H^m(M^\b)$, 
where $M^\b\coloneqq \Hom_{\bbZ[\Delta]}(K_\b,M)$. 
This immediately shows the first assertion. 

To prove the second assertion, we may assume that $M$ is a one-dimensional $\bbR$-vector space 
on which each $\eps\in\Delta$ acts as a scalar. 
If the $\Delta$-action on $M$ is trivial, then the all differentials of the complex $M^\b$ are zero and 
we have $H^m(M^\b)=M^m\cong M^{\oplus\binom{g-1}{m}}$. 
On the other hand, if the action is non-trivial, we may assume that $\eps_{g-1}$ acts as a non-trivial 
scalar multiplication. Then the isomorphism of complexes 
\[M^\b\cong \Cone\Bigl(M^{\prime\b}\xrightarrow{\eps_{g-1}-1} M^{\prime\b}\Bigr)\]
induced by \eqref{eq: Koszul cone}, where 
$M^{\prime\b}\cong\Hom_{\bbZ[\Delta]}(K'_\b\otimes \bbZ[\eps_{g-1}^{\pm 1}],M)$, 
implies that $M^\b$ is exact,  since $\eps_{g-1}-1$ acts as an isomorphism. 
\end{proof}

As a consequence of Proposition \ref{prop: complex}, we have the following: 

\begin{lemma}\label{lem: Delta}
	We have
	\[
		H^m(\Delta,\Sym^k\bbR(\bone))=
		\begin{cases}
			\bbR(k)^{\oplus\binom{g-1}{m}}  &  \text{if $g|k$}, \\
			0  &  \text{otherwise}.
		\end{cases}
	\]
\end{lemma}
\begin{proof}
Recall that $(e_\tau)_{\tau\in I}$ denotes the standard basis of $\bbR(\bone)$. 
We consider the basis of $\Sym^k\bbR(\bone)$ consisting of $e^\bsk=\prod_{\tau\in I}e_\tau^{k_\tau}$, 
where $\bsk=(k_\tau)$ runs through $\bsk\in\bbN^I$ with $\absv{\bsk}=k$. 
By construction, $\eps\in\Delta$ acts on $e^\bsk$ as the multiplication by 
$\eps^{-\bsk}=\prod_{\tau\in I}(\eps^\tau)^{-k_\tau}$, 
which shows that the $\Delta$-action on $\Sym^k\bbR(\bone)$ is diagonalizable. 
Moreover, the action on $\bbR e^\bsk$ is trivial if and only if $\bsk=(k/g,\ldots,k/g)$, 
which occurs only when $k$ is divisible by $g$. Thus the result follows from Proposition \ref{prop: complex}. 
\end{proof}

\begin{lemma}\label{lem: degenerate}
	For any $\fra\in\frI$ and integer $N\geq0$, the spectral sequence
	\begin{equation}\label{eq: sseq}
		E^{p,q}_2=H^p(\Delta, H^q(U^\fra,\bbLog^N_\fra))\Rightarrow H^{p+q}(U^\fra/\Delta,\bbLog^N_\fra)
	\end{equation}
	degenerates at $E_2$.
\end{lemma}
\begin{proof}
We have to show that all differentials $d_r^{p,q}$ vanishes for $r\ge 2$. 
We only need to consider $0\le q\le g$ or $q=2g-1$ since, by Proposition \ref{prop: U}, 
we have $H^p(\Delta, H^q(U^\fra,\bbLog^N_\fra))=0$ if $q>g$ and $q\neq 2g-1$. 
When $q=2g-1$, the above vanishing also shows that $d_r^{p,q}=0$ for $2\le r\le g-1$. 
For $r\ge g$, we have $d_r^{p,q}=0$ by Proposition \ref{prop: complex}. 
Next we assume $0\le q\le g$. Then, by Remark \ref{rem: weights}, 
$H^q(U^\fra,\bbLog^N_\fra)$ is pure of weight $2g$ if $q=g$ and $2(q-N)$ if $q<g$, 
and so is $H^p(\Delta, H^q(U^\fra,\bbLog^N_\fra))$. 
This shows that $d_r^{p,q}=0$ for all $r\ge 2$ in this case. 
	\begin{figure}[ht]
		\centering
		\includegraphics[width=8cm]{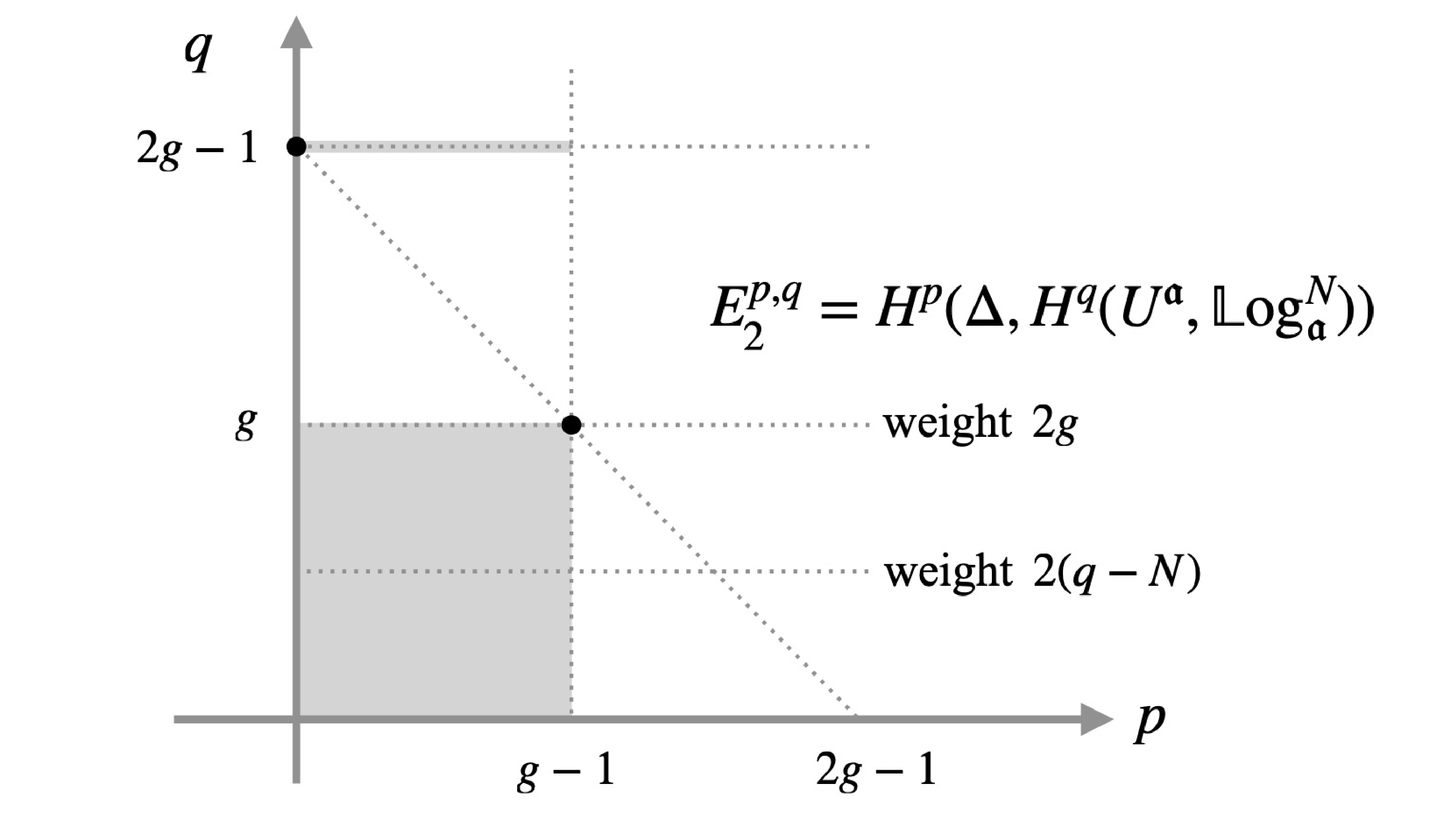}
		\caption{$E_2^{p,q}=H^p(\Delta,H^q(U^\fra,\bbLog_\fra^N))$} \label{fig: E_2 terms}
	\end{figure}
\end{proof}

\begin{proposition}\label{prop: splits}
	For any $\fra\in\frI$ and integer $N\geq 0$, we have a canonical short exact sequence 
	\begin{equation}\label{eq: split}
		0 \lra \bbR(-g) \lra H^{2g-1}(U^\fra/\Delta,\bbLog^N_\fra)
		\lra \prod_{n=0}^{\lfloor N/g \rfloor}\bbR((n-1)g) \lra 0 
	\end{equation}
	of mixed $\bbR$-Hodge structures, 
	where $\lfloor N/g\rfloor$ denotes the maximum integer not greater than $N/g$. 
	Moreover, this short exact sequence is (not canonically) split. 
\end{proposition}

\begin{proof}
    In the case of $g=1$, this follows from Proposition \ref{prop: U} (2) since the group $\Delta$ is trivial. 
    Let us assume that $g\ge 2$. 
	By Lemma \ref{lem: degenerate}, 
	the spectral sequence \eqref{eq: sseq} degenerates at $E_2$.
	For $p+q=2g-1$, the only nonzero $E_2$-terms are 
	$H^0\bigl(\Delta, H^{2g-1}(U^\fra,\bbLog^N_\fra)\bigr)$ and 
	$H^{g-1}\bigl(\Delta, H^g(U^\fra,\bbLog^N_\fra)\bigr)$.
	By Proposition \ref{prop: U} (1) and Lemma \ref{lem: Delta}, we have the exact sequence 
	\eqref{eq: split}. 
	For the case $N=0$, this gives the exact sequence
	\begin{equation}\label{eq: split2}
		0 \lra \bbR(-g) \lra H^{2g-1}(U^\fra/\Delta,\bbLog^0_\fra)
		\lra \bbR(-g) \lra 0, 
	\end{equation}
	which is split since such extension is always split as a 
	mixed $\bbR$-Hodge structure (see Lemma \ref{lem: ext1} below).
	For $N>0$, we have a commutative diagram 
	\begin{equation}\label{eq: split3}
	\xymatrix{
		0 \ar[r] &\bbR(-g) \ar[r] \ar@{=}[d]&H^{2g-1}(U^\fra/\Delta,\bbLog^N_\fra)
		\ar[r]\ar@{->>}[d]&\prod_{n=0}^{\lfloor N/g \rfloor}\bbR((n-1)g)\ar[r]\ar@{->>}[d]&0\\
		0 \ar[r]&\bbR(-g) \ar[r] & H^{2g-1}(U^\fra/\Delta,\bbLog^0_\fra)
		\ar[r] &\bbR(-g)\ar[r] &0,
	}
	\end{equation}
	and the splitting of \eqref{eq: split2} gives a splitting of \eqref{eq: split} as desired.
	This gives our assertion.
\end{proof}

\begin{corollary}\label{cor: splits}
	We have
	\[
		H^{2g-1}(U/\Ft,\bbLog^N)
		\cong\bigoplus_{\Cl^+_F}\biggl(\bbR(-g)\oplus\prod_{n=0}^{\lfloor N/g\rfloor}\bbR((n-1)g)\biggr).
	\]
\end{corollary}

\begin{proof}
	This follows from \eqref{eq: decomposition5} and Proposition \ref{prop: splits}.
\end{proof}

\begin{proposition}\label{prop: intermediate}
	Suppose $g|N$.  Then we have
	\[
		H^m(U^\fra/\Delta,\bbLog^N_\fra)=
		\begin{cases}
			 \bbR(-g)^{\oplus\binom{g-1}{m-g}}& g\leq m< 2g-1 \\
			 \bbR(N)^{\oplus\binom{g-1}{m}} &  0\leq m <g.
		\end{cases}
	\]
\end{proposition}

\begin{proof}
    By combining Lemma \ref{lem: divides}, Proposition \ref{prop: U} and Proposition \ref{prop: complex}, 
    we have 
    \[H^{m-q}(\Delta, H^{q}(U^\fra,\bbLog^N_\fra))\cong \begin{cases}
        \bbR(-g)^{\binom{g-1}{m-g}} & q=g, \\
        0 & q\ne g 
    \end{cases}\]
    for $g\le m<2g-1$, and 
    \[H^{m-q}(\Delta, H^{q}(U^\fra,\bbLog^N_\fra))\cong \begin{cases}
        \bbR(N)^{\binom{g-1}{m}} & q=0, \\
        0 & q\ne 0. 
    \end{cases}\]
    for $0\leq m<g$. 
    Thus the result follows from Lemma \ref{lem: degenerate}. 
\end{proof}

%
%
\subsection{The Polylogarithm Class}\label{subsec: polylog}
%
%

In this subsection, we define the polylogarithm class in the equivariant Deligne--Beilinson cohomology 
$H^{g-1}_{\sD}(U/\Ft,\bbLog^N)$.
Recall that this cohomology fits into the spectral sequence
\begin{equation}\label{eq: ss}
	E_2^{p,q}=\Ext^p_{\MHS_{\bbR}}\bigl(\bbR(0), H^{q}(U/\Ft,\bbLog^N)\bigr)\Rightarrow
	H^{p+q}_{\sD}(U/\Ft,\bbLog^N). 
\end{equation}
To compute the $E_2$-terms, we will use the following: 

\begin{lemma}\label{lem: ext1}
	For the Tate object $\bbR(n)$, we have
	\begin{align*}
		\Hom_{\MHS_{\bbR}}(\bbR(0),\bbR(n))
		&=\begin{cases}
		 \bbR &  n=0,\\
		0  & n\neq0,
		\end{cases}  &
		\Ext^1_{\MHS_{\bbR}}(\bbR(0),\bbR(n))
		&=\begin{cases}
		 (2\pi i)^{n-1}\bbR &  n> 0,\\
		0  & n\leq 0,
		\end{cases}
	\end{align*}
	and $\Ext^m_{\MHS_{\bbR}}(\bbR(0),\bbR(n))=0$ for $m\neq 0,1$.
\end{lemma}

\begin{proof}
	This follows from standard calculation of extension groups in $\MHS_{\bbR}$.
	See for example \cite{BHY18}*{(5.2)}.
\end{proof}

Then we have the following result. 

\begin{proposition}\label{prop: key}
	Let $N$ be a non-negative integer divisible by $g$. 
	Then we have canonical isomorphisms 
	\begin{equation}\label{eq: canonical2}
		H^{2g-1}_{\sD}(U/\Ft,\bbLog^N)
		\cong\Hom_{\MHS_\bbR}\bigl(\bbR(0),H^{2g-1}(U/\Ft,\bbLog^N)\bigr)
		\cong\begin{cases}
		    0 & N=0, \\
		    \bigoplus_{\Cl^+_F}\bbR & N>0.
		\end{cases}
	\end{equation}
	In particular, by taking the projective limit, we have a canonical isomorphism
	\begin{equation}\label{eq: canonical}
		H^{2g-1}_{\sD}(U/\Ft,\bbLog)\coloneqq \varprojlim_N H^{2g-1}_{\sD}(U/\Ft,\bbLog^N)
		\cong\bigoplus_{\Cl^+_F}\bbR.
	\end{equation}
\end{proposition}

\begin{proof}
	By Lemma \ref{lem: ext1}, 
	the spectral sequence \eqref{eq: ss} degenerates at $E_2$ and gives the short exact sequence
	\begin{align*}
		0\rightarrow
		\Ext^1_{\MHS_{\bbR}}\bigl(\bbR(0), H^{2g-2}(U/\Ft,\bbLog^N)\bigr)&\rightarrow
		H^{2g-1}_{\sD}(U/\Ft,\bbLog^N)\\
		&\rightarrow\Hom_{\MHS_{\bbR}}\bigl(\bbR(0), H^{2g-1}(U/\Ft,\bbLog^N)\bigr)
		\rightarrow 0.
	\end{align*}
	Note that $H^m(U/\Ft,\bbLog^N)\cong\bigoplus_{\fra\in\frC}H^m(U^\fra/\Delta,\bbLog^N_\fra)$.
	By Proposition \ref{prop: intermediate} and Lemma \ref{lem: ext1}, we have
	\[
		\Ext^1_{\MHS_{\bbR}}\bigl(\bbR(0), H^{2g-2}(U/\Ft,\bbLog^N)\bigr)
		\cong  \bigoplus_{\fra\in\frC}\Ext^1_{\MHS_{\bbR}}\bigl(\bbR(0),\bbR(-g)^{\oplus(g-1)}\bigr)=0.
	\]
	In addition, by Proposition \ref{prop: splits} and Lemma \ref{lem: ext1}, we have
	\begin{align*}
		\Hom_{\MHS_{\bbR}}\bigl(\bbR(0), H^{2g-1}(U/\Ft,\bbLog^N)\bigr)
		&\cong \bigoplus_{\fra\in\frC}\Hom_{\MHS_{\bbR}}\bigl(\bbR(0), H^{2g-1}(U^\fra/\Delta,\bbLog_\fra^N)\bigr)\\
		&\cong \begin{cases}
		    0 & N=0, \\
		    \bigoplus_{\Cl^+_F}\bbR & N>0.
		\end{cases}
	\end{align*}
	This gives \eqref{eq: canonical2} as desired. 
	From the construction, this isomorphism is compatible with the projection $\bbLog^N\to\bbLog^{N-g}$. 
	Thus we obtain the isomorphism \eqref{eq: canonical} 
	by taking the projective limit with respect to integers $N$ divisible by $g$. 
\end{proof}

As in Definition \ref{def: polylog}, we define the polylogarithm class as follows.

\begin{definition}\label{def: polylog class}
	For any integer $N>0$ such that $g|N$, 
	we define 
	the \textit{$N$-th polylogarithm class}
	\[
		\pol^N\in H^{2g-1}_{\sD}\bigl(U/\Ft,\bbLog^N\bigr)
	\]
	to be the element
	which maps to $(1,\ldots,1)\in \bigoplus_{\Cl^+_F}\bbR$ through the isomorphism 
	\eqref{eq: canonical2}.
	Moreover, 
	we define the \textit{polylogarithm class} $\pol$ to be their projective limit 
	\[
		\pol\coloneqq\varprojlim_N\pol^N\in H^{2g-1}_{\sD}(U/\Ft,\bbLog)=\varprojlim_N H^{2g-1}_{\sD}\bigl(U/\Ft,\bbLog^N\bigr),
	\]
	which maps to $(1,\ldots,1)\in \bigoplus_{\Cl^+_F}\bbR$ through the isomorphism 
	\eqref{eq: canonical}.
\end{definition}

%
%
%
%
%
\section{The Shintani Generating Class}\label{sec: shintani}
%
%
%
%
%

The purpose of this section is to prove our main theorem, which asserts that the de Rham 
realization of the polylogarithm class $\pol$ coincides with the de Rham Shintani class. 

%
%
\subsection{The Shintani Generating Class}\label{subsec: shintani}
%
%

In this subsection, we recall the construction of the Shintani generating class on $U$. 
We first review a complex defined in \cite{BHY20}*{\S2.2} calculating the $\Ft$-equivariant cohomology of $U$
with coefficients in an equivariant coherent $\sO_U$-module. 

Let $\fra\in\frI$. 
We say that $\alpha\in\fra_+$ is \textit{primitive} if $\alpha/N\not\in\fra_+$ for any integer $N>1$.
The set of primitive elements of $\fra_+$ is denoted by $\sA_\fra$. 
If we let $U^\fra_\alpha\coloneqq\bbT^\fra\setminus\{t^\alpha=1\}$ for any $\alpha\in\sA_\fra$,
then $\frU^\fra\coloneqq\{U^\fra_\alpha\}_{\alpha\in \sA_\fra}$ gives an affine open covering of $U^\fra$, 
which is $\Delta$-equivariant in the sense that $\pair{\varepsilon}(U^\fra_\alpha)=U^\fra_{\varepsilon\alpha}$,
and $\frU\coloneqq\{ U^\fra_\alpha\}_{\fra\in\frI,\alpha\in\sA_\fra}$ gives 
an $F^\times_+$-equivariant affine open covering of $U=\coprod_{\fra\in\frI}U^\fra$. 
Let $q\geq 0$ be an integer. For any $\bsalpha=(\alpha_0,\ldots,\alpha_q)\in\sA_\fra^{q+1}$, 
we let $U^\fra_\bsalpha\coloneqq U^\fra_{\alpha_0}\cap\cdots\cap U^\fra_{\alpha_q}$. 
For an $F^\times_+$-equivariant sheaf $\cF$ on $U$, 
we write $\cF_\fra$ for the restriction of $\cF$ to $U^\fra$, and let 
$\prod_{\bsalpha\in\sA_\fra^{q+1}}^\alt \Gamma(U^\fra_\bsalpha,\cF_\fra)$ denote 
the alternating part of the product $\prod_{\bsalpha\in\sA_\fra^{q+1}} \Gamma(U^\fra_\bsalpha,\cF_\fra)$, 
that is, 
\begin{equation*}
\prod_{\bsalpha\in\sA_\fra^{q+1}}^\alt \Gamma(U^\fra_\bsalpha,\cF_\fra)
\coloneqq\Biggl\{(s_\bsal)\in\prod_{\bsalpha\in\sA_\fra^{q+1}} \Gamma(U^\fra_\bsalpha,\cF_\fra)
\Biggm| \begin{array}{l}
s_{\rho(\bsalpha)}=\sgn(\rho)s_\bsalpha\ \text{ for all $\rho\in\frS_{q+1}$}, \\
s_\bsalpha=0 \text{ if $\alpha_i=\alpha_j$ for some $i\neq j$}
\end{array}\Biggr\}. 
\end{equation*}

For any $x\in F_+^\times$ and $\bsalpha=(\alpha_0,\ldots,\alpha_q)\in\sA^{q+1}_\fra$, 
let $x\bsalpha=(x\alpha_0,\ldots,x\alpha_q)\in\sA^{q+1}_{x\fra}$. 
Then $\bra{x}\colon U^{x\fra}\isomto U^\fra$ induces isomorphisms 
\[
	\Gamma(U^\fra_{\bsalpha},\cF_{\fra}) \overset{\bra{x}^{*}}{\cong}\Gamma(U^{x\fra}_{x\bsalpha},
	\bra{x}^{*}\cF_{\fra})
	\overset{\iota_x}{\cong}\Gamma(U^{x\fra}_{x\bsalpha},\cF_{x\fra}),
\] 
which define an action of the group $\Ft$ on 
$\prod_{\fra\in\frI}\prod_{\bsalpha\in\sA_\fra^{q+1}}^\alt \Gamma(U^{\fra}_\bsalpha,\cF_\fra)$.
Then we have the following: 

\begin{proposition}[\cite{BHY20}*{Proposition 2.13}]\label{prop: standard}
    For an $\Ft$-equivariant quasi-coherent sheaf $\cF$ on $U$, 
	let $C^\bullet(\frU/\Ft,\cF)$ be the complex consisting of the modules 
	\[
		C^q(\frU/\Ft,\cF)\coloneqq
		\Biggl(\prod_{\fra\in\frI}
		\prod_{\bsalpha\in\sA_\fra^{q+1}}^\alt \Gamma(U^\fra_\bsalpha,\cF_\fra)\Biggr)^{\Ft}
	\]
	and differentials 
	\begin{equation}\label{eq: standard}
		(d^qf)_{\alpha_0\cdots\alpha_{q+1}}\coloneqq\sum_{j=0}^{q+1}(-1)^j
		f_{\alpha_0\cdots\breve\alpha_j\cdots\alpha_{q+1}}\big|_{U^\fra_{\alpha_0\cdots\alpha_{q+1}}}. 
	\end{equation}
	Then for any integer $m\geq0$, the equivariant cohomology $H^m(U/\Ft,\cF)$ is given as
	\[
		H^m(U/\Ft,\cF)=H^m(C^\bullet(\frU/\Ft,\cF)).
	\]
\end{proposition}

Next we review the construction of the Shintani generating class given in \cite{BHY20}*{\S2.2}.

\begin{definition}
	Let $\bbR^I_+$ be the set of totally positive elements of $\bbR^I$.
	We define a \textit{cone} in $\bbR^I_+\cup\{0\}$ to be a set of the form
	\[
		\sigma_\bsalpha\coloneqq\{  x_1 \alpha_1+\cdots+x_m\alpha_m \mid x_1,\ldots,x_m \in\bbR_{\geq0}\}
	\]
	for some $m\geq 0$ and $\bsalpha=(\alpha_1,\ldots,\alpha_m)\in (F^\times_+)^m$.
	The dimension $\dim\sigma$ of a cone is the dimension 
	of the $\bbR$-vector space generated by $\sigma$. 
\end{definition}

Let us fix a numbering of elements of $I$ as $I=\{\tau_1,\ldots,\tau_g\}$. 
Then, for any subset $R\subset\bbR_+^I\cup\{0\}$, we let 
\[
	\breve R\coloneqq\{(x_{\tau_1},\ldots,x_{\tau_g})\in\bbR_+^I \mid
	\exists\delta>0,\,0<\forall\delta'<\delta, 
	(x_{\tau_1},\ldots,x_{\tau_{g-1}},x_{\tau_g}-\delta')\in R\}.
\]
As in \cite{BHY20}*{Definition 2.16},
we define the function $\cG^\fra_\sigma(t)$ as follows.

\begin{definition}\label{def: G}
	Let $\fra\in\frI$ and $\bsalpha=(\alpha_1,\ldots,\alpha_g)\in\sA_\fra^g$, 
	and put $\sigma=\sigma_\bsalpha$. Then we define 
	\[
		\cG^\fra_\sigma(t)\coloneqq 
		\frac{\sum_{\alpha\in\breve P_\bsalpha\cap\fra} t^\alpha}{(1-t^{\alpha_1})\cdots(1-t^{\alpha_g})}
	    \in\Gamma(U^\fra_\bsalpha,\sO_{\bbT}),
	\]
	where $P_\bsalpha\coloneqq\{ x_1\alpha_1+\cdots+x_g\alpha_g\mid
	\forall i\,\,0\leq x_i < 1\}$ is the parallelepiped spanned by $\alpha_1,\ldots,\alpha_g$.	
\end{definition}

For $\bsalpha=(\alpha_1,\ldots,\alpha_g)\in\sA_\fra^g$, 
let $\sgn(\bsalpha)\in\{0,\pm1\}$ be the sign of $\det\bigl(\alpha_j^{\tau_i}\bigr)$, 
the determinant of the $g\times g$ matrix $\bigl(\alpha_j^{\tau_i}\bigr)_{i,j}$. 
Then we have the following: 

\begin{proposition}[\cite{BHY20}*{Proposition 2.17}]\label{prop: generating}
	For any $\fra\in\frI$ and $\bsalpha\in\sA_\fra^g$, let
	\[
		\cG^\fra_\bsalpha
		\coloneqq\sgn(\bsalpha)\cG^\fra_{\sigma_\bsalpha}(t)\in\Gamma(U^\fra_\bsalpha,\sO_{\bbT}).
	\]
	Then $(\cG^\fra_\bsalpha)$ is a cocycle in $C^{g-1}(\frU/F_+^\times,\sO_{\bbT})$, 
	hence defines a cohomology class
	\[
		\cG\coloneqq[(\cG^\fra_\bsalpha)]\in H^{g-1}(U/F_+^\times,\sO_{\bbT}),
	\]
	which we call the \emph{Shintani generating class}.
\end{proposition}

We call $\cG$ the Shintani \emph{generating} class since it generates the 
special values of Lerch zeta functions in the sense of Theorem \ref{thm: generate} below. 
We first review the definition of the Lerch zeta functions.

\begin{definition}\label{def: Lerch}
	For a torsion point $\xi$ of $\bbT^\fra(\bbC)=\Hom_\bbZ(\fra,\bbC^\times)$, 
	let $\xi\Delta$ denote the $\Delta$-orbit of $\xi$ under the action $\xi^\eps(\alpha)\coloneqq \xi(\eps\alpha)$. 
	The same symbol also denotes the sum of characters in that orbit, that is, 
	\[\xi\Delta(\alpha)\coloneqq \sum_{\eps\in\Delta_\xi\backslash\Delta}\xi^\eps(\alpha) \qquad (\alpha\in\fra).\]
	Then we define the \textit{Lerch zeta function} $\cL(\xi\Delta,s)$ by the series
	\[
		\cL(\xi\Delta,s)\coloneqq\sum_{\alpha\in\Delta\backslash\fra_+}  
		\xi\Delta(\alpha)\Nr(\fra^{-1}\alpha)^{-s}
	\]
	for $\Re(s)>1$. The Lerch zeta function $\cL(\xi\Delta,s)$ continues meromorphically 
	to the whole complex plane, and is entire if $\xi\neq 1$.
\end{definition}

For $\fra\in\frI$, let $(\partial_{\tau,\fra})_{\tau\in I}$  be the system of vector fields on $\T^\fra$ 
which is dual to $(\dlog t_{\tau,\fra})_{\tau\in I}$ 
(here the differential forms $\dlog t_{\tau,\fra}$ are defined by \eqref{eq: base change}). 
If $(\alpha_1,\ldots,\alpha_g)$ is a basis of $\fra$, $\partial_{\tau,\fra}$ is explicitly written as 
\[\partial_{\tau,\fra}=\sum_{i=1}^g  \alpha_i^\tau t^{\alpha_i}\frac{\partial}{\partial t^{\alpha_i}} \]
in terms of the coordinate system $(t^{\alpha_1},\ldots,t^{\alpha_g})$ on $\T^\fra$. 
Then we define a differential operator $\partial\colon\sO_\T\to\sO_\T$ by 
\[\partial|_{\T^\fra}\coloneqq \frac{1}{\Nr\fra}\prod_{\tau\in I}\partial_{\tau,\fra}. \]
This operator $\partial$ is $\Ft$-equivariant and hence induces a map 
\[\partial\colon H^{g-1}(U/\Ft,\sO_\bbT)\rightarrow H^{g-1}(U/\Ft,\sO_\bbT).\] 
For any nontrivial torsion point $\xi\in\bbT^\fra(\bbC)$ and integer $k\geq 0$, let 
\[
	\partial^k\cG(\xi)\in H^{g-1}(\xi\Delta/\Delta,\sO_{\xi\Delta})
\] 
be the pull-back of $\partial^k\cG$ with respect to the equivariant 
morphism $\xi\Delta\rightarrow U$, where the $\Delta$-orbit $\xi\Delta$ is regarded as 
a $\Delta$-scheme with respect to the natural action. 
Then we have the following formula: 

\begin{theorem}\label{thm: generate}
	For any torsion point $\xi\in\bbT^\fra(\bbC)$ and integers $k\geq 0$, we have 
	\[
		\partial^k\cG(\xi)=\cL(\xi\Delta,-k)\in\bbQ(\xi)
	\]
	through a canonical isomorphism 
	\[
		H^{g-1}(\xi\Delta/\Delta,\sO_{\xi\Delta})\cong\bbQ(\xi),
	\]
	where $\bbQ(\xi)$ denotes the minimal extension of $\bbQ$ containing the image of $\xi$.
\end{theorem}

\begin{proof}
This is a straightforward generalization of \cite{BHY19}*{Theorem 5.1}. 
\end{proof}

%
%
\subsection{Characterization of the de Rham Shintani Class}\label{subsec: characterization}
%
%

In this subsection, we consider a variant $\cS$ of the Shintani generating class, 
which we call the \emph{de Rham Shintani class}, 
in the de Rham cohomology with coefficients in $\cLog^N$. 
This element $\cS$ is characterized in terms of the residue homomorphism 
(Proposition \ref{prop: characterization}). 

To construct $\cS$, we define a morphism $\sO_\T[-g]\to\cLog^N\otimes\Omega_\T^\bullet$ 
of complexes of $\Ft$-equivariant sheaves on $\T$ as follows: On $\T^\fra$, 
it maps any section $f$ to 
\[fu_\fra^{1_I}\otimes \frac{\sgn(\bsalpha)}{\absv{P_\bsalpha\cap\fra}}
\frac{dt^{\alpha_1}}{t^{\alpha_1}}\wedge\cdots\wedge\frac{dt^{\alpha_g}}{t^{\alpha_g}}, \]
where $1_I\coloneqq (1,\ldots,1)\in\bbN^I$ and 
$\bsalpha=(\alpha_1,\ldots,\alpha_g)$ is a tuple of elements of $\fra$ linearly independent over $\bbQ$. 
This does not depend on $\bsalpha$, and induces a homomorphism 
\begin{equation}\label{eq: inclusion}
	H^{g-1}(U/\Ft,\sO_\T)\lra H^{2g-1}_\dR(U/\Ft,\cLog^N). 
\end{equation}
The image of the Shintani generating class $\cG$ under this map is called the de Rham Shintani class 
and denoted by $\cS$. 

The de Rham cohomology $H_\dR^m(U/\Ft,\cLog^N)$ is described in terms of the \v{C}ech--de Rham complex 
as follows. 
Let $C^\bullet(\frU/\Ft,\cLog^N\otimes\Omega_U^\bullet)$ be the double complex 
associated with the open covering $\frU$ of $U$ introduced in the preceding subsection, 
and $C_\dR^\bullet(\frU/\Ft,\cLog^N)$ its total complex. 
Then we have 
\[
H_\dR^m(U/\Ft,\cLog^N)\cong H^m(C_\dR^\bullet(\frU/\Ft,\cLog^N)). 
\]
In fact, this can be shown in the same way as Proposition \ref{prop: standard}, 
i.e., by constructing an acyclic resolution of the complex $\cLog^N\otimes\Omega_U^\bullet$ 
from the covering $\frU$. 
In terms of this \v{C}ech--de Rham complex, the de Rham Shintani class $\cS$ is given by the cocycle 
\[
(\cS^\fra_\bsalpha)\in C^{g-1}(\frU/\Ft,\cLog^N\otimes\Omega_U^g), 
\]
where 
\begin{equation}\label{eq: S^fra_bsalpha}
\cS^\fra_\bsalpha\coloneqq \cG^\fra_\bsalpha\, u_\fra^{1_I}\otimes 
\frac{\sgn(\bsalpha)}{\absv{P_\bsalpha\cap\fra}}
\frac{dt^{\alpha_1}}{t^{\alpha_1}}\wedge\cdots\wedge\frac{dt^{\alpha_g}}{t^{\alpha_g}} 
\end{equation}
for $\fra\in\frI$ and $\bsalpha=(\alpha_1,\ldots,\alpha_g)\in\sA_\fra^g$. 

Note that there is a natural map 
\begin{equation}\label{eq: degeneration}
	 H^{2g-1}_\dR(U/\Ft,\cLog^N) \to H^{2g-1}_\dR(U,\cLog^N)
	 =\prod_{\fra\in\frI} H^{2g-1}_\dR(U^\fra,\cLog^N_\fra), 
\end{equation}
obtained by forgetting the action of $\Ft$. 
To describe explicitly the image of $\cS$ under this map, 
let us fix $\fra\in\frI$ and $\bsalpha=(\alpha_1,\ldots,\alpha_g)\in\sA_\fra^g$ giving a basis of $\fra$. 
Then $\frU^\fra_\bsalpha\coloneqq\{U^\fra_{\alpha_1},\ldots,U^\fra_{\alpha_g}\}$ is a \v{C}ech covering of $U^\fra$, 
and the de Rham cohomology $H_\dR^m(U^\fra,\cLog_\fra^N)$ is calculated by 
the usual \v{C}ech--de Rham complex 
\[
	C^{p,q}(\frU^\fra_\bsalpha, \cLog_\fra^N)\coloneqq\prod_{i_0<\cdots<i_q}
	\Gamma(U^\fra_{\alpha_{i_0}\cdots\alpha_{i_q}},\cLog^N_\fra\otimes\Omega^p_{U^\fra}),
\]
where $U^\fra_{\alpha_{i_0}\cdots\alpha_{i_q}}\coloneqq U^\fra_{\alpha_{i_0}}\cap\cdots\cap U^\fra_{\alpha_{i_q}}$.
The $\fra$-component of the image of $\cS$ under the map \eqref{eq: degeneration} is given by 
\[
\cS^\fra_\bsalpha\in C^{g,g-1}(\frU^\fra_\bsalpha,\cLog_\fra^N)
=\Gamma(U^\fra_\bsalpha,\cLog^N_\fra\otimes\Omega^g_{U^\fra}). 
\]

Next we consider the de Rham counterpart of the exact sequence \eqref{eq: localization bbLog}: 
\[
    \cdots\rightarrow H^m_\dR(\bbT^\fra, \cLog^N_\fra) \rightarrow  H^m_\dR(U^\fra, \cLog^N_\fra) 
    \to  H^{m+1-2g}_\dR(Z^\fra, i^*_\fra\cLog^N_\fra)(-g) \rightarrow \cdots.  
\]
In particular, we have a canonical map 
\begin{equation}\label{eq: residue map'}
H^{2g-1}_\dR(U^\fra,\cLog_\fra^N)\lra H^0_\dR(Z^\fra,i_\fra^*\cLog_\fra^N)(-g)
\cong\prod_{\absv{\bsk}\le N}\bbC\, u_\fra^\bsk. 
\end{equation}
In order to describe this map explicitly, 
we replace the \v{C}ech--de Rham complex $C^{\bullet,\bullet}(\frU^\fra_\bsalpha, \cLog^N_\fra)$ 
with the logarithmic \v{C}ech--de Rham complex defined as follows. 
Set $E_i\coloneqq\{t^{\alpha_i}=1\}\subset\T^\fra$, so that $Z^\fra=E_1\cap\cdots\cap E_g$. 
Let $C^\bullet_\dR(\frU^\fra_\bsalpha,\cLog^N_\fra(E))$ be the total complex of the double complex 
\[
	C^{p,q}(\frU^\fra_\bsalpha,\cLog^N_\fra(E))\coloneqq 
	\prod_{i_0<\cdots<i_q} \Gamma\bigl(\bbT^\fra, 
	\cLog^N_\fra\otimes\Omega^p_{\bbT^\fra}(E_{i_0\cdots i_q})\bigr), 
\]
where $E_{i_0\cdots i_q}\coloneqq E_{i_0}\cup\cdots\cup E_{i_q}$ 
and $\Omega^p_{\bbT^\fra}(E_{i_0\cdots i_q})$ denotes the sheaf of $p$-differentials 
with logarithmic poles along $E_{i_0\cdots i_q}$. 

\begin{lemma}\label{lem: log de Rham}
The natural inclusion 
\[
	C^\bullet_\dR(\frU^\fra_\bsalpha,\cLog^N_\fra(E))\hookrightarrow C^\bullet_\dR(\frU^\fra_\bsalpha,\cLog^N_\fra)
\]
induces an isomorphism 
\[
	H^m\bigl(C^\bullet_\dR(\frU^\fra_\bsalpha,\cLog^N_\fra(E))\bigr)
	\cong H^m\bigl(C^\bullet_\dR(\frU^\fra_\bsalpha,\cLog^N_\fra)\bigr). 
\]
\end{lemma}
\begin{proof}
For each $i_0<\cdots<i_q$, put $V\coloneqq U^\fra_{\alpha_{i_0}\cdots\alpha_{i_q}}$ 
and let $j\colon V\to \T^\fra$ be the inclusion. 
Then the natural inclusion 
\[
\cLog^N_\fra\otimes\Omega^\bullet_{\bbT^\fra}(E_{i_0\cdots i_q})
\hookrightarrow j_*\cLog_\fra^N\otimes\Omega^\bullet_{U^\fra_{\alpha_{i_0}\cdots\alpha_{i_q}}}
\]
is a quasi-isomorphism of complexes of sheaves on $\T^\fra$ by \cite{Del70}*{II.3.14}. 
Since $\T^\fra$ and $V$ are affine varieties and $j$ is an affine morphism, we have quasi-isomorphisms 
\begin{align*}
\Gamma\bigl(\T^\fra,\cLog_\fra^N\otimes\Omega_{\T^\fra}^\bullet(E_{i_0\cdots i_q})\bigr)
&\cong R\Gamma\bigl(\T^\fra,\cLog_\fra^N\otimes\Omega_{\T^\fra}^\bullet(E_{i_0\cdots i_q})\bigr)
\cong R\Gamma\bigl(\T^\fra,j_*\cLog_\fra^N\otimes\Omega^\bullet_{V}\bigr)\\
&\cong R\Gamma\bigl(\T^\fra,Rj_*\cLog_\fra^N\otimes\Omega^\bullet_{V}\bigr)
\cong R\Gamma\bigl(V,\cLog_\fra^N\otimes\Omega^\bullet_{V}\bigr)\\
&\cong \Gamma\bigl(V,\cLog_\fra^N\otimes\Omega^\bullet_{V}\bigr). 
\end{align*}
Thus the assertion immediately follows. 
\end{proof}

Consider the map
\[
	\cLog_\fra^N\otimes\Omega^g_{\bbT^\fra}(E_{1\cdots g})
	\lra i_{\fra*}(i_\fra^*\cLog_\fra^N\otimes\Omega^0_{Z^\fra})
\]
given in any neighborhood of $Z^\fra$ by
\[
	f\, u_\fra^\bsk\otimes\frac{dt^{\alpha_1}}{t^{\alpha_1}-1}\wedge\cdots
	\wedge\frac{dt^{\alpha_g}}{t^{\alpha_g}-1}
	\longmapsto f|_{ Z^\fra}\,u_\fra^\bsk\otimes 1.
\]
This map induces a map of complexes
\[
	C^\bullet_\dR(\frU^\fra_\bsalpha,\cLog^N_\fra(E))
	\rightarrow \Gamma(Z^\fra,i_\fra^*\cLog_\fra^N\otimes\Omega^\bullet_ {Z^\fra})[-2g+1],
\]
hence the residue map
\begin{equation}\label{eq: residue}
	\res_\fra\colon H^{2g-1}_\dR(U^\fra,\cLog^N_\fra)\rightarrow
	H^0_\dR(Z^\fra, i_\fra^*\cLog^N_\fra)=\prod_{\absv{\bsk}\leq N}\bbC u^\bsk_\fra.
\end{equation}

\begin{proposition}
The residue map $\res_\fra$ coincides with the map \eqref{eq: residue map'}. 
In particular, we have an exact sequence 
\begin{equation}\label{eq: Gyzin}
	H^{2g-1}_\dR(\bbT^\fra,\cLog^N_\fra)
	\to H^{2g-1}_\dR(U^\fra,\cLog^N_\fra)
	\xrightarrow{\res_\fra} H^0_\dR(Z^\fra, i_\fra^*\cLog^N_\fra)
	\to H^{2g}_\dR(\bbT^\fra,\cLog^N_\fra). 
\end{equation}
\end{proposition}

\begin{proof}
By passing through a sufficiently small neighborhood of $Z^\fra$ in $\bbT^\fra$, 
we can replace the coefficient $\cLog_\fra^N$ with (a direct sum of) $\cO_{\bbT^\fra}$. 
Then, in terms of differential forms, the map \eqref{eq: residue map'} is given by the integral 
over a $(2g-1)$-dimensional sphere around $Z^\fra$. 
On the other hand, the residue map $\res_\fra$ is given by $(2\pi i)^{-g}$ times the integral 
over the $(g-1)$-dimensional (real) torus 
$\{\lvert t^{\alpha_i}-1\rvert=\varepsilon,\,i=1,\ldots,g\}$ for some small $\varepsilon>0$. 
The equality of these two kinds of integrals is described in \cite{GriHar}*{Chapter 5, \S1}. 
\end{proof}

Since $\bbT^\fra$ is affine of dimension $g$, we have $H^{m}_\dR(\bbT^\fra,\cLog^N_\fra)=0$ for any $m>g$. 
Hence we see that $\res_\fra$ is an isomorphism if $g>1$. 

The de Rham Shintani class may be characterized as follows.

\begin{proposition}\label{prop: characterization}
	For any integer $N>0$ such that $g|N$, 
	the de Rham Shintani class $\cS$ 
	is a class in $H^{2g-1}_\dR(U/\Ft,\cLog^N)$
	characterized by the following properties: 
	\begin{enumerate}
		\item[(i)] $\cS$ maps to $(-1)^g$ times $(u_\fra^{1_I})_{\fra\in\frI}$ through the map 
		\begin{align*}
		    H^{2g-1}_\dR(U/\Ft,\cLog^N)
		    &\xrightarrow{\ \eqref{eq: degeneration}\ }
		    H^{2g-1}_\dR(U,\cLog^N)=\prod_{\fra\in\frI}H^{2g-1}_\dR(U^\fra,\cLog_\fra^N)\\
		    &\xrightarrow{\prod\res_\fra}
		    \prod_{\fra\in\frI}H^0_\dR(Z^\fra,i_\fra^*\cLog_\fra^N)
		    \cong\prod_{\fra\in\frI}\prod_{\absv{\bsk}\le N}\bbC\,u_\fra^\bsk. 
		\end{align*}
		\item[(ii)] $\cS$ maps to $0$ through the homomorphism
		\[
			H^{2g-1}_\dR(U/\Ft,\cLog^N)\rightarrow H^{2g-1}_\dR(U/\Ft,\cLog^0)
		\]
		induced from the surjection $\cLog^N\rightarrow\cLog^0$.
	\end{enumerate}
\end{proposition}

\begin{proof}
	Property (i) follows from the definition of the map \eqref{eq: degeneration},
	the definition of the residue map \eqref{eq: residue}, 
	and the description of the de Rham Shintani class in terms of the cocycle  \eqref{eq: S^fra_bsalpha}.
	Property (ii) follows immediately from the fact that $u_\fra^{1_I}=0$ in $\cLog^0_\fra$. 
	
	Let us show the uniqueness. First we assume $g>1$. 
	The spectral sequence 
	\[
    E_2^{p,q}=H^p(\Ft,H^q_\dR(U,\cLog^N))\Longrightarrow H^{p+q}_\dR(U/\Ft,\cLog^N) 
	\]
	and a similar one for $N=0$ give the commutative diagram with exact rows 
	\[\small\xymatrix{
		0\ar[r]& H^{g-1}(\Ft,H^g_\dR(U,\cLog^N))\ar[d]^\cong
		\ar[r]&  H^{2g-1}_\dR(U/\Ft,\cLog^N)\ar[d]
		\ar[r]& H^{2g-1}_\dR(U,\cLog^N)^{\Ft}\ar[d]
		\ar[r]& 0\\
		0\ar[r]& H^{g-1}(\Ft,H^g_\dR(U,\cLog^0))
		\ar[r]&  H^{2g-1}_\dR(U/\Ft,\cLog^0)
		\ar[r]&   H^{2g-1}_\dR(U,\cLog^0)^{\Ft}
		\ar[r]& 0,
	}\]
	where the exactness can be shown by similar arguments to the proof of Proposition \ref{prop: splits} 
	and the left vertical isomorphism follows from \eqref{eq: calc2} and Proposition \ref{prop: calculation}.
	This shows that any element in $H^{2g-1}_\dR(U/\Ft,\cLog^N)$ 
    is characterized by its images in $H^{2g-1}_\dR(U,\cLog^N)^{\Ft}$ and $H^{2g-1}_\dR(U/\Ft,\cLog^0)$ as desired. 
    
    For $g=1$, we note that $H^m_\dR(U/\Ft,\cLog^N)\cong H^m_\dR(U^\bbZ,\cLog_\bbZ^N)$. 
    The exact sequence \eqref{eq: Gyzin} gives the commutative diagram 
	\[\small\xymatrix{
		0\ar[r]& H^1_\dR(\T^\bbZ,\cLog_\bbZ^N)\ar[d]^\cong
		\ar[r]&  H^1_\dR(U^\bbZ,\cLog_\bbZ^N)\ar[d]
		\ar[r]& H^0_\dR(Z^\bbZ,i_\bbZ^*\cLog_\bbZ^N)\ar[d]
		\ar[r]& 0\\
		0\ar[r]& H^1_\dR(\T^\bbZ,\cLog_\bbZ^0)
		\ar[r]&  H^1_\dR(U^\bbZ,\cLog_\bbZ^0)
		\ar[r]& H^0_\dR(Z^\bbZ,i_\bbZ^*\cLog_\bbZ^0)
		\ar[r]& 0, 
	}\]
    which shows the assertion. 
\end{proof}

%
%
\subsection{The Main Theorem}\label{subsec: main}
%
%

In this subsection, we prove our main theorem, Theorem \ref{thm: main}, which asserts
that the de Rham realization of the polylogarithm class $\pol$ coincides with the
de Rham Shintani class $\cS$.

Let $N>0$ be such that $g|N$.
Recall that we have an isomorphism \eqref{eq: canonical2} 
\[
	H^{2g-1}_{\sD}(U/\Ft,\bbLog^N)\cong\Hom_{\MHS_{\bbR}}\bigl(\bbR(0),H^{2g-1}(U/\Ft,\bbLog^N)\bigr).		
\]
Then, by evaluating $1\in\bbR(0)$, we obtain a natural injection 
\[
    H^{2g-1}_{\sD}(U/\Ft,\bbLog^N)\hookrightarrow H^{2g-1}(U/\Ft,\bbLog^N). 
\]
In addition, we have a canonical isomorphism 
\[
H^{2g-1}(U/\Ft,\bbLog^N)\otimes\bbC=H^{2g-1}_\dR(U/\Ft,\cLog^N), 
\]
since the de Rham complex $\cLog^N\otimes \Omega_U^\bullet$ gives a resolution of the sheaf $\bbLog^N_\bbC$. 
Therefore, we obtain a natural injection 
\begin{equation}\label{eq: natural injection}
    H^{2g-1}_{\sD}(U/\Ft,\bbLog^N)\hookrightarrow H^{2g-1}_\dR(U/\Ft,\cLog^N). 
\end{equation}

\begin{theorem}\label{thm: main}
    The injection \eqref{eq: natural injection} maps the polylogarithm class $\pol^N$ 
    to $(-1)^g$ times the de Rham Shintani class $\cS$. 
\end{theorem}
\begin{proof}
Let us confirm the conditions (i) and (ii) in Proposition \ref{prop: characterization} 
for the image of $(-1)^g\pol^N$. 
First note that we have the following natural homomorphisms 
\begin{align*}
H^{2g-1}_{\sD}(U/\Ft,\bbLog^N)
&\lra H^{2g-1}(U/\Ft,\bbLog^N)\\
&\lra H^{2g-1}(U,\bbLog^N)=\prod_{\fra\in\frI}H^{2g-1}(U^\fra,\bbLog_\fra^N)\\
&\lra \prod_{\fra\in\frI}\Biggl(\prod_{k=0}^N\Sym^k\bbR(\bone)\Bigg)(-g)
=\prod_{\fra\in\frI}\prod_{\absv{\bsk}\le N} \bbR\cdot(2\pi i)^{\absv{\bsk}-g}\,u_\fra^\bsk, 
\end{align*}
where the last map is induced from Proposition \ref{prop: U}. 
By construction, the image of $\pol^N$ under the composite of these maps is $(u_\fra^{1_I})$. 
Since we have a natural commutative diagram 
\[
\xymatrix{
H^{2g-1}(U/\Ft,\bbLog^N) \ar[r] \ar[d]
&H^{2g-1}(U^\fra,\bbLog_\fra^N)\ar[r] \ar[d] 
& \prod_{\absv{\bsk}\le N} \bbR\cdot(2\pi i)^{\absv{\bsk}-g}\,u_\fra^\bsk \ar[d]\\
H^{2g-1}_\dR(U/\Ft,\cLog^N) \ar[r] 
&H^{2g-1}_\dR(U^\fra,\cLog_\fra^N)\ar[r]^-{\res_\fra} 
& \prod_{\absv{\bsk}\le N} \bbC\,u_\fra^\bsk 
}
\]
for each $\fra\in\frI$, we see that the image of $(-1)^g\pol^N$ satisfies the condition (i). 
On the other hand, it follows from Proposition \ref{prop: key} that $\pol^N$ vanishes under the map 
\[
H^{2g-1}_{\sD}(U/\Ft,\bbLog^N)\hookrightarrow H^{2g-1}(U/\Ft,\bbLog^N)\lra H^{2g-1}(U/\Ft,\bbLog^0), 
\]
which implies the condition (ii). 
\end{proof}

%
%
%
%
%
\section{The Plectic Polylogarithm Class}\label{sec: plectic}
%
%
%
%
%

%
%
\subsection{Construction of the Plectic Polylogarithm}
%
%
In \cite{NS16}, \Nekovar and Scholl defined 
the category of $g$-mixed plectic $\bbR$-Hodge structures to be the category $\Rep_{\bbR}(\cG^g)$, 
where $\cG^g$ is the $g$-fold product of the Tannakian fundamental group $\cG$ of 
the category $\MHS_{\bbR}$ of mixed $\bbR$-Hodge structures.
This category is known to be equivalent to the category $\MHS^{\boxtimes I}_{\bbR}$
defined as follows (see \cite{BHKY17}*{Theorem 5.7}). 

\begin{definition}\label{def: MHSg}
	We define $\MHS^{\boxtimes I}_{\bbR}$ to be the category whose object is a triple 
	$V = (V, \{W^\tau_\bullet\}, \{F^\bullet_\tau\})$, 
	where $V$ is a finite dimensional $\bbR$-vector space, $\{W^\tau_\bullet\}_{\tau\in I}$ is 
	a family of finite ascending filtrations by $\bbR$-linear subspaces of $V$,
	and $\{F^\bullet_\tau\}_{\tau\in I}$ is 
	a family of finite descending filtrations by $\bbC$-linear subspaces of $V_\bbC\coloneqq V \otimes \bbC$, 
	satisfying the following conditions:
	\begin{enumerate}
		\item For $\tau\in I$, the triple $(V,W^\tau_\bullet,F^\bullet_\tau)$ is 
		a mixed $\bbR$-Hodge structure in the usual sense.
		\item For $\tau,\nu\in I$ with $\tau\ne\nu$, the spaces $W^\nu_n V_\bbC$, $F^m_\nu V_\bbC$, 
		$\bar F^m_\nu V_\bbC$ are mixed $\bbC$-Hodge structures with respect to the filtrations
		induced from $W^\tau_\bullet$, $F^\bullet_\tau$ and $\bar F^\bullet_\tau$.
	\end{enumerate}
	A morphism in $\MHS^{\boxtimes I}_{\bbR}$ is an $\bbR$-linear homomorphism of the underlying $\bbR$-vector spaces 
	compatible with the filtrations.
\end{definition}

In this article, a mixed plectic $\bbR$-Hodge structure refers to any object in $\MHS^{\boxtimes I}_{\bbR}$.
One standard method for constructing an object in $\MHS^{\boxtimes I}_{\bbR}$ is 
to consider the outer product of objects in $\MHS_{\bbR}$. 

\begin{proposition}
	Let $\{V^{(\tau)}\}_{\tau\in I}$ be a family of objects in $\MHS_{\bbR}$. 
	We define the \emph{outer product} 
	$V\coloneqq \bigboxtimes_{\tau\in I} V^{(\tau)}$ 
	to be the triple $V = (V, \{W^\tau_\bullet\}, \{F^\bullet_\tau\})$ consisting of the $\bbR$-vector space 
	\[
		V \coloneqq \bigotimes_{\tau\in I}V^{(\tau)}, 
	\]
	and filtrations $W^\tau_\bullet$ of $V$ and $F^\bullet_\tau$ of $V_\bbC\coloneqq V\otimes\bbC$ given by
	\begin{align*}
		W^\tau_n V &\coloneqq \bigotimes_{\nu\ne\tau} V^{(\nu)} \otimes W^\tau_n V^{(\tau)},  &
		F^p_\tau V_\bbC &\coloneqq \bigotimes_{\nu\ne\tau}V^{(\nu)}_{\bbC} \otimes F^p_\tau V^{(\tau)}_{\bbC}. 
	\end{align*}
	Then $V$ is an object of $\MHS^{\boxtimes I}_\bbR$. 
\end{proposition}

\begin{proof}
The conditions of Definition \ref{def: MHSg} can be easily verified. 
\end{proof}

\begin{example}\label{ex: plectic Tate}
	For any $\bsn=(n_\tau)_{\tau\in I}\in\bbZ^I$, we let
	\[
		\bbR(\bsn)\coloneqq \bigboxtimes_{\tau\in I}\bbR(n_\tau),
	\]
	where $\bbR(n_\tau)$ is the usual Tate object in $\MHS_{\bbR}$.
	Then $\bbR(\bsn)$ is an object in $\MHS^{\boxtimes I}_{\bbR}$, which we call the \emph{plectic Tate object}.
\end{example}

We have the following result on the extensions of plectic Tate objects. 

\begin{lemma}\label{lem: ext2}
	For $n\in\bbZ$, let $n_I\coloneqq(n,\ldots,n)\in\bbZ^I$ and 
	$\bbR(n_I)$ the corresponding plectic Tate object. 
	Then we have
	\begin{align*}
		\Hom_{\MHS^{\boxtimes I}_{\bbR}}(\bbR(0_I),\bbR(n_I))
		&=\begin{cases}
		 \bbR &  n=0,\\
		\{0\}  & n\neq0,
		\end{cases}  &
		\Ext^g_{\MHS^{\boxtimes I}_{\bbR}}(\bbR(0_I),\bbR(n_I))
		&=\begin{cases}
		(2\pi i)^{g(n-1)}\bbR &  n> 0,\\
		\{0\}  & n\leq 0,
		\end{cases}
	\end{align*}
	and $\Ext^m_{\MHS_{\bbR}}(\bbR(0_I),\bbR(n_I))=\{0\}$ for $m\neq 0,g$.
\end{lemma}

\begin{proof}
	This is essentially \cite{BHKY17}*{Example 5.28}, but 
	corrected for the action of the complex conjugation on
	$\bbC(\bsn)$.
\end{proof}

\begin{remark}
	In \cite{BHKY17}*{Example 5.28}, we have claimed that 
	\[
		\Ext^g_{\MHS^{\boxtimes I}_{\bbR}}(\bbR(0_I),\bbR(n_I))
		=\begin{cases}
		(2\pi i)^{gn}\bbR &  n> 0\\
		\{0\}  & n\leq 0,
		\end{cases}
	\]
	but the calculation there is incorrect.  The statement in Lemma \ref{lem: ext2} is the correct calculation.
	We apologize for any inconvenience this may have caused.
\end{remark}

The underlying conjectural philosophy of the plectic theory of \Nekovar and Scholl \cite{NS16} 
states that in the presence of real multiplication, the associated motive has an additional structure 
called a plectic structure. In light of this philosophy, there should be a general theory of 
variations of mixed plectic $\bbR$-Hodge structures and a natural way to equip 
their cohomology groups with plectic $\bbR$-Hodge structures. 
Although we do not have such a general theory at present, 
we can guess, in some special cases, the plectic $\bbR$-Hodge structures by computing the cohomology groups concretely. 

Let us recall the isomorphism $\bbR(\bone)\cong H_\fra$ given in \eqref{eq: R(1) to H_fra}. 
By taking the dual, we have an isomorphism of $\bbR$-Hodge structures 
\begin{equation}\label{eq: H^1(T) plectic}
	H^1(\bbT^\fra,\bbR)\isomto \bbR(-\bone)=\bigoplus_{\tau\in I}\bbR(-1_\tau), 
\end{equation}
which maps the basis $(2\pi i)^{-1}\dlog t_\tau$ to the standard basis. 
We assume that $H^1(\bbT^\fra,\bbR)$ has a natural plectic $\bbR$-Hodge structure 
and the isomorphism \eqref{eq: H^1(T) plectic} preserves it, 
viewing $\bbR(-1_\tau)$ as the plectic Tate object 
$\bbR(0)\boxtimes\cdots\boxtimes\bbR(-1)\boxtimes\cdots\boxtimes\bbR(0)$ 
with $\bbR(-1)$ in the $\tau$-component.

We also consider the cohomology with coefficients in the logarithm sheaf $\bbLog^N$ in a similar way. 
Let us recall the isomorphisms in Corollary \ref{cor: splits} and Proposition \ref{prop: intermediate}: 
For $N>0$ divisible by $g$, 
\begin{equation}\label{eq: H^m(bbLog)}
H^m(U/\Ft,\bbLog^N)\cong \begin{cases}
    \bigoplus_{\Cl^+_F}\bbR(N)^{\oplus \binom{g-1}{m}} & (0\le m<g),\\
    \bigoplus_{\Cl^+_F}\bbR(-g)^{\oplus \binom{g-1}{m-g}} & (g\le m<2g-1),\\
    \bigoplus_{\Cl^+_F}\Bigl(\bbR(-g)
	\oplus\prod_{n=0}^{N/g}\bbR((n-1)g)\Bigr) & (m=2g-1).
\end{cases}
\end{equation}
Again we assume that $H^m(U/\Ft,\bbLog^N)$ has a natural plectic $\bbR$-Hodge structure 
and \eqref{eq: H^m(bbLog)} is refined as an isomorphism of plectic $\bbR$-Hodge structures 
\begin{equation}\label{eq: H^m(bbLog) plectic}
H^m(U/\Ft,\bbLog^N)\cong \begin{cases}
    \bigoplus_{\Cl^+_F}\bbR((N/g)_I)^{\oplus \binom{g-1}{m}} & (0\le m<g),\\
    \bigoplus_{\Cl^+_F}\bbR(-1_I)^{\oplus \binom{g-1}{m-g}} & (g\le m<2g-1),\\
    \bigoplus_{\Cl^+_F}\Bigl(\bbR(-1_I)
	\oplus\prod_{n=0}^{N/g}\bbR((n-1)_I)\Bigr) & (m=2g-1).
\end{cases}
\end{equation}

The plectic philosophy also suggests the existence of the theory of plectic Deligne--Beilinson cohomology 
with coefficients in variations of mixed plectic $\bbR$-Hodge structures, 
though we have not yet established such a theory. 
In this paper, we assume that the equivariant plectic Deligne--Beilinson cohomology 
$H^m_{\sD^I}(U/\Ft,\bbLog^N)$ of $U/\Ft$ with coefficients in $\bbLog^N$ exists 
and fits into the spectral sequence 
\begin{equation}\label{eq: ss2}
	E_2^{p,q}=\Ext^p_{\MHS^{\boxtimes I}_{\bbR}}(\bbR(0_I),H^q(U/\Ft,\bbLog^N))
	\Rightarrow H^{p+q}_{\sD^I}(U/\Ft,\bbLog^N).
\end{equation}

In summary, we assume the following: 

\begin{hypothesis}\label{hypo:plectic}\ 
\begin{itemize}
\item 
For each $\fra\in\frI$, $H^1(\bbT^\fra,\bbR)$ is equipped with a natural plectic $\bbR$-Hodge structure 
and \eqref{eq: H^1(T) plectic} is an isomorphism in $\MHS^{\boxtimes I}_\bbR$. 
\item 
For each $m\in\bbZ$ and $N\in\bbN$, $H^m(U/\Ft,\bbLog^N)$ is equipped with 
a natural plectic $\bbR$-Hodge structure and \eqref{eq: H^m(bbLog) plectic} is an isomorphism in $\MHS^{\boxtimes I}_\bbR$. 
\item 
The equivariant plectic Deligne--Beilinson cohomology $H^m_{\sD^I}(U/\Ft,\bbLog^N)$ exists 
and fits into the spectral sequence \eqref{eq: ss2}. 
\end{itemize}
\end{hypothesis}

\begin{proposition}\label{prop: assumption}
    We have a canonical isomorphism
	\begin{equation}\label{eq: pcan}
		 H^{2g-1}_{\sD^I}(U/\Ft,\bbLog)\cong \bigoplus_{\Cl^+_F}\bbR,
	\end{equation}
	where $ H^{2g-1}_{\sD^I}(U/\Ft,\bbLog)\coloneqq\varprojlim_N  H^{2g-1}_{\sD^I}(U/\Ft,\bbLog^N)$.
\end{proposition}

\begin{proof}
	Assume that $N>0$ and $g|N$. 
	By \eqref{eq: H^m(bbLog) plectic} and Lemma \ref{lem: ext2}, we have 
    \[
    E_2^{0,2g-1}\cong E_\infty^{0,2g-1}\cong\bigoplus_{\Cl^+_F}\bbR,\qquad 
    E_2^{g,g-1}\cong E_\infty^{g,g-1}\cong\bigoplus_{\Cl^+_F}(2\pi i)^{N-g}\bbR, 
    \]
    and these are the only nonzero terms with $p+q=2g-1$. 
    Thus there is an exact sequence 
    \[
    0\lra E_2^{g,g-1} \lra H^{2g-1}_{\sD^I}(U/\Ft,\bbLog^N) 
    \lra \bigoplus_{\Cl^+_F}\bbR \lra 0. 
    \]
	Moreover, the transition map $H^{g-1}(U/\Ft,\bbLog^{N+g})\to H^{g-1}(U/\Ft,\bbLog^{N})$ is zero 
	by \eqref{eq: H^m(bbLog) plectic} and Lemma \ref{lem: ext2}. 
	Hence the corresponding transition map on $E_2^{g,g-1}$ is also zero, 
	and Proposition follows by taking the projective limit. 
\end{proof}

\begin{definition}
	We define the \textit{plectic polylogarithm class} $\pol$ to be the element 
	\[
		\pol\in H^{2g-1}_{\sD^I}(U/\Ft,\bbLog)
	\]
	which maps to $(1,\ldots,1)\in \bigoplus_{\Cl^+_F}\bbR$ through the isomorphism \eqref{eq: pcan}.
\end{definition}

\begin{remark}\label{rem: same}
The diagonal embedding $\cG\to\cG^I$ of the Tannakian fundamental groups induces a functor $\MHS^{\boxtimes I}_\bbR\to\MHS_\bbR$. 
This should give a commutative diagram 
\[\xymatrix{
H^{2g-1}_{\sD^I}(U/\Ft,\bbLog) \ar[d] \ar[r]^-{\cong} 
& \Hom_{\MHS^{\boxtimes I}_{\bbR}}(\bbR(0_I),H^{2g-1}(U/\Ft,\bbLog)) \ar[d] \ar[r]^-{\cong} 
& \bigoplus_{\Cl^+_F}\bbR \ar@{=}[d] \\
H^{2g-1}_{\sD}(U/\Ft,\bbLog) \ar[r]^-{\cong} 
& \Hom_{\MHS_{\bbR}}(\bbR(0),H^{2g-1}(U/\Ft,\bbLog)) \ar[r]^-{\cong} 
& \bigoplus_{\Cl^+_F}\bbR 
}\]
from which we see that the vertical arrows are isomorphisms, 
and the plectic polylogarithm class $\pol\in H^{2g-1}_{\sD^I}(U/\Ft,\bbLog)$ 
maps to the polylogarithm class $\pol\in H^{2g-1}_{\sD}(U/\Ft,\bbLog)$. 
\end{remark}

%
%
\subsection{Specializations of the Polylogarithm and Lerch Zeta Values}\label{subsec: Beilinson}
%
%

The advantage of considering the plectic Deligne--Beilinson cohomology over the usual 
Deligne--Beilinson cohomology is that the equivariant plectic Deligne--Beilinson cohomology of degree $2g-1$ 
does not vanish on torsion points even if $g>1$.   Hence we are able to consider the specializations of 
the plectic polylogarithm class to torsion points of the algebraic torus. 

For any torsion point $\xi\in\bbT(\bbC)$, we denote by $\xi\Delta$ the $\Delta$-orbit of $\xi$ 
viewed as a zero-dimensional scheme over $\bbC$, which is isomorphic to a finite disjoint union of $\Spec\bbC$. 
We endow $\xi\Delta$ with a natural $\Delta$-action. 
Then we define the notion of an equivariant variation of mixed plectic $\bbR$-Hodge structures on $\xi\Delta$ as follows. 
Although we do not have a general theory of equivariant variations of mixed plectic $\bbR$-Hodge structures, 
we believe that this is the right definition in this case. 

\begin{definition}
    A variation of mixed plectic $\bbR$-Hodge structures on $\xi\Delta$ is 
    a collection $V=(V_\eta)_{\eta\in\xi\Delta}$ of mixed plectic $\bbR$-Hodge structues 
    indexed by the points of $\xi\Delta$. 
    A $\Delta$-equivariant structure on such $V$ is a collection of isomorphisms 
    $\iota_\eps\colon V_{\eta^\eps}\isomto V_\eta$ indexed by $\eta\in\xi\Delta$ and $\eps\in\Delta$ 
    satisfying $\iota_1=\id$ and $\iota_\eps\circ\iota_{\eps'}=\iota_{\eps\eps'}$. 
\end{definition}

\begin{remark}
    We can also define the notions of equivariant variations of mixed plectic $\bbR$-Hodge structures 
    on the $\Ft$-scheme $\xi\Ft$ and the $\Delta_\xi$-scheme $\xi$ in the same way. 
    We easily see that these three notions are equivalent through the equivariant maps $\xi\to\xi\Delta\to\xi\Ft$. 
\end{remark}

Note that a $\Delta$-equivariant variation of mixed plectic $\bbR$-Hodge structures $V$ on $\xi\Delta$ 
can be regarded as a $\Delta$-equivariant sheaf, whose global section is given by 
\[
    \Gamma(\xi\Delta/\Delta,V)=\Biggl(\bigoplus_{\eta\in\xi\Delta}V_\eta\Biggr)^\Delta
    \cong V_\xi^{\Delta_\xi}. 
\]
Since the right hand side can be written as 
\[
    V_\xi^{\Delta_\xi}=\Ker\Biggl(\bigl(\iota_{\eps_i}-\id\bigr)_i\colon V_\xi\lra \bigoplus_{i=1}^{g-1}V_\xi\Biggr)
\]
by choosing a set of generators $\{\eps_1,\ldots,\eps_{g-1}\}$ of $\Delta_\xi$, 
it has a natural mixed plectic $\bbR$-Hodge structure. More generally, we have the following: 

\begin{proposition}
Let $V$ be a $\Delta$-equivariant variation of mixed plectic $\bbR$-Hodge structures on $\xi\Delta$. 
Then, for any $q\ge 0$, the cohomology $H^q(\xi\Delta/\Delta,V)$ with coefficients in the associated sheaf 
has a natural mixed plectic $\bbR$-Hodge structure. 
\end{proposition}
\begin{proof}
Similarly to the proof of Proposition \ref{prop: complex}, we can compute the cohomology as 
\[
H^q(\xi\Delta/\Delta,V)\cong H^q(\Delta_\xi,V_\xi)\cong H^q(\Hom_{\Delta_\xi}(K_\bullet,V_\xi)). 
\]
Since $\Hom_{\Delta_\xi}(K_\bullet,V_\xi)$ can be regarded as 
a complex in the category of mixed plectic $\bbR$-Hodge structures, 
its cohomology carries a mixed plectic $\bbR$-Hodge structure. 
\end{proof}

Next we define the equivariant plectic Deligne--Beilinson cohomology on $\xi\Delta$ by 
\[H^m_{\sD^I}(\xi\Delta/\Delta,V)
\coloneqq H^m\bigl(R\Hom_{\MHS^{\boxtimes I}_\bbR}\bigl(\bbR(0_I),R\Gamma(\xi\Delta/\Delta,V)\bigr)\bigr).\]
Then we have the spectral sequence 
\begin{equation}\label{eq: ss3}
	E_2^{p,q}=\Ext^p_{\MHS^{\boxtimes I}_{\bbR}}(\bbR(0_I),H^q(\xi\Delta/\Delta,V)) 
	\Rightarrow H^{p+q}_{\sD^I}(\xi\Delta/\Delta,V). 
\end{equation}

Now we want to consider the pull-back $V=i_{\xi\Delta}^*\bbLog^N$ of the logarithm sheaf 
along the inclusion $i_{\xi\Delta}\colon\xi\Delta\rightarrow U$. 
We make the following hypotheses on its plectic structure: 

\begin{hypothesis}\label{hypo: specialization}\ 
\begin{itemize}
\item The sheaf $V=i_{\xi\Delta}^*\bbLog^N$ on $\xi\Delta$ naturally becomes 
a $\Delta$-equivariant variation of mixed plectic $\bbR$-Hodge structures. 
Moreover, for each $\eta\in\xi\Delta$, the splitting principle \eqref{eq: splitting} is 
refined to be an isomorphism 
\[
	(i_{\xi\Delta}^*\bbLog^N)_\eta=i^*_{\eta}\bbLog^N\cong\prod_{k=0}^N\Sym^k\bbR(\bone)
\]
in $\MHS_\bbR^{\boxtimes I}$. 
\item The equivariant map $i_{\xi\Delta}\colon\xi\Delta\to U$ induces a homomorphism on 
the plectic Deligne--Beilinson cohomology 
\[
	i^*_{\xi\Delta}\colon H^m_{\sD^I}(U/\Ft,\bbLog^N) \rightarrow 
	H^m_{\sD^I}(\xi\Delta/\Delta,i^*_{\xi\Delta}\bbLog^N). 
\]
\end{itemize}
\end{hypothesis}

The first hypothesis implies the following: 

\begin{lemma}\label{lem: splitting plectic cohomology}
	For any torsion point $\xi\in\bbT(\bbC)$, we have an isomorphism 
	\[
		H^{2g-1}_{\sD^I}(\xi\Delta/\Delta,i^*_{\xi\Delta}\bbLog)\cong
		\prod_{n=1}^\infty(2\pi i)^{(n-1)g}\bbR.
	\]
\end{lemma}

\begin{proof}
	By applying Lemma \ref{lem: Delta} to $\Delta_\xi$ instead of $\Delta$, we have 
    \[
        H^q(\xi\Delta/\Delta,i^*_{\xi\Delta}\bbLog^N)
        \cong\prod_{k=0}^N H^q(\Delta_\xi,\Sym^k\bbR(\bone))
        \cong\prod_{n=0}^{\lfloor N/g\rfloor} \bbR(n_I)^{\binom{g-1}{q}}. 
    \]
    Then Lemma \ref{lem: ext2} and the spectral sequence \eqref{eq: ss3} implies that 
    \[
        E_\infty^{g,g-1}\cong E_2^{g,g-1}\cong \prod_{n=1}^{\lfloor N/g\rfloor}(2\pi i)^{(n-1)g}\bbR 
    \]
    and this is the only non-zero term with $p+q=2g-1$. 
	Thus our assertion follows.
\end{proof}

Inspired by the relation between the Shintani generating class and the 
Lerch zeta values given in Theorem \ref{thm: generate}, 
as well as the classical case when $U=\bbP^1\setminus\{0,1,\infty\}$, 
we conjecture that the specialization of the polylogarithm $\pol$ at $\xi\Delta/\Delta$ gives 
the Lerch zeta values at positive integers, as follows. 

\begin{conjecture}\label{conj: i^*pol}
	Let $\xi$ be a torsion point in $U$, and we let $\xi\Delta$ be the orbit of $\xi$. 
	Then 
	$
		i_{\xi\Delta}^*\pol\in
		H^{2g-1}_{\sD^I}(\xi\Delta/\Delta,i_{\xi\Delta}^*\bbLog)=\prod_{n=1}^\infty(2\pi i)^{(n-1)g}\bbR
	$
	satisfies
	\[
		i_{\xi\Delta}^*\pol=(d_F^{1/2}\cL^\infty(\xi\Delta,n))_{n=1}^\infty\in\prod_{n=1}^\infty(2\pi i)^{(n-1)g}\bbR, 
	\]
	where $d_F$ denotes the discriminant of $F$ and 
	\begin{align*}
		\cL^\infty(\xi\Delta,n)
		&\coloneqq\begin{cases}
		     \Re \cL(\xi\Delta,n) & \text{if $(n-1)g$ is even}, \\
		     i\,\Im \cL(\xi\Delta,n) & \text{if $(n-1)g$ is odd}
		\end{cases}\\
		&=\frac{1}{2}\Bigl(\cL(\xi\Delta,n)+(-1)^{(n-1)g}\overline{\cL(\xi\Delta,n)}\Bigr)
		\in (2\pi i)^{(n-1)g}\bbR. 
	\end{align*}
\end{conjecture}

\begin{remark}
\begin{enumerate}
\item 
    The appearance of $d_F^{1/2}$ is motivated by Conjecture \ref{conj: Ler} below. 
    Note also that Zagier's conjecture \cite{Zag89} on the Dedekind zeta values contains the factor $d_F^{1/2}$. 
\item 
    For the case $g=1$ and $\fra=\bbZ$, we have $\bbT^\bbZ=\bbG_m$, 
	and a nontrivial torsion point $\xi\in\bbT^\bbZ(\bbC)$ is given by a root of unity.  
	Then $\Delta=\{1\}$, and we have
	\[
		\cL(\xi,n)=\Li_n(\xi)
	\]
	for any integer $n>0$,
	where $\Li_n(\xi)$ is the value at $\xi$ of the polylogarithm function $\Li_n(t)$.
	Noting that
	\[
		\cL^\infty(\xi,n)\equiv\Li_n(\xi)\mod{(2\pi i)^n\bbR}
	\]
	in $\bbC/(2\pi i)^n\bbR\cong(2\pi i)^{n-1}\bbR$, we see that
	our conjecture in this case is the theorem of Beilinson and Deligne
	(see for example \cite{HW98}*{Proposition V2.1} or \cite{BHY18}*{Corollary 5.10}).
\end{enumerate}
\end{remark}

%
%
\subsection{Lerch versus Hecke}\label{subsec: Lerch vs Hecke}
%
%

In the rest of this paper, based on Conjecture \ref{conj: i^*pol}, 
we will establish a ``plectic framework'' for attacking Beilinson's conjectures 
for some types of Hecke $L$-functions over totally real fields. 
In this subsection, we give a relation between Lerch zeta functions introduced in Section \ref{sec: shintani} 
and Hecke $L$-functions. 

\subsubsection{Gauss Sums}
First we recall the definition of Gauss sums discussed in \cite{BHY20}, and give some standard properties. 

Let $\frg$ be a nonzero integral ideal of $\mathcal{O}_F$, and 
we denote by $\Cl_F^+(\mathfrak{g})$ the narrow ray class group. 
For $\fra\in\frI$, we denote by $\bbT^\fra[\frg]=\Hom(\fra/\frg\fra,\bbG_m)$ 
the subgroup scheme of $\bbT^\fra$ consisting of $\frg$-torsion points, 
and set $\bbT[\frg]=\coprod_{\fra\in\frI}\bbT^\fra[\frg]$. 
Note that $\bbT[\frg]$ is stable under the isomorphism $\pair{x}\colon\bbT\isomto\bbT$ for each $x\in \Ft$. 

The torsion point $\xi\in\mathbb{T}^\fra[\frg](\mathbb{C})$ is called \emph{primitive} 
if $\xi\notin\mathbb{T}^\fra[\frg'](\mathbb{C})$ for any $\frg'\supsetneq\frg$. 
We denote by $\mathbb{T}^\fra_0[\frg](\mathbb{C})$ the subset of primitive torsion points. As is easily proved, 
the set $\mathbb{T}_0[\frg](\mathbb{C}):=\coprod_{\fra\in\frI}\bbT^\fra_0[\frg](\bbC)$ is stable 
under the action of $\Ft$. 

We define $\sT$, $\sT[\frg]$, and $\sT_0[\frg]$ to be the quotient set $\bbT(\bbC)/\Ft$, $\bbT[\frg](\bbC)/\Ft$, 
and $\bbT_0[\frg](\bbC)/\Ft$, respectively. 
As in \cite{BHY20}*{\S5.3}, we define an action of $\Cl_F^+(\frg)$ on the set $\sT[\frg]$ as follows: 
For an integral ideal $\frb$ and $\xi\in\bbT^\fra(\bbC)$, 
we define $\xi^\frb\in\mathbb{T}^{\frb\fra}(\mathbb{C})$ to be the composite 
\[
\frb\fra \hooklongrightarrow \fra \overset{\xi}{\longrightarrow} \mathbb{C}^\times. 
\]
The map $(-)^\frb\colon\bbT(\bbC)\to\bbT(\bbC)$ is $\Ft$-equivariant. 
Moreover, as shown in \cite{BHY20}*{Lemma 5.11, Lemma 5.12}, 
it induces a well-defined action $(\xi F_+^\times, [\frb])\mapsto \xi^{\frb}F_+^\times$ of $\Cl_F^+(\frg)$ on $\sT[\frg]$ 
and the subset $\sT_0[\frg]\subset\sT[\frg]$ becomes a $\Cl_F^+(\frg)$-torsor. 

\begin{definition}
Let $\psi:\Cl_F^+(\frg)\rightarrow\mathbb{C}^\times$ be a Hecke character and $\xi\in\bbT^\fra[\frg](\bbC)$. 
Then we define the Gauss sum $g(\psi, \xi)$ to be
\[
g(\psi, \xi)\coloneqq\sum_{\alpha\in\fra/\frg\fra}\psi_\fra(\alpha)\xi(-\alpha).
\]
\end{definition}

Here, for $\alpha\in\fra/\frg\fra$, we define
\[
\psi_\fra(\alpha)\coloneqq\begin{cases}
\psi([\widetilde{\alpha}\fra^{-1}]) & \text{if $\alpha$ generates $\fra/\frg\fra$ as an $\mathcal{O}_F/\fra$-module,}\\
0 & \text{otherwise,}
\end{cases}
\]
where $\widetilde{\alpha}$ is any element of $\fra_+$ lifting $\alpha$ (cf.~\cite{BHY20}*{Definition 5.9}). 
Note that the multiplicativity $\psi_{\fra\frb}(\alpha\beta)=\psi_\fra(\alpha)\psi_\frb(\beta)$ holds 
for any $\alpha\in\fra$ and $\beta\in\frb$ (\cite{BHY20}*{Eq.~(7)}). 

Note that, since this sum satisfies $g(\psi, \xi) = g(\psi, \pair{x}\xi)$ for any $x\in \Ft$ \cite{BHY20}*{Lemma 2.6}, 
the sum $g(\psi, \eta)$ makes sense for any element $\eta\in\sT[\frg]$. 
In addition, as in the classical case, the Gauss sum has the following properties:

\begin{proposition}\label{prop: The property of Gauss sums}
Let $\psi\colon\Cl_F^+(\frg)\rightarrow\bbC^\times$ be a \emph{primitive} Hecke character.
\begin{enumerate}
\item For $\xi\in\mathbb{T}_0^{\fra}[\frg](\mathbb{C})$ and $\beta \in \frb/\frg \frb$, we have 
\[g(\psi,\xi)\overline{\psi}_\frb(\beta)
=\sum_{\gamma\in \fra\frb^{-1}/\frg\fra\frb^{-1}} \psi_{\fra\frb^{-1}}(\gamma) \xi(-\beta \gamma). \]

\item For $\eta\in\sT_0[\frg]$, we have the equality
\[
g(\psi, \eta) g(\overline{\psi}, \eta) = \psi_{\cO_F}(-1)\Nr\frg. 
\]
\end{enumerate}
\end{proposition}

\begin{proof}
(1) First we assume that $\beta$ generates $\frb/\frg\frb$ as an $\cO_F/\frg$-module. 
Then we may write the Gauss sum as 
\[g(\psi,\xi)=\sum_{\gamma\in \fra\frb^{-1}/\frg\fra\frb^{-1}} 
\psi_{\fra}(\beta\gamma) \xi(-\beta \gamma)
=\psi_\frb(\beta)\sum_{\gamma\in \fra\frb^{-1}/\frg\fra\frb^{-1}} 
\psi_{\fra\frb^{-1}}(\gamma) \xi(-\beta \gamma). \]
Thus the result follows. 
On the other hand, if $\beta$ is not a generator of $\frb/\frg\frb$, the character 
\[\fra\frb^{-1}/\frg\fra\frb^{-1}\lra\bbC^\times;\ \gamma\longmapsto \xi(\beta\gamma)\]
is not primitive in $\bbT^{\fra\frb^{-1}}[\frg](\bbC)$. 
Hence the right hand side is zero by \cite{BHY20}*{Proposition 2.8}, 
and the result follows. 

(2) We take a representative $\xi\in\mathbb{T}_0^{\fra}[\frg](\mathbb{C})$ of $\eta$. 
Then, by (1), we have 
\begin{align*}
g(\psi, \xi)g(\overline{\psi}, \xi) 
&=\sum_{\beta\in\fra/\frg\fra}g(\psi,\xi)\overline{\psi}_{\fra}(\beta)\xi(-\beta) 
=\sum_{\beta\in\fra/\frg\fra} \sum_{\gamma \in \cO_F/\frg}\psi_{\cO_F}(\gamma)\xi(-\beta\gamma)\xi(-\beta)\\
&=\sum_{\gamma \in \cO_F/\frg} \psi_{\cO_F}(\gamma) \sum_{\beta \in \fra/\frg\fra} \xi(-\beta(\gamma+1))
=\psi_{\cO_F}(-1) \Nr\frg. 
\end{align*}
The last equality holds since $\xi$ is primitive. 
\end{proof}

\subsubsection{Functional Equation for Hecke $L$-Functions}\label{subsub:Functional equation}
Next we recall the functional equations for Hecke $L$-functions in the case we will discuss. 
Recall here that $\Cl_F^+(\frg)$ is a quotient of the idele group $\bbA_F^\times$, 
so we can define an element $c'_\tau\in\Cl_F^+(\frg)$ for each embedding $\tau\in I$ to be 
the class of the image of the idele $\tilde{c}_\tau\in\bbA_F^\times$ 
such that $(\tilde{c}_\tau)_v=1$ for all places $v$ except for $\tau$, and $(\tilde{c}_\tau)_\tau=-1$. 

\begin{definition}
Let $\psi\colon\Cl_F^+(\frg)\rightarrow\mathbb{C}^\times$ be a Hecke character.
\begin{enumerate}
\item We denote by $u(\psi)$ the number of elements $\tau\in I$ such that $\psi(c'_\tau)=-1$.

\item We denote by $\mathfrak{D}_F$ and $d_F$ the different and the discriminant of $F$, respectively. 

\item We define an additive character 
$\xi_\mathrm{can}\in\mathbb{T}^{\frg^{-1}\mathfrak{D}_F^{-1}}[\frg](\mathbb{C})$ to be
\[
\xi_{\mathrm{can}}:\frg^{-1}\mathfrak{D}_F^{-1}/\mathfrak{D}_F^{-1}
\xrightarrow{\mathrm{Tr}_{F/\bbQ}} \mathbb{Q}/\mathbb{Z}
\xrightarrow{\exp(-2\pi i\cdot)} \mathbb{C}^\times.
\]
Note that $\xi_\mathrm{can}$ is primitive. 
\item We set
\[
\Lambda(\psi, s)=(d_F\Nr\frg)^{s/2}
\Gamma_{\bbR}(s)^{g-u(\psi)}\Gamma_{\bbR}(s+1)^{u(\psi)}L(\psi, s), 
\]
where $\Gamma_{\bbR}(s)\coloneqq \pi^{-s/2}\Gamma(s/2)$. 
\end{enumerate}
\end{definition}

With these definitions, we have the following functional equation: 

\begin{theorem}\label{thm: Functional equations}
Assume that $\psi$ is primitive. Then we have
\[
\Lambda(\psi, 1-s) = W(\psi)\Lambda(\overline{\psi}, s),
\]
where $W(\psi)\coloneqq i^{-u(\psi)}g(\psi, \xi_{\mathrm{can}})/(\Nr\frg)^{1/2}$.
\end{theorem}

\begin{proof}
Refer, for example, to \cite{Neu99}*{(8.6) Corollary in Chapter VII}.
\end{proof}

\begin{corollary}
Let $\psi:\Cl_F^+(\frg)\rightarrow\mathbb{C}^\times$ be a primitive character and $k>1$ an integer. 
We denote by $L^*(\psi,1-k)$ the leading Taylor coefficient of $L(\psi,s)$ at $s=1-k$. 
Then, 
\begin{enumerate}
\item If $k$ is even, we have
\[
L(\psi, k) = 2^{-g+2u(\psi)}(\Nr\frg)^{1-k}((k-1)!)^{-g}d_F^{\frac{1}{2}-k}
(2\pi i)^{kg-u(\psi)}g(\overline{\psi},\xi_{\mathrm{can}})^{-1}L^\ast(\overline{\psi}, 1-k).
\]
\item If $k$ is odd, we have
\[
L(\psi, k) = 2^{g-2u(\psi)}(\Nr\frg)^{1-k}((k-1)!)^{-g}d_F^{\frac{1}{2}-k}(2\pi i)^{(k-1)g+u(\psi)}g(\overline{\psi},\xi_{\mathrm{can}})^{-1}L^\ast(\overline{\psi}, 1-k).
\]
\end{enumerate}
\end{corollary}

We say that a character $\psi$ is \emph{totally noncritical for $k$} if $k$ is even and $u(\psi)=g$, 
or if $k$ is odd and $u(\psi)=0$. Then, as a particular case, we obtain the following:

\begin{corollary}\label{cor: Corollary of functional equations}
For a primitive character $\psi$ totally noncritical for $k$, we have
\[
  L(\psi, k)\in\mathbb{Q}((k-1)g)\cdot d_F^{\frac{1}{2}}
  g(\overline{\psi}, \xi_{\mathrm{can}})^{-1}L^\ast(\overline{\psi}, 1-k).
\]
\end{corollary}

\subsubsection{Hecke versus Lerch}

Here we express some type of Hecke $L$-functions in terms of Lerch zeta functions. Recall that, for $\fra\in\frI$ and 
a torsion point $\xi\in\mathbb{T}^\fra(\mathbb{C})$, we defined the Lerch zeta function $\cL(\xi\Delta, s)$ 
in Definition \ref{def: Lerch} (cf. \cite{BHY20}*{Definition 2.1}). 
Note that, since $\cL(\xi\Delta, s)=\cL(\pair{x}\xi\Delta, s)$ \cite{BHY20}*{Lemma 2.3}, 
the notion $\cL(\eta, s)$ makes sense for any torsion point $\eta\in\mathscr{T}$.

Moreover, from Lemmas 2.8 and 5.13 of \cite{BHY20}, we can directly prove the following result:

\begin{theorem}\label{thm: Hecke vs Lerch}
For $\eta\in\sT_0[\frg]$ and a primitive character $\psi\colon\Cl_F^+(\frg)\rightarrow\mathbb{C}^\times$, we have
\[
L(\psi, s) =\frac{g(\psi, \eta)}{\Nr\frg}\sum_{\frb\in\Cl_F^+(\frg)}\psi(\frb)^{-1}\cL(\eta^{\frb}, s).
\]
In particular, we have 
\[
\sum_{\frb\in\Cl_F^+(\frg)}\psi(\frb)^{-1}\cL(\xi_\mathrm{can}^{\frb}\Delta, k)
\in \mathbb{Q}\cdot g(\psi, \xi_{\mathrm{can}})^{-1}L(\psi, k).
\]
\end{theorem}

Combining Theorem \ref{thm: Hecke vs Lerch}, Corollary \ref{cor: Corollary of functional equations}, 
and Proposition \ref{prop: The property of Gauss sums}, we obtain the following:
\begin{corollary}\label{cor: Hecke vs Lerch}
If $\psi\colon\Cl_F^+(\frg)\rightarrow\mathbb{C}^\times$ is primitive and totally noncritical for $k$, we have
\[
\sum_{\frb\in\Cl_F^+(\frg)}\psi(\frb)^{-1}\mathcal{L}(\xi_\mathrm{can}^{\frb}\Delta, k)\in \mathbb{Q}((k-1)g)\cdot d_F^{1/2}L^\ast(\overline{\psi}, 1-k).
\]
\end{corollary}

\subsubsection{Primitive versus Imprimitive}
Here we investigate the linear combinations of Lerch zeta functions 
appearing in Theorem \ref{thm: Hecke vs Lerch} for imprimitive Hecke characters. 

Let $\frg_0$ be a nonzero integral ideal, $\frn$ an integral ideal, and set $\frg_1=\frn\frg_0$. 
Then we have a natural surjection $\mathrm{pr}\colon\Cl_F^+(\frg_1) \rightarrow \Cl_F^+(\frg_0)$. 
On the other hand, the homomorphism $(-)^\frn\colon\bbT^\fra(\bbC)\to\bbT^{\frn\fra}(\bbC)$ induces 
\[
(-)^\frn\colon\mathbb{T}^\fra[\frg_1](\mathbb{C})\rightarrow\mathbb{T}^{\frn\fra}[\frg_0](\mathbb{C}).
\]
Note that it induces maps $\sT[\frg_1]\to\sT[\frg_0]$ and $\sT_0[\frg_1]\to\sT_0[\frg_0]$. 

For the moment, we consider the case where $\frn=\frp$ is a prime ideal. 

\begin{theorem}\label{thm: Primitive vs imprimitive}
Let $\frg_0$ be a nonzero integral ideal, $\frp$ a prime ideal, and set $\frg_1=\frp\frg_0$. 
Let $\xi_1$ be an element of $\mathbb{T}_0^\fra[\frg_1](\mathbb{C})$ and $\psi_0$ a character of $\Cl_F^+(\frg_0)$. 
Setting $\xi_0=\xi_1^\frp\in \mathbb{T}_0^{\frp\fra}[\frg_0](\mathbb{C})$ and $\psi_1=\psi_0\circ\mathrm{pr}$, we have 
\[
(\Nr\frp)^{s-1}\sum_{\frb\in\Cl_F^+(\frg_1)}\psi_1(\frb)^{-1}\mathcal{L}(\xi_1^\frb\Delta,s) 
= C(\frp)\sum_{\frb\in\Cl_F^+(\frg_0)}\psi_0(\frb)^{-1}\mathcal{L}(\xi_0^\frb\Delta,s),
\]
where
\[
C(\frp)=\begin{cases}
1 & \text{if $\frp\mid\frg_0$,} \\
1-\psi_0(\frp)^{-1}(\Nr\frp)^{s-1}  &  \text{if $\frp\nmid\frg_0$.}
\end{cases}
\]
\end{theorem}

Before the proof of Theorem \ref{thm: Primitive vs imprimitive}, we prove two lemmas.

\begin{lemma}\label{lem: The character sum}
Let $\fra$ be a nonzero fractional ideal and $\frp$ a prime ideal. 
For a torsion point $\xi\in\mathbb{T}^{\frp\fra}(\mathbb{C})$ and $\alpha\in\fra$, we have
\[
\frac{1}{\Nr\frp}\sum_{\zeta\Delta\mapsto\xi\Delta}\zeta\Delta(\alpha) =
\begin{cases}
\xi\Delta(\alpha) & \text{if $\alpha\in\frp\fra$}, \\
0 & \text{if $\alpha\notin\frp\fra$},
\end{cases}
\]
where the sum is taken over all $\zeta\Delta\in\mathbb{T}^\fra(\mathbb{C})/\Delta$ such that $\zeta^\frp\Delta=\xi\Delta$.
\end{lemma}
\begin{proof}
By definition we have
\[
\sum_{\zeta\Delta\mapsto\xi\Delta}\zeta\Delta(\alpha) 
= \sum_{\kappa\in\xi\Delta}\sum_{\zeta^\frp=\kappa}\zeta(\alpha).
\]
If $\alpha\in\frp\fra$, we have $\zeta^\frp(\alpha)=\zeta(\alpha)$ for every $\zeta\in\mathbb{T}^\fra(\mathbb{C})$, 
and for every $\kappa\in\mathbb{T}^{\frp\fra}(\mathbb{C})$, 
the number of elements in $\zeta\in\mathbb{T}^\fra(\mathbb{C})$ such that $\zeta^\frp=\kappa$ is $\Nr\frp$. 
Therefore we obtain 
\[
\sum_{\kappa\in\xi\Delta}\sum_{\zeta^\frp=\kappa}\zeta(\alpha) 
= \sum_{\kappa\in\xi\Delta}(\Nr\frp)\kappa(\alpha) = (\Nr\frp)\xi\Delta(\alpha),
\]
as desired.

If $\alpha\notin\frp\fra$, there exists a character $\lambda\in\mathbb{T}^\fra[\frp](\mathbb{C})$ such that $\lambda(\alpha)\neq 1$. 
Then we have 
\[
\sum_{\zeta^\frp=\kappa}\zeta(\alpha) = \sum_{\zeta^\frp=\kappa}(\lambda\zeta)(\alpha) =\lambda(\alpha)\sum_{\zeta^\frp=\kappa}\zeta(\alpha),
\]
which implies that the sum is zero. 
\end{proof}

\begin{lemma}\label{lem: The pullback of primitive sets}
Let $\frg_0$ be a nonzero integral ideal and $\frp$ a prime ideal. Set $\frg_1=\frp\frg_0$. We define 
\[
\mathscr{T}'[\frg_1] \coloneqq \{\zeta\in\mathscr{T}[\frg_1]\mid \zeta^\frp\in\mathscr{T}_0[\frg_0] \},
\]
that is, the inverse image of $\mathscr{T}_0[\frg_0]$ under $(-)^\frp:\mathscr{T}[\frg_1]\rightarrow\mathscr{T}[\frg_0]$.
Then we have
\[
\mathscr{T}'[\frg_1] =\begin{cases}
\mathscr{T}_0[\frg_1] & \text{if $\frp\mid\frg_0$}, \\
\mathscr{T}_0[\frg_1]\amalg \mathscr{T}_0[\frg_0] & \text{if $\frp\nmid\frg_0$}.
\end{cases}
\]
\end{lemma}
\begin{proof}
Take a torsion point $\zeta\in\bbT^\fra[\frg_1](\bbC)$, and regard it as a character $\zeta\colon\fra\to\bbC^\times$ 
satisfying $\frg_1\fra\subset\Ker(\zeta)$. 
Then the maximal fractional ideal contained in $\Ker(\zeta)$ is written as $\frf_1\fra$ 
with an integral ideal $\frf_1\mid\frg_1$, and we have $\zeta\in\bbT^\fra_0[\frf_1]$. 
Similarly, there is an integral ideal $\frf_0\mid\frg_0$ such that $\zeta^\frp\in\bbT^{\frp\fra}_0[\frf_0]$, 
which is characterized by the property that the maximal fractional ideal contained in $\Ker(\zeta^\frp)$ is $\frf_0\frp\fra$. 
Since $\Ker(\zeta^\frp)=\Ker(\zeta)\cap\frp\fra$, we have $\frf_0\frp=\frf_1\cap\frp$. 
In particular, we see that 
\[\frf_0=\frg_0\iff \begin{cases}
\frf_1=\frg_0\frp & \text{if $\frp\mid\frg_0$}, \\
\frf_1=\frg_0 \text{ or }\frg_0\frp & \text{if $\frp\nmid\frg_0$}, 
\end{cases}\]
as desired. 
\end{proof}

\begin{proof}[Proof of Theorem \ref{thm: Primitive vs imprimitive}]
Let $\xi\in\mathbb{T}^{\frp\fra}(\mathbb{C})$ be any torsion point. 
By Lemma \ref{lem: The character sum}, we have (at least for $\Re\,s\gg 0$)
\begin{align*}
  (\Nr\frp)^{s-1}\sum_{\zeta\Delta\mapsto\xi\Delta}\mathcal{L}(\zeta\Delta, s) 
  &= (\Nr\frp)^{s-1}\sum_{\zeta\Delta\mapsto\xi\Delta}\sum_{\alpha\in\Delta\backslash\fra_+}
  (\zeta\Delta)(\alpha)\Nr(\alpha\fra^{-1})^{-s} \\
  &= \sum_{\alpha\in\Delta\backslash\fra_+}\Biggl(\frac{1}{\Nr\frp}\sum_{\zeta\Delta\mapsto\xi\Delta}
  (\zeta\Delta)(\alpha)\Biggr)\Nr(\alpha\fra^{-1}\frp^{-1})^{-s} \\
  &= \sum_{\alpha\in\Delta\backslash(\frp\fra)_+}(\xi\Delta)(\alpha)\Nr(\alpha\fra^{-1}\frp^{-1})^{-s} \\
  &= \mathcal{L}(\xi\Delta, s).
\end{align*}
Hence, by noting that $\mathscr{T}_0[\frg_1]$ is a $\Cl_F^+(\frg_1)$-torsor and 
by using Lemma \ref{lem: The pullback of primitive sets}, if $\frp\mid\frg_0$, we have 
\begin{align*}
(\Nr\frp)^{s-1}\sum_{\frb\in\Cl_F^+(\frg_1)}\psi_1(\frb)^{-1}\mathcal{L}(\xi_1^\frb\Delta,s)
&= (\Nr\frp)^{s-1}\sum_{\frb'\in\Cl_F^+(\frg_0)}\sum_{\frb\in\Cl_F^+(\frg_1), 
\frb\mapsto\frb'}\psi_1(\frb)^{-1}\mathcal{L}(\xi_1^\frb\Delta,s) \\
&= \sum_{\frb'\in\Cl_F^+(\frg_0)}\psi_0(\frb')^{-1}\Biggl(\sum_{\frb\in\Cl_F^+(\frg_1), 
\frb\mapsto\frb'}(\Nr\frp)^{s-1}\mathcal{L}(\xi_1^\frb\Delta,s)\Biggr)\\
&= \sum_{\frb'\in\Cl_F^+(\frg_0)}\psi_0(\frb')^{-1}
\Biggl(\sum_{\zeta\Delta\mapsto\xi_0^{\frb'}\Delta}(\Nr\frp)^{s-1}\mathcal{L}(\zeta\Delta,s)\Biggr)\\
&= \sum_{\frb'\in\Cl_F^+(\frg_0)}\psi_0(\frb')^{-1}\mathcal{L}(\xi_0^{\frb'}\Delta,s).
\end{align*}
Similarly, if $\frp\nmid\frg_0$, we have
\begin{align*}
& (\Nr\frp)^{s-1}\sum_{\frb\in\Cl_F^+(\frg_1)}\psi_1(\frb)^{-1}\mathcal{L}(\xi_1^\frb\Delta,s) \\
&= \sum_{\frb'\in\Cl_F^+(\frg_0)}\psi_0(\frb')^{-1}\Biggl(\sum_{\frb\in\Cl_F^+(\frg_1), 
\frb\mapsto\frb'}(\Nr\frp)^{s-1}\mathcal{L}(\xi_1^\frb\Delta,s)\Biggr)\\
&= \sum_{\frb'\in\Cl_F^+(\frg_0)}\psi_0(\frb')^{-1}\Biggl(\sum_{\zeta\Delta\mapsto\xi_0^{\frb'}\Delta}
(\Nr\frp)^{s-1}\mathcal{L}(\zeta\Delta,s)-(\Nr\frp)^{s-1}\mathcal{L}(\xi_0^{\frp^{-1}\frb'}\Delta, s)\Biggr) \\
&= \sum_{\frb'\in\Cl_F^+(\frg_0)}\psi_0(\frb')^{-1}\mathcal{L}(\xi_0^{\frb'}\Delta,s)- 
(\Nr\frp)^{s-1}\sum_{\frb'\in\Cl_F^+(\frg_0)}\psi_0(\frb')^{-1}\mathcal{L}(\xi_0^{\frp^{-1}\frb'}\Delta, s) \\
&= \bigl(1-\psi_0(\frp)^{-1}(\Nr\frp)^{s-1}\bigr)
\sum_{\frb'\in\Cl_F^+(\frg_0)}\psi_0(\frb')^{-1}\mathcal{L}(\xi_0^{\frb'}\Delta,s).\qedhere
\end{align*}
\end{proof}

\begin{corollary}\label{cor: Primitive vs imprimitive}
Let $\frg$ be a nonzero integral ideal, $\eta\in\mathscr{T}_0[\frg]$, and $\psi$ a character of $\Cl_F^+(\frg)$ of conductor $\frg_0$. We denote by $\psi_0$ the corresponding primitive character of $\Cl_F^+(\frg_0)$, and set $\frn =\frg\frg_0^{-1}$ and $\eta_0 =\eta^\frn$. Then we have 
\[
(\Nr\frn)^{s-1}\sum_{\frb\in\Cl_F^+(\frg)}\psi(\frb)^{-1}\mathcal{L}(\eta^\frb,s) = \prod_{\frp\mid\frg, \frp\nmid\frg_0}(1-\psi_0(\frp)^{-1}(\Nr\frp)^{s-1})\sum_{\frb\in\Cl_F^+(\frg_0)}\psi_0(\frb)^{-1}\mathcal{L}(\eta_0^\frb,s).
\]
In particular, for an integer $k>1$, we have
\[
\sum_{\frb\in\Cl_F^+(\frg)}\psi(\frb)^{-1}\mathcal{L}(\eta^\frb,k) 
\in \mathbb{Q}(\psi)\cdot \sum_{\frb\in\Cl_F^+(\frg_0)}\psi_0(\frb)^{-1}\mathcal{L}(\eta_0^\frb,k), 
\]
where $\bbQ(\psi)$ denotes the field generated by the values of $\psi$. 
\end{corollary}

Now we rewrite the above results in terms of the Artin $L$-function. 
Let $F(\frg)$ be the narrow ray class field and 
\[
  \mathrm{Art}_F\colon\Cl_F^+(\mathfrak{g})\stackrel{\cong}{\longrightarrow} \Gamma\coloneqq\Gal(F(\frg)/F)
\]
be the reciprocity isomorphism which sends the class of each prime ideal $\mathfrak{p}$ prime to $\mathfrak{g}$ 
to the corresponding geometric Frobenius $\mathrm{Frob}_\mathfrak{p}$. 

For a character $\chi:\Gamma\rightarrow\mathbb{C}^\times$, the corresponding Artin $L$-function can be rewritten 
as the Hecke $L$-function, that is, we have the equality
\[
L(\chi, s) = L(\psi_\chi, s) \coloneqq \sum_{\fra\subset\mathcal{O}_F}\psi_\chi(\fra)(\Nr\fra)^{-s},
\]
where $\psi_\chi\colon\Cl_F^+(\frg_0)\rightarrow\mathbb{C}^\times$ is the primitive character of $\Cl_F^+(\frg_0)$ 
for some $\frg_0\mid\frg$ associated with $\chi\circ\mathrm{Art}_F$. 
We say that $\chi$ is \emph{totally noncritical} for $k>1$ if $\psi_\chi$ is so. 

On the other hand, through the isomorphism $\mathrm{Art}_F$, 
we define a simply transitive action of $\Gamma$ on $\mathscr{T}_0[\mathfrak{g}]$, that is, 
we define the action, so that the equality
\begin{equation}\label{eq: Galois action}
\mathrm{Art}_F([\fra])\eta = \eta^{[\fra^{-1}]}
\end{equation}
holds for $\eta\in\mathscr{T}_0[\mathfrak{g}]$ and $[\fra]\in \Cl_F^+(\frg)$. 
Then, by combining Corollaries \ref{cor: Hecke vs Lerch} and \ref{cor: Primitive vs imprimitive}, 
we obtain the following theorem: 

\begin{theorem}\label{thm: Lerch vs Artin}
For $\eta\in\mathscr{T}_0[\frg]$ and a totally noncritical character $\chi\colon\Gamma\rightarrow\mathbb{C}^\times$ for $k$, 
we obtain
\[
\sum_{\gamma\in\Gamma}\chi(\gamma)\mathcal{L}(\gamma(\eta), k)\in\mathbb{Q}(\chi)((k-1)g)\cdot 
d_F^{1/2}L^\ast(\chi^{-1}, 1-k).
\]
\end{theorem}

%
%
\subsection{A Plectic Roadmap for Beilinson's Conjectures}\label{subsec: plectic roadmap}
%
%

In the former half of this subsection, we recall a fragment of Beilinson's conjectures and 
rewrite it in terms of Lerch zeta functions, and in the latter half, we discuss a possible plectic approach to it. 
For the details on Beilinson's conjectures, we refer to the book \cite{RSS88}, 
especially the articles by Schneider \cite{Schneider88} and Neukirch \cite{Neu88} therein. 

Recall that, for an integer $k>1$, Borel defined a canonical homomorphism 
$K_{2k-1}(\mathbb{C})\rightarrow \mathbb{R}(k-1)=\mathbb{R}\,i^{k-1}$. 
By using this, for a number field $L$, we obtain a natural homomorphism 
\[
r_\mathrm{Bo}:K_{2k-1}(L)\longrightarrow 
\bigoplus_{\tau:L\rightarrow\mathbb{C}}K_{2k-1}(\mathbb{C})\longrightarrow 
\bigoplus_{\tau:L\rightarrow\mathbb{C}}\mathbb{R}(k-1).
\]
For this homomorphism, Borel proved the following theorem:
\begin{theorem}\label{thm: Borel}
For $k>1$, the homomorphism $r_{\mathrm{Bo}}$ induces an isomorphism
\[
\mathbb{R}\otimes_\mathbb{Z} K_{2k-1}(L)\stackrel{\cong}{\longrightarrow}\Bigl(\bigoplus_{\tau}\mathbb{R}(k-1)\Bigr)^{+},
\]
where ``$+$'' means the invariant part for the involution 
\begin{equation}\label{eq:involution}
    (a_\tau)_\tau\longmapsto (\overline{a_{\ol\tau}})_\tau =((-1)^{k-1}a_{\ol\tau})_\tau.
\end{equation}
\end{theorem}

In the modern language, the domain and the codomain have cohomological interpretations, that is, we have 
\[
\mathbb{Q}\otimes_\mathbb{Z} K_{2k-1}(L)=H^1_{\mathscr{M}}(\Spec L,\mathbb{Q}(k))_{\mathrm{fin}},
\]
the finite part of the motivic cohomology, and
\[
\Bigl(\bigoplus_{\tau}\mathbb{R}(k-1)\Bigr)^{+} = H^1_{\mathscr{D}}((\Spec L)_{/\mathbb{R}},\mathbb{R}(k)),
\]
the Deligne cohomology over $\mathbb{R}$. 
Let us recall, in passing, the Deligne cohomology over $\mathbb{C}$ is given by 
\[
H^1_{\mathscr{D}}((\Spec L)_{/\mathbb{C}},\mathbb{R}(k))=\bigoplus_{\tau}\mathbb{R}(k-1).
\]

Note that Theorem \ref{thm: Borel} gives a $\mathbb{Q}$-structure on the Deligne cohomology over $\bbR$. 
On the other hand, the Deligne cohomology has another $\mathbb{Q}$-structure induced from ``Betti-de Rham structure''. 
In our case, it can be explicitly written as follows: 
\[
B^+_L = \Bigl(\bigoplus_{\tau}\mathbb{Q}(k-1)\Bigr)^{+} \subset H^1_{\mathscr{D}}((\Spec L)_{/\mathbb{R}},\mathbb{R}(k)).
\]
Here $B_L^+$ denotes the invariant part of $B_L \coloneqq \bigoplus_{\tau}\mathbb{Q}(k-1) = \bigoplus_{\tau}\mathbb{Q}\cdot(2\pi i)^{k-1}$ 
with respect to the involution \eqref{eq:involution}. 

Now let $L/K$ be an abelian extension of number fields and $G=\Gal(L/K)$. As is easily checked, 
the isomorphism $r_\mathrm{Bo}$ is $G$-equivariant with $\gamma\in G$ acting from the left on the codomain 
as $(a_\tau)_\tau\mapsto (a_{\tau\gamma})_\tau$. Thus we obtain an $\bbR[G]$-isomorphism 
\begin{equation}\label{eq: Borel}
\mathrm{det}_{\mathbb{R}[G]}(\mathbb{R}\otimes_\mathbb{Q} H^1_{\mathscr{M}}(\Spec L,\mathbb{Q}(k))_\mathrm{fin}) \cong
\mathrm{det}_{\mathbb{R}[G]}(H^1_{\mathscr{D}}((\Spec L)_{/\mathbb{R}},\mathbb{R}(k))).
\end{equation}

The following conjecture is known as a part of (the equivariant version of) Beilinson's conjectures:
\begin{conjecture}\label{conj: equiv Beilinson conjectures}
Under the isomorphism (\ref{eq: Borel}), the equality
\[
\det_{\mathbb{Q}[G]}(H^1_{\mathscr{M}}(\Spec L,\mathbb{Q}(k))_\mathrm{fin}) = L^\ast(1-k)\bigl(\mathrm{det}_{\mathbb{Q}[G]}B^+_L\bigr) 
\]
holds. Here $L^\ast(1-k)\in\mathbb{R}[G]$ is the element which maps to $(L^\ast(\chi^{-1}, 1-k))_{\chi}$ under the map
\[
\mathbb{R}[G]\hookrightarrow \mathbb{C}[G]\stackrel{\cong}{\longrightarrow} \prod_{\chi:G\rightarrow\mathbb{C}^\times}\mathbb{C};\quad 
[\gamma]\mapsto (\chi(\gamma))_\chi. 
\]
\end{conjecture}

Next we rewrite this conjecture in an inequivariant way. 

\begin{definition}
Let $\chi\colon G\rightarrow\mathbb{C}^\times$ be a character. 
For a $\mathbb{Q}[G]$-module $V$, we set $V_\chi\coloneqq\mathbb{Q}(\chi)\otimes_{\mathbb{Q}[G]}V$, 
where $\bbQ[G]\to\bbQ(\chi)$ is the natural ring homomorphism induced by $\chi$, which will be also denoted by $\chi$. 
\end{definition}

Note that we have a natural isomorphism 
\[
\mathbb{C}\otimes_{\mathbb{Q}(\chi)}V_{\chi} =\mathbb{C}\otimes_{\chi, \mathbb{Q}[G]}V 
\cong \{v\in\mathbb{C}\otimes_\mathbb{Q} V \mid [\gamma](v)=\chi(\gamma)v \text{ for all $\gamma\in G$}\},
\]
under which $1\otimes v$ is mapped to $\displaystyle \frac{1}{\#G}\sum_{\gamma\in G}\chi^{-1}(\gamma)[\gamma](v)$. 
In other words, $V_\chi$ gives a $\mathbb{Q}(\chi)$-structure of $\mathbb{C}\otimes_{\chi, \mathbb{Q}[G]}V$ and 
of the ``$\chi$-part'' of $\mathbb{C}\otimes_\mathbb{Q} V$.
For instance, the one-dimensional $\mathbb{Q}(\chi)$-vector space $\bigl(\mathrm{det}_{\mathbb{Q}[G]}B_{L}\bigr)_\chi$ gives a $\mathbb{Q}(\chi)$-structure of the one-dimensional $\mathbb{C}$-vector space
\[
\mathbb{C}\otimes_{\chi, \mathbb{Q}[G]}\mathrm{det}_{\mathbb{Q}[G]}B_{L} =
\mathbb{C}\otimes_{\chi, \mathbb{R}[G]}\mathrm{det}_{\mathbb{R}[G]}\Bigl(H^1_{\mathscr{D}}((\Spec L)_{/\mathbb{C}}, \mathbb{R}(k))\Bigr).
\]
Similarly, $\bigl(\mathrm{det}_{\mathbb{Q}[G]}B^+_{L}\bigr)_\chi$ gives a $\mathbb{Q}(\chi)$-structure of $\mathbb{C}\otimes_{\chi, \mathbb{R}[G]}\mathrm{det}_{\mathbb{R}[G]}\Bigl(H^1_{\mathscr{D}}((\Spec L)_{/\mathbb{R}}, \mathbb{R}(k))\Bigr)$.

Conjecture \ref{conj: equiv Beilinson conjectures} is equivalent to the following:

\begin{conjecture}\label{conj: Beilinson conjectures}
For any character $\chi:G\rightarrow\mathbb{C}^\times$, 
 under the isomorphism
\[
\bbC\otimes_{\chi,\mathbb{Q}[G]}\det_{\mathbb{Q}[G]}(H^1_{\mathscr{M}}(\Spec L,\mathbb{Q}(k))_{\mathrm{fin}}) \cong
\bbC\otimes_{\chi,\mathbb{R}[G]}\det_{\mathbb{R}[G]}(H^1_{\mathscr{D}}((\Spec L)_{/\mathbb{R}},\mathbb{R}(k))),
\]
the equality
\[
\bigl(\det_{\mathbb{Q}[G]}(H^1_{\mathscr{M}}(\Spec L,\mathbb{Q}(k))_{\mathrm{fin}})\bigr)_\chi 
= L^\ast(\chi^{-1},1-k)\bigl(\det_{\mathbb{Q}[G]}B^+_L\bigr)_\chi
\]
holds.
\end{conjecture}

In the following, we will interpret this conjecture in terms of Lerch zeta functions in the case where $G=\Gamma\coloneqq\Gal(F(\frg)/F)$.

First we describe the determinant module of Deligne cohomology more explicitly. 
For each $\tau\in I=\Hom(F,\mathbb{C})$, we fix its lifting $\tilde{\tau}\colon F(\mathfrak{g})\rightarrow\mathbb{C}$, 
and denote by $c_\tau\in\Gamma$ the ``complex conjugate'', that is, the homomorphism making the diagram 
\[
\xymatrix{
F(\mathfrak{g}) \ar[r]^-{\tilde{\tau}} \ar[d]_{c_\tau}& \mathbb{C} \ar[d]^{\text{conj.}} \\
F(\mathfrak{g}) \ar[r]^-{\tilde{\tau}} & \mathbb{C}
}
\]
commute. Note that, since $\Gamma$ is abelian, $c_\tau$ is independent of the choice of the lifting $\tilde{\tau}$. 
Note also that $c'_\tau\in \Cl_F^+(\frg)$ (see the first paragraph of \S\ref{subsub:Functional equation}) 
is mapped by $\mathrm{Art}_F$ to $c_\tau\in\Gamma$. 

\begin{definition}
We define an isomorphism 
\[
j_1\colon\bigoplus_{\sigma\colon F(\mathfrak{g})\rightarrow\mathbb{C}} \mathbb{R}(k-1) 
\stackrel{\cong}{\longrightarrow} \bigoplus_{\tau\in I} \mathbb{R}(k-1)\otimes_\mathbb{R}\mathbb{R}[\Gamma]
\]
via $(a_\sigma)_\sigma\mapsto (\sum_{\alpha\in\Gamma}a_{\tilde{\tau}\circ\alpha^{-1}}\otimes[\alpha])_{\tau}$.
Note that this map depends on the liftings $\tilde{\tau}$. 
\end{definition}

\begin{lemma}\label{lem:j_1}
The map $j_1$ is compatible with:
\begin{enumerate}
\item the left $\Gamma$-actions, where $\gamma\in\Gamma$ acts as 
\begin{align*}
(a_\sigma)_\sigma&\longmapsto (a_{\sigma\gamma})_\sigma \text{ on the left hand side and }\\
[\alpha]&\longmapsto[\gamma\alpha] \text{ on the $\tau$-component of the right hand side}, and 
\end{align*}
\item the involutions, which are defined by 
\begin{align*}
(a_\sigma)_\sigma&\longmapsto (\overline{a_{\overline{\sigma}}})_\sigma=((-1)^{k-1}a_{\overline{\sigma}})_\sigma 
\text{ on the left hand side and }\\
a\otimes[\alpha]
&\longmapsto \overline{a}\otimes[\alpha c_\tau]
=(-1)^{k-1}a\otimes[\alpha c_\tau]
\text{ on the $\tau$-component of the right hand side}. 
\end{align*}
\end{enumerate}
\end{lemma}
\begin{proof}
Straightforward from the definitions. 
\end{proof}

\begin{corollary}\label{cor: determinant}
By fixing a total order $\mathbf{o}$ of $I$, we have an isomorphism $j_1^{\mathbf{o}}$
\begin{equation}\label{eq: isomorphism 3}
\mathrm{det}_{\mathbb{R}[\Gamma]}\Bigl(H^1_{\mathscr{D}}((\Spec F(\mathfrak{g}))_{/\mathbb{C}}, \mathbb{R}(k))\Bigr)\cong \mathrm{det}_{\mathbb{R}[\Gamma]}(\mathbb{R}(k-1)\otimes_\mathbb{R}\mathbb{R}[\Gamma])^{\otimes I}\cong \mathbb{R}((k-1)g)\otimes_\mathbb{R}\mathbb{R}[\Gamma],
\end{equation}
and, under this isomorphism, we have an equality
\[
\mathrm{det}_{\mathbb{Q}[\Gamma]}(B_{F(\mathfrak{g})})=\mathbb{Q}((k-1)g)\otimes_\mathbb{Q}\mathbb{Q}[\Gamma].
\]
\end{corollary}

By Corollary \ref{cor: determinant}, we obtain a natural isomorphism 
\[
\mathbb{C}\otimes_{\chi, \mathbb{R}[\Gamma]}\mathrm{det}_{\mathbb{R}[\Gamma]}\Bigl(H^1_{\mathscr{D}}((\Spec F(\mathfrak{g}))_{/\mathbb{C}}, \mathbb{R}(k))\Bigr)
\cong \mathbb{C}\otimes_{\chi,\mathbb{R}[\Gamma]}\bigl(\mathbb{R}((k-1)g)\otimes_\mathbb{R}\mathbb{R}[\Gamma]\bigr) \cong \mathbb{C}
\]
(the second isomorphism sending $\lambda\otimes a\otimes [\gamma]$ to $\lambda a \chi(\gamma)$), 
and through this isomorphism, $\bigl(\mathrm{det}_{\mathbb{Q}[\Gamma]}(B_{F(\mathfrak{g})})\bigr)_\chi$ is mapped to 
$\mathbb{Q}(\chi)((k-1)g)=\mathbb{Q}(\chi)\cdot (2\pi i)^{(k-1)g}\subset\mathbb{C}$.

Next we investigate the case for Deligne cohomology over $\bbR$.
For $\tau\in I$, we denote by $(\mathbb{Q}(k-1)\otimes_\mathbb{Q}\mathbb{Q}[\Gamma])_\tau^+$ 
the $\mathbb{Q}[\Gamma]$-submodule of $\mathbb{Q}(k-1)\otimes_\mathbb{Q}\mathbb{Q}[\Gamma]$ 
consisting of elements satisfying 
\[
\sum_{\alpha\in\Gamma}a_\alpha\otimes [\alpha] = \sum_{\alpha\in\Gamma}\overline{a_\alpha}\otimes [\alpha c_\tau].
\]
We also define $(\mathbb{R}(k-1)\otimes_\mathbb{R}\mathbb{R}[\Gamma])_\tau^+$ in a similar way.

\begin{proposition}\label{over_R_vs_over_C}
We have 
\[
\bigl((\mathbb{Q}(k-1)\otimes_\mathbb{Q}\mathbb{Q}[\Gamma])_\tau^+\bigr)_\chi\cong
\begin{cases}
\mathbb{Q}(\chi)(k-1) &  \text{if $\chi(c_\tau)=(-1)^{k-1}$,} \\
0  &  \text{if $\chi(c_\tau)\neq (-1)^{k-1}$.} 
\end{cases}
\]
In particular, if $\chi$ is totally noncritical for $k$, we have a natural isomorphism
\[
\mathbb{C}\otimes_{\chi, \mathbb{R}[\Gamma]}\mathrm{det}_{\mathbb{R}[\Gamma]}
\Bigl(H^1_{\mathscr{D}}((\Spec F(\mathfrak{g}))_{/\mathbb{R}}, \mathbb{R}(k))\Bigr) \cong
\mathbb{C}\otimes_{\chi, \mathbb{R}[\Gamma]}\mathrm{det}_{\mathbb{R}[\Gamma]}
\Bigl(H^1_{\mathscr{D}}((\Spec F(\mathfrak{g}))_{/\mathbb{C}}, \mathbb{R}(k))\Bigr). \]
\end{proposition}
\begin{proof}
The proof of the former part is straightforward. 
For the latter part, first note that, 
for any finitely generated  $\bbR[\Gamma]$-module $V$, we have 
\[\bbC\otimes_{\chi,\bbR[\Gamma]}\det_{\bbR[\Gamma]}(V)\cong \det_{\bbC}(\bbC\otimes_{\chi,\bbR[\Gamma]}V). \]
Thus it suffices to show that 
\[
\bbC\otimes_{\chi,\bbR[\Gamma]}H^1_{\mathscr{D}}((\Spec F(\mathfrak{g}))_{/\mathbb{R}}, \mathbb{R}(k))
\cong
\bbC\otimes_{\chi,\bbR[\Gamma]}H^1_{\mathscr{D}}((\Spec F(\mathfrak{g}))_{/\mathbb{C}}, \mathbb{R}(k)). 
\]
Now we have 
\[
H^1_{\mathscr{D}}((\Spec F(\mathfrak{g}))_{/\mathbb{C}}, \mathbb{R}(k))
\cong \bigoplus_{\sigma\colon F(\mathfrak{g})\rightarrow\mathbb{C}} \mathbb{R}(k-1) 
\stackrel{j_1}{\cong} \bigoplus_{\tau\in I} \mathbb{R}(k-1)\otimes_\mathbb{R}\mathbb{R}[\Gamma]
\]
and 
\[
H^1_{\mathscr{D}}((\Spec F(\mathfrak{g}))_{/\mathbb{R}}, \mathbb{R}(k))
\cong \Biggl(\bigoplus_{\sigma\colon F(\mathfrak{g})\rightarrow\mathbb{C}} \mathbb{R}(k-1) \Biggr)^+
\stackrel{j_1}{\cong} \bigoplus_{\tau\in I} \bigl(\mathbb{R}(k-1)\otimes_\mathbb{R}\mathbb{R}[\Gamma]\bigr)_\tau^+ 
\]
by Lemma \ref{lem:j_1}. 
In particular, if $\chi$ is totally noncritical for $k$, the former part implies that 
\begin{align*}
\bbC\otimes_{\chi,\bbR[\Gamma]}H^1_{\mathscr{D}}((\Spec F(\mathfrak{g}))_{/\mathbb{R}}, \mathbb{R}(k))
&\cong \bigoplus_{\tau\in I}\bbC\otimes_{\chi,\bbR[\Gamma]}
\bigl(\mathbb{R}(k-1)\otimes_\mathbb{R}\mathbb{R}[\Gamma]\bigr)_\tau^+\cong \bigoplus_{\tau\in I}\bbC\\
&\cong \bigoplus_{\tau\in I}\bbC\otimes_{\chi,\bbR[\Gamma]}
\bigl(\mathbb{R}(k-1)\otimes_\mathbb{R}\mathbb{R}[\Gamma]\bigr)\\
&\cong\bbC\otimes_{\chi,\bbR[\Gamma]}H^1_{\mathscr{D}}((\Spec F(\mathfrak{g}))_{/\mathbb{C}}, \mathbb{R}(k)). \qedhere
\end{align*}
\end{proof}

Now we introduce an element by using Lerch zeta values. Fix an element $\eta_0\in\sT_0[\frg]$, and put
\begin{align*}
\mathbf{L}(\eta_0, k) 
&\coloneqq \sum_{\gamma\in\Gamma}\cL^\infty(\gamma(\eta_0), k)\otimes[\gamma] \\
&= \frac{1}{2}\sum_{\gamma\in\Gamma}\bigl(\cL(\gamma(\eta_0), k)+(-1)^{(k-1)g}\overline{\cL(\gamma(\eta_0), k)}\bigr)
\otimes[\gamma] \in \bbR((k-1)g)\otimes_\bbR\bbR[\Gamma].
\end{align*}

We also denote by $\mathbf{L}_{\sD}(\eta_0, k)$ the element of 
$\mathrm{det}_{\mathbb{R}[\Gamma]}
\Bigl(H^1_{\mathscr{D}}((\Spec F(\mathfrak{g}))_{/\mathbb{C}}, \mathbb{R}(k)\Bigr)$ 
corresponding to $\mathbf{L}(\eta_0, k)$ under the isomorphism $j^{\mathbf{o}}_1$ 
in Corollary \ref{cor: determinant}. 
Then we have the following interpretation of a part of Conjecture \ref{conj: Beilinson conjectures} 
in terms of Lerch zeta values: 

\begin{theorem}\label{thm: Lerch version of Beilinson}
Let $\chi\colon\Gamma\rightarrow\bbC$ be a character totally noncritical for $k$. Then, under the isomorphism 
\[
\mathbb{C}\otimes_{\chi,\mathbb{R}[\Gamma]}\bigl(\mathbb{R}((k-1)g)\otimes_\mathbb{R}\mathbb{R}[\Gamma]\bigr)
\stackrel{\cong}{\longrightarrow} \mathbb{C};\quad
\lambda\otimes a\otimes[\gamma]\longmapsto \lambda\cdot a\cdot\chi(\gamma),
\]
the element $1\otimes \mathbf{L}(\eta_0, k)$ is mapped to $\sum_{\gamma\in\Gamma}\chi(\gamma)\mathcal{L}(\gamma(\eta_0), k)$.

In particular, for a character $\chi$ totally noncritical for $k$, 
Conjecture \ref{conj: Beilinson conjectures} is equivalent to the assertion that
\[
1\otimes d_F^{1/2}\mathbf{L}_{\sD}(\eta_0, k)\in
\bbC\otimes_{\chi,\bbR[\Gamma]}\det_{\bbR[\Gamma]}\bigl(H^1_{\sD}((\Spec F(\frg))_{/\bbC},\bbR(k))\bigr)
\]
belongs to
$\bigl(\mathrm{det}_{\mathbb{Q}[\Gamma]}
(H^1_{\mathscr{M}}(\Spec F(\frg),\mathbb{Q}(k))_\mathrm{fin})\bigr)_\chi$ through the isomorphisms 
in Proposition \ref{over_R_vs_over_C} and Conjecture \ref{conj: Beilinson conjectures}.
\end{theorem}

Before the proof, we need some preparations. 
Note first that $\mathbb{T}$ has an involution $\iota$ 
which is given by $\colon t\mapsto t^{-1}$ on each component $\bbT^\fra$. 
This induces an involution $\sT[\frg]$, which is denoted again by $\iota$. 

\begin{lemma}\label{lem: Comparison of involutions}
We set $c_{\mathrm{abs}}=\prod_\tau c_\tau \in \Gamma$. Then, for $\eta\in\sT[\frg]$, we have an equality
\[
c_\mathrm{abs}(\eta) = \iota(\eta).
\]
\end{lemma}

\begin{proof}
Let $c'_\mathrm{abs}=\prod_\tau c'_\tau$ denote the element of $\Cl_F^+(\frg)$ 
corresponding to $c_\mathrm{abs}$. This element is the image of the idele 
$\prod_\tau\tilde{c}_\tau$ whose Archimedean components are $-1$ and non-Archimedean components are $1$, 
under the canonical map $\mathbb{A}_F^\times\to\Cl_F^+(\frg)$. 
By the definition of this canonical map, $c'_\mathrm{abs}$ is equal to  
the class of a principal ideal $\frb=(\beta)$ where $\beta$ is a totally negative element of $\cO_F$ 
satisfying $\beta\equiv 1\bmod \frg$. 

For any $\xi\in\bbT^\fra[\frg]$, the element $\xi^\frb\in\bbT^{\frb\fra}[\frg]$ is given by 
$\xi^\frb(\alpha\beta)=\xi(\alpha\beta)=\xi(\alpha)$ ($\alpha\in\fra$), 
where the latter equality follows from the congruence $\beta\equiv 1\bmod\frg$. 
On the other hand, since $-\beta$ is totally positive, 
$\xi^{\frb}\in\bbT^{\frb\fra}[\frg]$ is identified with 
$\pair{-\beta}\xi^\frb\in\bbT^{\fra}[\frg]$ defined by 
\[\pair{-\beta}\xi^\frb(\alpha)=\xi^\frb(-\beta\alpha)=\xi(-\alpha)=\xi(\alpha)^{-1}\] 
under the $F^\times_+$-action. 

By our definition \eqref{eq: Galois action} of the $\Gamma$-action, 
we obtain the identity $c_\mathrm{abs}^{-1}(\eta)=\iota(\eta)$ . 
Since $c_\mathrm{abs}$ is an involution, this completes the proof. 
\end{proof}

\begin{proof}[Proof of Theorem \ref{thm: Lerch version of Beilinson}]
By Lemma \ref{lem: Comparison of involutions}, we have 
\begin{align*}
\sum_{\gamma}\chi(\gamma)\cdot(-1)^{(k-1)g}\overline{\mathcal{L}(\gamma(\eta_0), k)} 
& = \sum_{\gamma}\chi(\gamma)\cdot(-1)^{(k-1)g}\mathcal{L}(\gamma(\iota(\eta_0)), k) \\
& = \sum_{\gamma}\chi(\gamma)\cdot(-1)^{(k-1)g}\mathcal{L}((\gamma c_\mathrm{abs})(\eta_0), k) \\
& = \sum_{\gamma'}\chi(\gamma')\bigl(\chi(c_\mathrm{abs})^{-1}\cdot(-1)^{(k-1)g}\bigr)\mathcal{L}(\gamma'(\eta_0), k).
\end{align*}
Since $\chi$ is totally noncritical, we have
\[
\chi(c_\mathrm{abs}) = \prod_{\tau\in I}\chi(c_\tau) = (-1)^{(k-1)g}. 
\]
Thus we obtain the former part. 

For the latter part, we consider the isomorphisms 
\begin{align*}
\bbC\otimes_{\chi,\bbR[\Gamma]}\det_{\bbR[\Gamma]}
\Bigl(H^1_{\sD}((\Spec F(\frg))_{/\bbR},\bbR(k))\Bigr)
&\cong
\bbC\otimes_{\chi,\bbR[\Gamma]}\det_{\bbR[\Gamma]}
\Bigl(H^1_{\sD}((\Spec F(\frg))_{/\bbC},\bbR(k))\Bigr)\\
&\cong
\bbC\otimes_{\chi,\bbR[\Gamma]}
\Bigl(\bbR((k-1)g)\otimes_{\bbR}\bbR[\Gamma]\Bigr)\\
&\cong \bbC
\end{align*}
given in Proposition \ref{over_R_vs_over_C}, Corollary \ref{cor: determinant} and the former part. 
The composition maps $(\det_{\bbQ[\Gamma]}B_{F(\frg)}^+)_\chi$ onto $\bbQ(\chi)((k-1)g)$. 
Therefore, Theorem \ref{thm: Lerch vs Artin} and the former part show that 
\[L^*(\chi^{-1},1-k)\cdot\bigl(\det_{\bbQ[\Gamma]}B_{F(\frg)}^+\bigr)_\chi
=(1\otimes d_F^{1/2}\mathbf{L}_{\sD}(\eta_0,k))\cdot\bbQ(\chi). \]
Thus the desired equivalence follows. 
\end{proof}

Now we can discuss a plectic approach. First, recalling the calculation in Lemma \ref{lem: ext2}, 
we introduce an element of plectic cohomology groups:

\begin{definition}
Let $\frg\subsetneq\cO_F$ be a nonzero ideal, so that $\bbT_0[\frg]\subset U$. 
An element $\mathbf{Ler}_\frg(k)$ of 
$H^{2g-1}_{\mathscr{D}^I}(\mathscr{T}_0[\mathfrak{g}]_{/\mathbb{C}}, \mathbb{R}(k_I))
\coloneqq H^{2g-1}_{\mathscr{D}^I}(\bbT_0[\frg]/\Ft, \mathbb{R}(k_I))$ 
is defined as 
\[
\mathbf{Ler}_\frg(k) = (\cL^\infty(\eta, k))_{\eta\in\sT_0[\frg]}
\]
under the identification via the natural isomorphism (cf.~Lemma \ref{lem: splitting plectic cohomology})
\begin{equation}\label{eq: isomorphism 1}
H^{2g-1}_{\mathscr{D}^I}(\sT_0[\frg]_{/\mathbb{C}}, \mathbb{R}(k_I)) 
\cong  \bigoplus_{\eta\in\sT_0[\frg]}\bbR((k-1)g).
\end{equation}
\end{definition}

To discuss the relation between this element $\mathbf{Ler}_\frg(k)$ and the element $\bL(\eta_0, k)$, 
we introduce the following isomorphism: 

\begin{definition}
For an element $\eta_0\in \sT_0[\frg]$, we define an isomorphism 
\begin{equation}\label{eq: isomorphism 2}
j^{\eta_0}_2:\bigoplus_{\eta\in\mathscr{T}_0[\mathfrak{g}]}\mathbb{R}((k-1)g) 
\cong \mathbb{R}((k-1)g)\otimes_\mathbb{R}\mathbb{R}[\Gamma]
\end{equation}
to be 
\[
(a_\eta)_\eta \mapsto \sum_{\gamma\in\Gamma}a_{\gamma(\eta_0)}\otimes[\gamma].
\]
\end{definition}

We can easily check the following lemma by using Lemma \ref{lem: Comparison of involutions} 
and by noting that we have $\overline{\iota\eta}=\eta$:

\begin{lemma}
The homomorphism $j_2^{\eta_0}$ is $\Gamma$-equivariant, where the multiplication by $[\gamma]$($\gamma\in\Gamma$) is 
defined via ``$(a_\eta)_\eta\mapsto (a_{\gamma^{-1}(\eta)})_\eta$'' on the left hand side, and 
compatible with the involutions defined as 
\begin{itemize}
\item ``$(a_\eta)_\eta\mapsto (\overline{a_{\overline{\eta}}})_\eta$'' on the left hand side, and
\item ``$a\otimes[\delta]\mapsto \overline{a}\otimes[c_{\mathrm{abs}}\delta]$'' on the right hand side.
\end{itemize}
\end{lemma}

By the composition of isomorphisms (\ref{eq: isomorphism 1}), (\ref{eq: isomorphism 2}), and (\ref{eq: isomorphism 3}), we obtain an isomorphism
\begin{equation}\label{eq: isomorphism final}
H^{2g-1}_{\mathscr{D}^I}(\mathscr{T}_0[\mathfrak{g}]_{/\mathbb{C}}, \mathbb{R}(k_I))\cong \mathrm{det}_{\mathbb{R}[\Gamma]}\Bigl(H^1_{\mathscr{D}}((\Spec F(\mathfrak{g}))_{/\mathbb{C}}, \mathbb{R}(k))\Bigr), 
\end{equation} 
and, through this isomorphism, the element $\mathbf{Ler}_{\frg}(k)$ is mapped to $\bL_{\sD}(\eta_0, k)$. 
Then we can restate the assertion in the latter part of Theorem \ref{thm: Lerch version of Beilinson} 
as follows: 

\begin{conjecture}\label{conj: Ler}
For a character $\chi$ totally noncritical for $k$, the element 
\[
1\otimes d_F^{1/2}\mathbf{Ler}_\frg(k)
\in \mathbb{C}\otimes_{\chi,\mathbb{R}[\Gamma]}
H^{2g-1}_{\mathscr{D}^I}(\mathscr{T}_0[\mathfrak{g}]_{/\mathbb{C}}, \mathbb{R}(k_I))
\]
is ``motivic'' in the sense that its image through the isomorphism (\ref{eq: isomorphism final}) 
tensored with $\bbC\otimes_{\chi,\bbR[\Gamma]}(-)$ belongs to 
$\bigl(\mathrm{det}_{\mathbb{Q}[\Gamma]}(H^1_{\mathscr{M}}(\Spec F(\frg),\mathbb{Q}(k))_\mathrm{fin})\bigr)_\chi$.
\end{conjecture}

On the other hand, Conjecture \ref{conj: i^*pol} implies that the element $\mathbf{pol}\in H^{2g-1}_{\sD^I}(U/\Ft, \bbLog)$ 
is specialized to $d_F^{1/2}\mathbf{Ler}_\frg(k)$. 
Therefore, if we establish the motivicity of $\pol$ and the specialization map 
(in addition to Conjecture \ref{conj: i^*pol}), then Conjecture \ref{conj: Ler} follows. 
We hope to pursue this problem in the future. 

\begin{remark}
In this section, we have considered only the totally noncritical case. 
To treat the general case, it might be necessary to introduce the notion of 
the (equivariant) plectic Deligne--Beilinson cohomology over $\bbR$. 
\end{remark}

\appendix
%
%
\section{Equivariant Deligne-Beilinson Cohomology}\label{sec: Appendix}
%
%

Let $\fra$ be a fractional ideal of $F$, and let $\T^\fra$ be as in \S
\ref{subsec: equivariant cohomology of Log}.
For each integer $N\geq0$, we let $\bbLog^N_\fra$ be the logarithm sheaf, 
which is an admissible unipotent variation of mixed $\bbR$-Hodge structures
on $\T^\fra$ with an equivariant action of $\Delta=\cO^\times_{F+}$.
In this appendix, using the logarithmic Dolbeault complex of Burgos \cite{Bur94},
we will construct an explicit mixed $\bbR$-Hodge complex $R\Gamma_{\mathrm{Hdg}}(U^\fra/\Delta,\bbLog^N_\fra)$,
which may be used to calculate the equivariant cohomology
$H^m(U^\fra/\Delta,\bbLog^N_\fra)$
and the equivariant Deligne-Beilinson cohomology 
$H^m_{\sD}(U^\fra/\Delta,\bbLog^N_\fra)$.

%
\subsection{Logarithmic Dolbeault Complex}\label{subsec: LDC}
%

In this subsection, we review the theory of logarithmic Dolbeault complex defined by Burgos.
Let $U$ be a $g$-dimensional smooth algebraic variety over $\bbC$, which we view as a complex manifold.
We denote by $\sA_{U,\bbR}$ (resp.  $\sE_{U,\bbR}$) the sheaf of real-valued real analytic (resp. $C^\infty$) functions on $U$, 
and by $\sA^\bullet_{U,\bbR}$ (resp. $\sE^\bullet_{U,\bbR}$) the corresponding complex of sheaves of differential forms.
The $C^\infty$-de Rham complex gives a resolution
\begin{equation}\label{eq: infinity resolution}
	0  \rightarrow \bbR \rightarrow \sE_{U,\bbR} \rightarrow  \sE^1_{U,\bbR} \rightarrow \sE^2_{U,\bbR}  \rightarrow \cdots
\end{equation}
of the constant sheaf $\bbR$ on $U$ by the fine sheaves $\sE^\bullet_{U,\bbR}$. 

Let $X$ be a smooth compactification of $U$ such that $D\coloneqq X\setminus U$ is 
a simple normal crossing divisor, and let $j\colon U\hookrightarrow X$ be the inclusion. 
To define the complex $\sA^\bullet_{X,\bbR}(D)$, we take a local coordinate system $(t_1,\ldots,t_g)$ such that $D$ is defined by $t_1\cdots t_h=0$. 
Then, on this coordinate neighborhood, we define the complex $\sA^\bullet_{X,\bbR}(D)$ to be the $\sA_{X,\bbR}$-subalgebra of $j_*\sA^\bullet_{U,\bbR}$ generated 
by  $\sA^\bullet_{X,\bbR}$ and the sections 
\begin{equation}\label{eq: weight one}
\log|t_\mu|, \qquad \Re(d\log t_\mu), \qquad \Im(d\log t_\mu) 
\end{equation}
for $1\le\mu \le h$. 
We define the weight filtration $W_\bullet$ on $\sA^\bullet_{X,\bbR}(D)$ 
to be the multiplicative ascending filtration obtained by assigning weight $0$ to the sections of $\sA_{X,\bbR}^\bullet$ and weight $1$ 
to the local sections given in \eqref{eq: weight one}.

Following Burgos \cite{Bur94}*{\S 2}, we define the complex $\sE^\bullet_{X,\bbR}(D)$ to be the image of the natural map
\[
	 \sA^\bullet_{X,\bbR}(D) \otimes_{\sA_{X,\bbR}} \sE_{X,\bbR} \rightarrow j_*  \sE^\bullet_{U,\bbR}.
\]
We let $\sE_{X,\bbR}(D):= \sE^0_{X,\bbR}(D)$.
The complex  $\sE^\bullet_{X,\bbR}(D)$ has a weight filtration $W_\bullet$ induced from the weight filtration $W_\bullet$ on $\sA^\bullet_{X,\bbR}(D)$.
The complex $\sE^\bullet_{U,\bbC}:=\sE^\bullet_{U,\bbR}\otimes\bbC$ has a bigrading 
\[
	\sE^m_{U,\bbC}=\bigoplus_{p,q\in\bbZ, p+q=m} \sE^{p,q}_{U,\bbC}
\]
induced from the complex structure of $U$, and this bigrading induces a bigrading 
\begin{align*}
	\sE^m_{X,\bbC}(D) &= \bigoplus_{p,q\in\bbZ, p+q=m} \sE^{p,q}_{X,\bbC}(D), &
	\sE^{p,q}_{X,\bbC}(D) &:= \sE^{p+q}_{X,\bbC}(D) \cap j_* \sE^{p,q}_{U,\bbC}
\end{align*}
on $\sE^\bullet_{X,\bbC}(D):=\sE^\bullet_{X,\bbR}(D)\otimes\bbC$.  We define the Hodge filtration on the complex $\sE^\bullet_{X,\bbC}(D)$ by 
\[
	F^p \sE^\bullet_{X,\bbC}(D) = \bigoplus_{r\geq p} \sE^{r,s}_{X,\bbC}(D).
\]
Then we have the following.

\begin{theorem}[\cite{Bur94}*{Theorem 1.2, Theorem 2.1}]\label{thm: Bu}
 	There exist filtered quasi-isomorphisms
	\[
		(Rj_*\bbR,\tau_{\leq\bullet}) 
		\xrightarrow\cong (\sA^\bullet_{X,\bbR}(D),W_\bullet) 
		\xrightarrow\cong (\sE^\bullet_{X,\bbR}(D),W_\bullet). 
	\]
	There also exists a bifiltered quasi-isomorphism
	\begin{equation}\label{eq: resolution2}
		(\Omega^\bullet_{X}(D),W_\bullet,F^\bullet) \xrightarrow\cong ( \sE^\bullet_{X,\bbC}(D),W_\bullet,F^\bullet).
	\end{equation}
\end{theorem}
We note that the filtration $W_\bullet$ on $\Omega^\bullet_{X}(D)$ is the multiplicative ascending
filtration obtained by assigning weight $0$ to the sections of $\Omega^\bullet_X$ and weight $1$
to the local sections $d\log t_\mu$ for $1\leq\mu\leq h$, where $(t_1,\ldots,t_g)$ are local coordinates
such that $D$ is defined by $t_1\cdots t_h=0$, and the filtration $F^\bullet$ on 
$\Omega^\bullet_{X}(D)$ is the stupid filtration.

The advantage of using the sheaf $\sE_{X,\bbR}$ is that it is \emph{fine}, so that
for integers $m,n,p\in\bbZ$, the sheaves $\Gr^W_n \sE^m_{X,\bbR}(D)$ and 
$\Gr^W_n \Gr^p_F \sE^m_{X,\bbC}(D)$ are acyclic with respect to the functor $\Gamma(X, -)$. 
Hence by Theorem \ref{thm: Bu} and the classical results (cf.~\cites{Del71,Del74}), 
the triple $(\Gamma(X,\sE^\bullet_{X,\bbR}(D)),W_\bullet,F^\bullet)$ defines
a \emph{mixed $\mathbb{R}$-Hodge complex} in the sense of Deligne \cite{Del74}*{(8.1.5)}.
We denote by $\Dec(W)_\bullet$ (the dual of) Deligne's d\'{e}calage \cite{Del71}*{(1.3.3)} of the weight filtration.
Then the triple $(\Gamma(X,\sE^\bullet_{X,\bbR}(D)),\Dec(W)_\bullet,F^\bullet)$ defines an $\bbR$-Hodge complex in the sense of Beilinson \cite{Bei86}*{Definition 3.2}, which we call an \emph{absolute $\bbR$-Hodge complex} according to the terminology of \cite{CG16}*{Definition 4.4}.
We note that, these objects correspond to a complex of mixed $\bbR$-Hodge structures in the following sense.

\begin{theorem}[{\cite{CG16}*{Theorem 4.11}, \cite{Bei86}*{Theorem 3.4} (see also \cite{CG16}*{Theorem 4.12})}]\label{thm: A2}
	Let $\Ho(\MHC_\bbR)$ and $\Ho(\AHC_\bbR)$ be the localizations by level-wise quasi-isomorphisms of the 
	categories of mixed $\bbR$-Hodge complexes and absolute $\bbR$-Hodge complexes, respectively.
	Then we have equivalences of categories
	\[
		\Ho(\MHC_\bbR)\xrightarrow[\Dec]{\cong}\Ho(\AHC_\bbR)\xleftarrow{\cong}D^+(\MHS_\bbR),
	\]
	where the first equivalence is induced by the d\'{e}calage with respect to weight filtration, 
	and the second equivalence is induced by the natural inclusion functor.
\end{theorem}

The mixed $\bbR$-Hodge structure on $H^m(U,\bbR)$ is the $m$-th cohomology of the corresponding object of $D^+(\MHS_\bbR)$.
Namely, its weight filtration and Hodge filtration are induced by $\Dec(W)_\bullet$ and $F^\bullet$ via the isomorphism $H^m(\Gamma(X,\sE^\bullet_{X,\bbR}(D)))\isomto H^m(U,\bbR)$.

Suppose that there is another smooth compactification $X'$ with $D'\coloneqq X'\setminus U$ 
a simple normal crossing divisor, and a morphism $\pi\colon X\to X'$ which gives a commutative diagram 
\[
	\xymatrix{
		X \ar[r]^-{\pi}&  X'\\
		U \ar@{^(->}[u]^{j} \ar@{=}[r] & U \ar@{^(->}[u]_{j'}.
	}
\]
Then we have the pullback morphism $\sE^\bullet_{X',\bbR}(D')\to\pi_*\sE^\bullet_{X,\bbR}(D)$ and, 
again by Theorem \ref{thm: Bu} and the classical results, 
the induced morphism $\Gamma(X',\sE^\bullet_{X',\bbR}(D'))\rightarrow\Gamma(X,\sE^\bullet_{X,\bbR}(D))$ 
is a bifiltered quasi-isomorphism for the filtrations $\Dec(W)_\bullet$ and $F^\bullet$. 


We now return to the case of the logarithm sheaf $\bbLog_\fra$. 
We let $U^\fra\coloneqq\T^\fra\setminus\{1\}$, and we let $\bbLog_\fra$ be the logarithm sheaf on $\T^\fra$, which 
by Proposition \ref{prop: bbLog_fra^N} is an admissible unipotent variation of mixed $\bbR$-Hodge 
structures on $\T^\fra$.
If we fix a $\bbZ$-basis $\a=(\alpha_1,\ldots,\alpha_g)$ of $\fra$, then we have an isomorphism
$\T^\fra\cong\bbG_m^g$ and hence a compactification $\T^\fra\hookrightarrow Y\coloneqq(\bbP^1)^g$.
Let $X$ be a smooth compactification of $U^\fra$ whose complement $D\coloneqq X\setminus U^\fra$ 
is a simple normal crossing divisor, and assume that there exists a morphism $X\to Y$ 
which is the identity on $U^\fra$. 
Then there are the sections $\log\absv{t^{\alpha_i}}$ ($i=1,\ldots,g$) of $\sE_{X,\bbR}(D)$.  

Recall that, for any integer $N\geq 0$, the logarithm sheaf $\cLog^N_\fra$ on $\bbT^\fra$ has 
a natural extension to a coherent $\sO_Y$-module with logarithmic connection 
(see the proof of Proposition \ref{prop: bbLog_fra^N}). 
We denote by $\cLog^N_\fra(D)$ its pullback to $X$, 
which is a coherent $\sO_{X}$-module with logarithmic connection 
 \[
 	\nabla\colon\cLog^N_\fra(D)\rightarrow\cLog^N_\fra(D)\otimes\Omega^1_{X}(D).
 \]
 Then the module
 $\cLog^N_\fra(D)\otimes\sE_{X,\bbC}(D)$ has an induced connection
 \[
 	\nabla\colon\cLog^N_\fra(D)\otimes\sE_{X,\bbC}(D)\rightarrow\cLog^N_\fra(D)\otimes\sE^1_{X,\bbC}(D).
 \]
 If we define the local sections $\upsilon_\a^\bsk$ of $\cLog^N_\fra(D)\otimes\sE_{X,\bbC}(D)$ by 
 \[
 	\upsilon_\a^\bsk\coloneqq i^{\absv{\bsk}}\exp\left(\sum_{i=1}^g(-\log\absv{t^{\alpha_i}})u_{\alpha_i}\right)
	\cdot u^{\bsk}_\a,
 \]
 then we have
 \begin{align*}
	\nabla(\upsilon_\a^\bsk)&=i^{\absv{\bsk}}\exp\left(\sum_{i=1}^g(-\log\absv{t^{\alpha_i}})u_{\alpha_i}\right)
	\cdot \sum_{j=1}^g u^{\bsk+1_j}_\a\otimes d\log t^{\alpha_j}\\
	&\qquad-i^{\absv{\bsk}}
	 \sum_{j=1}^g
	\exp\left(\sum_{i=1}^g(-\log\absv{t^{\alpha_i}})u_{\alpha_i}\right)
	u_{\alpha_j}\cdot u^{\bsk}_\a\otimes d\log\absv{t^{\alpha_j}}\\
	&=\sum_{j=1}^gi^{\absv{\bsk}}\exp\left(\sum_{i=1}^g(-\log\absv{t^{\alpha_i}})u_{\alpha_i}\right)
	\cdot  u^{\bsk+1_j}_\a\otimes i\Im(d\log t^{\alpha_j})
	= \sum_{j=1}^g \upsilon^{\bsk+1_j}_\a\otimes \Im(d\log t^{\alpha_j}).
 \end{align*}
 That is,
 \begin{equation}\label{eq: up}
	 \nabla(\upsilon_\a^\bsk)=\sum_{j=1}^g \upsilon^{\bsk+1_j}_\a\otimes \Im(d\log t^{\alpha_j})
 \end{equation}
 for any $\bsk\in\bbN^g$.
 
 \begin{definition}\label{def: sLog(D)}
 	We let $\sLog^N_{\fra,\bbR}(D)$ be the free 
	$\sE_{X,\bbR}(D)$-module generated by $\upsilon^\bsk_\a$ for $\bsk\in\bbN^g$
	with $\absv{\bsk}\leq N$, with connection
	\[
		\nabla\colon\sLog^N_{\fra,\bbR}(D)\rightarrow\sLog^N_{\fra,\bbR}(D)
		\otimes_{\sE_{X,\bbR}(D)}\sE^1_{X,\bbR}(D)
	\]	
	given as in \eqref{eq: up}. 
    In order to define the filtrations, 
    we introduce an auxiliary sheaf $\sL^N_{\fra,\bbR}$, defined to be the free $\sE_{X,\bbR}$-module 
    generated by $\upsilon^\bsk_\a$ for $\bsk\in\bbN^g$ with $\absv{\bsk}\leq N$, 
    so that $\sLog^N_{\fra,\bbR}(D)=\sL^N_{\fra,\bbR}\otimes_{\sE_{X,\bbR}}\sE_{X,\bbR}(D)$. 
    Then we define the filtrations $W_\bullet\sL^N_{\fra,\bbR}$ and 
    $F^\bullet\sL^N_{\fra,\bbC}$ by
    \begin{align*}
    W_{-2m}\sL^N_{\fra,\bbR}=W_{-2m+1}\sL^N_{\fra,\bbR}
    &\coloneqq\prod_{m\le\absv{\bsk}\le N}\sE_{X,\bbR}\cdot\upsilon^\bsk_\a,&
    F^{-p}\sL^N_{\fra,\bbC}
    &\coloneqq\prod_{0\leq\absv{\bsk}\leq p}\sE_{X,\bbC}\cdot\upsilon^\bsk_\a, 
    \end{align*}
    where $(-)_\bbC=(-)_\bbR\otimes\bbC$ as usual. 
    Then we define the corresponding filtrations on $\sLog^N_{\fra,\bbR}(D)$ 
    to be the tensor products of those on $\sL^N_{\fra,\bbR}$ and $\sE_{X,\bbR}(D)$. 
 \end{definition}
 
 From the identity
 \[
\gamma^{\mathbf{k}}_{\bsa}
=
(2\pi)^{\absv{\bsk}}\exp \left(\sum_{j=1}^g(-\Im (\log t^{\alpha_j})) \upsilon_{\bsa}^{1_j}\right)\upsilon_{\bsa}^{\bsk},  
 \]
 which can be verified easily, we see that 
 \[
 \sLog^N_{\fra, \bbR}(D)|_{U^{\fra}}
 =
 \bbLog^N_{\fra} \otimes_{\bbR} \sE_{U^{\fra}, \bbR}. 
 \]
In particular, $\sLog^N_{\fra, \bbR}(D)|_{U^{\fra}}$ naturally has a $\Delta$-equivariant structure.

By construction, the connection $\nabla$ is integrable, and defines a complex 
$\sLog^N_{\fra,\bbR}(D)\otimes_{\sE_{X,\bbR}(D)}\sE^\bullet_{X,\bbR}(D)
=\sL^N_{\fra,\bbR}\otimes_{\sE_{X,\bbR}}\sE^\bullet_{X,\bbR}(D)$, 
which is bifiltered by the tensor products of filtrations on $\sL^N_{\fra,\bbR}$ 
and $\sE^\bullet_{X,\bbR}(D)$ given as in \cite{BHY18}*{\S4.3}.

\begin{proposition}[{cf.~\cite{BHY18}*{\S4}}]
    Let $N$ be a non-negative integer and $j\colon U^\fra\hookrightarrow X$ be the natural inclusion.
    \begin{enumerate}
    \item 
    $R\Gamma\bigl(X,\sLog^N_{\fra,\bbR}(D)\otimes\sE^\bullet_{X,\bbR}(D)\bigr)$ 
    with the filtrations $W_\bullet$ (resp.~$\Dec(W)_\bullet$) and $F^\bullet$ is 
    a mixed $\bbR$-Hodge complex (resp.~an absolute $\bbR$-Hodge complex). 
    In particular, the hypercohomology groups 
    $\bbH^m\bigl(X,\sLog_{\fra,\bbR}^N(D)\otimes\sE_{X,\bbR}^\bullet(D)\bigr)$ 
    with the filtrations $\Dec(W)_\bullet$ and $F^\bullet$ are mixed $\bbR$-Hodge structures. 
    \item 
    We have a canonical quasi-isomorphism
	\begin{equation*}
		Rj_*\bbLog^N_\fra\cong\sLog^N_{\fra,\bbR}(D)\otimes\sE^\bullet_{X,\bbR}(D)
	\end{equation*}
	in the derived category of abelian sheaves on $X$, which induces an isomorphism 
    \[
        H^m(U^\fra,\bbLog_\fra^N)\cong 
        \bbH^m\bigl(X,\sLog_{\fra,\bbR}^N(D)\otimes\sE_{X,\bbR}^\bullet(D)\bigr). 
    \]
    \item 
    The natural inclusion 
	\begin{equation*}
		\cLog^N_\fra(D)\otimes_{\sO_X}\Omega^\bullet_{X}(D)
		\hookrightarrow
		\sLog^N_{\fra,\bbR}(D)\otimes\sE^\bullet_{X,\bbC}(D)
	\end{equation*}
    is a bifiltered quasi-isomorphism, where the filtrations on the left-hand side are those 
    given by Hain--Zucker \cite{HZ87}*{\S8}. 
    In particular, the mixed $\bbR$-Hodge structure on $H^m(U^\fra,\bbLog_\fra^N)$ 
    induced via (1) and (2) above coincides with the one 
    constructed in \cite{HZ87}*{(8.6) Proposition}.
    \item
    The complex $\sLog^N_{\fra,\bbR}(D)\otimes\sE^\bullet_{X,\bbR}(D)$ is acyclic 
    with respect to $\Gamma(X,-)$, so that we have a canonical isomorphism 
	\[
		\bbH^m\bigl(X,\sLog_{\fra,\bbR}^N(D)\otimes\sE_{X,\bbR}^\bullet(D)\bigr)\cong 
		H^m\bigl(\Gamma(X,\sLog^N_{\fra,\bbR}(D)\otimes\sE^\bullet_{X,\bbR}(D))\bigr). 
	\]
    Moreover, this isomorphism is compatible with the filtrations $\Dec(W)_\bullet$ and $F^\bullet$. 
    \end{enumerate}
\end{proposition}

\begin{proof}
    To show (1), we note that 
    \begin{align*}
    \Gr^W_n\bigl(\sLog_{\fra,\bbR}^N(D)\otimes_{\sE_{X,\bbR}(D)}\sE_{X,\bbR}^\bullet(D)\bigr)
    &\cong\Gr^W_n\bigl(\sL_{\fra,\bbR}^N\otimes_{\sE_{X,\bbR}}\sE_{X,\bbR}^\bullet(D)\bigr)\\
    &\cong\bigoplus_{m=0}^N\Gr^W_{-2m}\sL_{\fra,\bbR}^N
    \otimes_{\sE_{X,\bbR}}\Gr^W_{n+2m}\sE_{X,\bbR}^\bullet(D)\\
    &\cong\bigoplus_{m=0}^N \bbR(m)^{\#\{\bk\in\bbN^g\mid \absv{\bk}=m\}}
    \otimes_{\bbR}\Gr^W_{n+2m}\sE_{X,\bbR}^\bullet(D). 
    \end{align*}
    By Theorem \ref{thm: Bu} and the classical results (cf.~\cites{Del71,Del74}), 
    this is an $\bbR$-Hodge complex of sheaves of weight $n$ 
    (in the sense of Peters--Steenbrink \cite{PS08}*{Definition 2.32}). 
    Hence $\sLog_{\fra,\bbR}^N(D)\otimes_{\sE_{X,\bbR}(D)}\sE_{X,\bbR}^\bullet(D)$ is 
    a mixed $\bbR$-Hodge complex of sheaves (\cite{PS08}*{Definition 3.13}), and 
    our statement (1) follows from \cite{PS08}*{Theorem 3.18}. 

    (3) is proved by the same argument, that is, reducing to Theorem \ref{thm: Bu} 
    by taking $\Gr^W_n$. See also \cite{BHY18}*{Theorem 4.7}. 

    To prove (2), we use the long exact sequence associated to the short exact sequence
	\[
		0\rightarrow\Sym^{N}\bbR(\bone)\rightarrow
		\bbLog^{N}_\fra\rightarrow\bbLog^{N-1}_\fra\rightarrow 0.
	\]
    The quasi-isomorphism 
    \[Rj_*\Sym^{N}\bbR(\bone)\cong \Sym^{N}\bbR(\bone)\otimes_{\bbR}\sE_{X,\bbR}^\bullet(D)\]
    for the constant local system $\Sym^{N}\bbR(\bone)$ follows from Theorem \ref{thm: Bu}, 
    and the statement (2) holds by induction on $N$. 

	(4) follows from the fact that any $\sE_{X,\bbR}$-module is fine. 
\end{proof}

%
\subsection{Equivariant Deligne-Beilinson Cohomology}
%

In this subsection, we will construct the complex 
$R\Gamma_{\mathrm{Hdg}}(U^\fra/\Delta,\bbLog^N_\fra)$
and define equivariant Deligne-Beilinson cohomology $H_{\sD}^m(U^\fra/\Delta,\bbLog_\fra^N)$. 
Since the $\Delta$-action on $U^\fra$ does not seem to extend to a fixed compactification, 
we consider a certain system of compactifications of $U^\fra$. 
We first fix a basis $\bsalpha=(\alpha_1,\ldots,\alpha_g)$ of $\fra$ 
and let $X_\bsalpha\to Y=(\bbP^1)^g$ be the blowup of $Y$ with respect to $\{1\} \subset \mathbb{T}^{\fra} \hookrightarrow Y$. 
Note that the compactification $U^\fra\hookrightarrow X_\bsalpha$ satisfies 
the conditions considered in \S\ref{subsec: LDC}. 
We also fix a generator $\ve_1,\ldots,\ve_{g-1}$ of $\Delta$. 

\begin{proposition}\label{prop: system}
	There exists a system of smooth compactifications 
	$j_J\colon U^\fra\hookrightarrow X_J$ of $U^\fra$ over $\bbQ$, 
	indexed by subsets $J\subset\{1,\ldots,g-1\}$, 
	whose complement $X_J\setminus U^\fra$ is a simple normal crossing divisor 
	such that $X_\emptyset=X_\bsalpha$ 
	and satisfying the following conditions:
	\begin{enumerate}
		\item For any inclusion $J'\subset J$, there exists a morphism 
		$\pi=\pi_{J,J'}\colon X_J\rightarrow X_{J'}$ 
		such that the diagram
		\[
			\xymatrix{
				X_J\ar[r]^{\pi_{J,J'}} &X_{J'}\\
				U^\fra\ar@{=}[r]^\id\ar@{^(->}[u]&U^\fra\ar@{^(->}[u]
			}
		\]
		commutes.
		\item For any $i\in J$, there exists a morphism  
		$\pair{\wt\ve_i}=\pair{\wt\ve_i}_J\colon X_J\rightarrow X_{J'}$ for $J'\coloneqq J\setminus\{i\}$
		such that the diagram
		\[
			\xymatrix{
				X_J\ar[r]^{\pair{\wt\ve_i}} &X_{J'}\\
				U^\fra\ar[r]^{\pair{\ve_i}}\ar@{^(->}[u]&U^\fra\ar@{^(->}[u]
			}
		\]
		commutes.
	\end{enumerate}
\end{proposition}

\begin{proof}
	We construct $X_J$ by induction on $\absv{J}$. 
	For $\absv{J}=0$, we let $X_J=X_\emptyset\coloneqq X_\bsalpha$.
	For $J\ne\emptyset$, suppose we have constructed $X_{J'}$ for all $J'\subsetneq J$.
	Using the embeddings $j_{J'}\colon U^\fra\hookrightarrow X_{J'}$, consider the embedding
	\[
		U^\fra\rightarrow\prod_{\substack{J'\subsetneq J\\\absv{J'}=\absv{J}-1}} X_{J'}\times X_{J'}
	\]
	given as the product of the embeddings
	\[
		U^\fra\xrightarrow{(j_{J'},j_{J'}\circ\pair{\ve_i})}X_{J'}\times X_{J'},
	\]
	where $\pair{\ve_i}\colon U^\fra\rightarrow U^\fra$ is the morphism induced from the multiplication by $\ve_i$.
	We let $\ol U^\fra$ be the closure of $U^\fra$ in $\prod_{\absv{J'}=\absv{J}-1} X_{J'}\times X_{J'}$,
	and we take $X_J$ to be any smooth blowup of $\ol U^\fra$ such that the complement
	$X_J\setminus U^\fra$ is a simple normal crossing divisor.  
	For $J'\subsetneq J$ such that $|J'|=|J|-1$, the morphism $\pi_{J,J'}\colon X_{J}\rightarrow X_{J'}$ is 
	defined as the projection to the first component, 
	and $\pair{\wt\ve_i}$ for $\{i\}=J\setminus J'$ 
	is defined as the projection to the second component of $X_{J'}\times X_{J'}$. 
	Then, for $J''\subset J'$, we define $\pi_{J,J''}$ to be the composition $\pi_{J',J''}\circ\pi_{J,J'}$. 
\end{proof}

By \cite{Liu02}*{Chap. 3, Prop. 3.11}, the morphisms $\pi$ and $\pair{\widetilde{\varepsilon}_i}$ in Proposition \ref{prop: system} are unique. 
For the same reason, we have the commutativity like 
$\pair{\widetilde{\varepsilon}_i}\circ\pair{\widetilde{\varepsilon}_j}
=\pair{\widetilde{\varepsilon}_j}\circ\pair{\widetilde{\varepsilon}_i}$. 

We will define $R\Gamma_{\mathrm{Hdg}}(U^\fra/\Delta,\bbLog^N_\fra)$
using the system of smooth compactifications 
and the logarithmic Dolbeault complex.
Let $\{X_J\}$ be a system of compactifications as in Proposition \ref{prop: system} 
and set $D_J\coloneqq X_J\setminus U^\fra$. 
On each $X_J$, the logarithm sheaf $\sLog^N_{\fra,\bbR}(D_J)$ is defined 
as in Definition \ref{def: sLog(D)}. 

\begin{lemma}
	For $J'\coloneqq J\setminus\{i\}$, we have natural isomorphisms
	\[
        \pi^*\sLog^N_{\fra,\bbR}(D_{J'})\xrightarrow\cong\sLog^N_{\fra,\bbR}(D_J),\qquad
		\pair{\wt{\ve}_i}^*\sLog^N_{\fra,\bbR}(D_{J'})\xrightarrow\cong\sLog^N_{\fra,\bbR}(D_J),
	\]
	which are compatible with the connection and the filtrations $W_\b$ and $F^\b$.
	Here, the pullbacks are performed as $\sE_{X_J,\bbR}(D_J)$-modules, i.e., we set 
	\[\pi^*\sLog^N_{\fra,\bbR}(D_{J'})\coloneqq
	\pi^{-1}\sLog^N_{\fra,\bbR}(D_{J'})\otimes_{\pi^{-1}\sE_{X_{J'},\bbR}(D_{J'})}\sE_{X_J,\bbR}(D_J)\]
	and similarly for $\pair{\wt{\ve}_i}^*$. 
 Moreover, on $U^{\fra}$, these isomorphisms agree with $\id$ and the action of $\varepsilon_i$ respectively. 
\end{lemma}

\begin{proof}
Let $\sL_J$ denote the sheaf $\sL^N_{\fra,\bbR}$ introduced in Definition \ref{def: sLog(D)} 
for $X=X_J$. Then we have natural identifications 
\begin{align*}
\pi^*\sLog^N_{\fra,\bbR}(D_{J'})
&=\pi^{-1}\sL_{J'}\otimes_{\pi^{-1}\sE_{X_{J'},\bbR}}\sE_{X_J,\bbR}(D_J)\\
&=\pi^{-1}\sL_{J'}\otimes_{\pi^{-1}\sE_{X_{J'},\bbR}}
\sE_{X_J,\bbR}\otimes_{\sE_{X_J,\bbR}}\sE_{X_J,\bbR}(D_J)\\
&=\sL_J\otimes_{\sE_{X_J,\bbR}}\sE_{X_J,\bbR}(D_J)\\
&=\sLog^N_{\fra,\bbR}(D_J). 
\end{align*}
Here we identify the bases $\upsilon_\fra^\bsk$ for $\sL_{J'}$ and $\sL_J$. 
This gives the first isomorphism. 

The second isomorphism, roughly speaking, sends the basis $\upsilon_{\bsalpha}^\bsk$ 
to $\upsilon_{\ve_i\bsalpha}^\bsk$ defined by another basis $\ve_i\bsalpha$ of $\fra$. 
To describe it more precisely, we define an integer matrix $(c_{jl})_{j,l}$ by 
\[\ve_i\alpha_j=\sum_{l=1}^g c_{jl}\alpha_l \]
and denote the transpose of its inverse matrix by $(\tilde{c}_{jl})_{j,l}$. 
Then we have $u_{\ve_i\alpha_j}=\sum_{l=1}^g\tilde{c}_{jl}u_{\alpha_l}$ 
where $u_{\ve_i\alpha_j}$ is associated with the basis $\ve_i\bsalpha$, 
and hence the formula 
\[u_{\ve_i\bsalpha}^\bsk 
=\prod_{j=1}^g\Biggl(\sum_{l=1}^g\tilde{c}_{jl}u_{\alpha_l}\Biggr)^{k_j}\]
(we interpret the right-hand side by first expanding it into a linear combination of 
monomials $u_{\alpha_1}^{n_1}\cdots u_{\alpha_g}^{n_g}$ and then replacing them with 
$u_{\bsalpha}^{(n_1,\ldots,n_g)}$). 
Thus we obtain an expression $u_{\ve_i\bsalpha}^\bsk=\sum_{\bsl}a_{\bsk,\bsl}u_\bsalpha^\bsl$. 
Using the same coefficients, we define an isomorphism 
\[\pair{\wt{\ve}_i}^{-1}\sL_{J'}\otimes_{\pair{\wt{\ve}_i}^{-1}\sE_{X_{J'},\bbR}}\sE_{X_J,\bbR}
\to \sL_J\]
by sending $\upsilon_{\bsalpha}^\bsk$ to $\sum_{\bsl}a_{\bsk,\bsl}\upsilon_{\bsalpha}^\bsl$. 
This induces the second isomorphism of our statement. 
\end{proof}

Using the above isomorphisms, we obtain morphisms of complexes
\begin{align*}
	\pi^*,\,\pair{\wt\ve_i}^*\colon\Gamma(X_{J'},\sLog^N_{\fra,\bbR}(D_{J'})
	\otimes\sE^\b_{X_{J'},\bbR}(D_{J'}))
	&\rightarrow\Gamma(X_J,\sLog^N_{\fra,\bbR}(D_J)
	\otimes\sE^\b_{X_J,\bbR}(D_J))
\end{align*}
compatible with the filtrations $W_\b$ and $F^\b$.
We define the complex 
$R\Gamma_{\mathrm{Hdg}}(U^\fra/\Delta,\bbLog^N_\fra)$ as follows.

\begin{definition}
	We define the complex $R\Gamma_{\mathrm{Hdg}}(U^\fra/\Delta,\bbLog^N_\fra)$ to be the total complex
	associated to the double complex 
	\begin{align*}
		&\Gamma(X,\sLog^N_{\fra,\bbR}(D)\otimes\sE^\bullet_{X,\bbR}(D))
		 \rightarrow
		\bigoplus_{\substack{J\subset\{1,\ldots,g-1\}\\\absv{J}=1}}\Gamma(X_J,\sLog^N_{\fra,\bbR}(D_J)
	\otimes\sE^\b_{X_J,\bbR}(D_J))\rightarrow\cdots\\
	&\rightarrow
		\bigoplus_{\substack{J\subset\{1,\ldots,g-1\}\\\absv{J}=g-2}}\Gamma(X_J,\sLog^N_{\fra,\bbR}(D_J)
	\otimes\sE^\b_{X_J,\bbR}(D_J))\rightarrow
		\bigoplus_{\substack{J\subset\{1,\ldots,g-1\}\\\absv{J}=g-1}}\Gamma(X_J,\sLog^N_{\fra,\bbR}(D_J)
	\otimes\sE^\b_{X_J,\bbR}(D_J)),
	\end{align*}
	where the differential from the $J'=J\setminus\{i\}$ component
	to the $J$ component is given by 
	\[
		\pair{\wt\ve_i}^*-\pi^*\colon
		\Gamma(X_{J'},\sLog^N_{\fra,\bbR}(D_{J'})
	\otimes\sE^\b_{X_{J'},\bbR}(D_{J'}))
	\rightarrow\Gamma(X_J,\sLog^N_{\fra,\bbR}(D_J)
	\otimes\sE^\b_{X_J,\bbR}(D_J)).
	\]
	We equip a filtration $W_\b$ on $R\Gamma_{\mathrm{Hdg}}(U^\fra/\Delta,\bbLog^N_\fra)$ 
	so that $W_n R\Gamma_{\mathrm{Hdg}}(U^\fra/\Delta,\bbLog^N_\fra)$ for $n\in\bbN$
	is the total complex associated to the double complex
	\begin{align*}
	\Dec(W)_n\Gamma(X,\sLog^N_{\fra,\bbR}(D)\otimes\sE^\b_{X,\bbR}(D))
	& \rightarrow
	\bigoplus_{\substack{J\subset\{1,\ldots,g-1\}\\\absv{J}=1}}
	\Dec(W)_n\Gamma(X_J,\sLog^N_{\fra,\bbR}(D_J)	\otimes\sE^\b_{X_J,\bbR}(D_J))
	\rightarrow\cdots\\
	&\rightarrow
	\bigoplus_{\substack{J\subset\{1,\ldots,g-1\}\\\absv{J}=g-1}}
	\Dec(W)_n\Gamma(X_J,\sLog^N_{\fra,\bbR}(D_J)\otimes\sE^\b_{X_J,\bbR}(D_J)),
	\end{align*}
	and we equip a filtration $F^\b$ 
	so that $F^p(R\Gamma_{\mathrm{Hdg}}(U^\fra/\Delta,\bbLog^N_\fra)\otimes_\bbR\bbC)$ for $p\in\bbN$
	is the total complex associated to the double complex
	\begin{align*}
		\Gamma(X,F^p(\sLog^N_{\fra}(D)\otimes\sE^\b_{X,\bbC}(D)))
		& \rightarrow
		\bigoplus_{\substack{J\subset\{1,\ldots,g-1\}\\\absv{J}=1}}\Gamma(X_J,
		F^p(\sLog^N_{\fra}(D_J)
	\otimes\sE^\b_{X_J,\bbC}(D_J)))\rightarrow\cdots\\
	&\rightarrow
		\bigoplus_{\substack{J\subset\{1,\ldots,g-1\}\\\absv{J}=g-1}}
		\Gamma(X_J,F^p(\sLog^N_{\fra}(D_J)
	\otimes\sE^\b_{X_J,\bbC}(D_J))).
	\end{align*}
\end{definition}

\begin{lemma}\label{lem: A8}
	Let $K^{\bullet,\bullet}$ be a double complex of $\bbR$-vector spaces equipped with an $\bbR$-linear filtration 
	$W_\bullet$ and a $\bbC$-linear filtration $F^\bullet$ on $K^{\bullet,\bullet}_\bbC$.
	Suppose that $K^{i,j}=0$ if $i<0$ or $j<0$, and that the complex $K^{\bullet,j}$ is an absolute $\bbR$-Hodge complex for any $j$.
	Then the total complex $\Tot K^{\bullet,\bullet}$ is also an absolute $\bbR$-Hodge complex.
	In particular, $R\Gamma_{\mathrm{Hdg}}(U^\fra/\Delta,\bbLog_\fra^N)$ is an absolute $\bbR$-Hodge complex.
\end{lemma}

\begin{proof}
	For integers $i$ and $j$, let $A_i^{j}:=H^i(K^{\bullet,j})$ be the cohomology group with respect to the first differentials.
	Then $A_i^{\bullet}$ forms a complex of mixed $\bbR$-Hodge structures.
	We have the spectral sequence
	\[
		E_2^{p,q}=H^p(A_q^\bullet)\Rightarrow H^{p+q}(\Tot K^{\bullet,\bullet}).
	\]
	Since $K^{\bullet,j}$ and $A_q^\bullet$ are absolute $\bbR$-Hodge complexes, we have
	\begin{align*}
		\Gr^W_nE_2^{p,q}=H^p(\Gr^W_nA_q^\bullet),&&\Gr^W_nA_q^j=H^q(\Gr^W_nK^{\bullet,j}).
	\end{align*}
	Thus we have a spectral sequence
	\[
		\Gr^W_nE_2^{p,q}\Rightarrow H^{p+q}(\Gr^W_n\Tot K^{\bullet,\bullet}).
	\]
	As each $\Gr^W_nE_2^{p,q}$ is a pure $\bbR$-Hodge structure of weight $n$, 
	$\Gr^W_nE_\infty^{p,q}$ is also a pure $\bbR$-Hodge structure of weight $n$.
	Since Hodge structures are closed under extensions of filtered modules (\cite{Hub95}*{Lemma 8.1.4} 
	or \cite{PS08}*{Criterion 3.10}), we see that $H^{p+q}(\Gr^W_n\Tot K^{\bullet,\bullet})$ 
	is also a pure $\bbR$-Hodge structure of weight $n$.

	Moreover, we have
	\[
		\dim H^k(\Tot K^{\bullet,\bullet})=\sum_{l\in\bbZ}\dim E_\infty^{l,k-l}=\sum_{l\in\bbZ}\sum_{n\in\bbZ}	
		\dim\Gr^W_nE_\infty^{l,k-l}=\sum_{n\in\bbZ}\dim H^k(\Gr^W_n\Tot K^{\bullet,\bullet}).
	\]
	This shows that the spectral sequence with respect to the filtration $W_\bullet$ on 
	$\Tot K^{\bullet,\bullet}$ degenerates at $E_1$.
	The $E_1$-degeneration of other three spectral sequences can be shown in similar ways.
\end{proof}

\begin{lemma}\label{lem: A9}
	We have an isomorphism of $\bbR$-vector spaces 
	\[
		H^m(U^\fra/\Delta,\bbLog^N_\fra)\cong H^m(R\Gamma_{\mathrm{Hdg}}(U^\fra/\Delta,\bbLog^N_\fra)).
	\]
\end{lemma}

\begin{proof}
    Let $C^\bullet$ be the total complex associated to the double complex 
    \begin{align*}
		R\Gamma(U^\fra,\bbLog^N_{\fra})
		& \rightarrow
		\bigoplus_{\substack{J\subset\{1,\ldots,g-1\}\\\absv{J}=1}}R\Gamma(U^\fra,
		\bbLog^N_{\fra})\rightarrow\cdots
	\rightarrow
		\bigoplus_{\substack{J\subset\{1,\ldots,g-1\}\\\absv{J}=g-1}}
		R\Gamma(U^\fra,\bbLog^N_{\fra}),
	\end{align*}
    where we let $R\Gamma(U^\fra,\bbLog^N_{\fra})$ denote the Dolbeault complex 
    $\Gamma(U^\fra,\sLog^N_{\fra,\bbR}\otimes\sE^\bullet_{U^\fra,\bbR})$ and 
    the differential from the $J'=J\setminus\{i\}$ component to the $J$ component is 
    given by $\langle \varepsilon_i\rangle^*-1$.
    Then, by Lemma \ref{lemma: resolution}, the complex $C^\bullet$ realizes 
    $R\Hom_{\bbZ[\Delta]}(\bbZ,R\Gamma(U^\fra,\bbLog^N_\fra))$, 
    which shows that $H^m(U^\fra/\Delta,\bbLog^N_\fra)\cong H^m(C^\bullet)$. 

    On the other hand, by construction there exists a quasi-isomorphism 
    $R\Gamma_{\mathrm{Hdg}}(U^\fra/\Delta,\bbLog^N_\fra)\isomto C^\bullet$. 
    Hence we have $H^m(R\Gamma_{\mathrm{Hdg}}(U^\fra/\Delta,\bbLog^N_\fra))\cong H^m(C^\bullet)$. 
\end{proof}

By using Lemmas \ref{lem: A8} and \ref{lem: A9}, we equip $H^m(U^\fra/\Delta,\bbLog^N_\fra)$ with 
a mixed $\bbR$-Hodge structure. 
Moreover, by Theorem \ref{thm: A2}, the absolute $\bbR$-Hodge complex 
$R\Gamma_{\mathrm{Hdg}}(U^\fra/\Delta,\bbLog_\fra^N)$ gives a complex in the derived category
$D^+(\MHS_\bbR)$, which we again denote $R\Gamma_{\mathrm{Hdg}}(U^\fra/\Delta,\bbLog_\fra^N)$.

\begin{definition}\label{def: A10}
	We define the equivariant Deligne-Beilinson cohomology by 
	\[
		H_{\sD}^m(U^\fra/\Delta,\bbLog_\fra^N)\coloneqq 
		\Ext^m_{\MHS_\bbR}\bigl(\bbR(0),R\Gamma_{\mathrm{Hdg}}(U^\fra/\Delta,\bbLog_\fra^N)\bigr)
	\]
	for any $m\in\bbN$, where
	\begin{align*}
		\Ext^m_{\MHS_\bbR}\bigl(\bbR(0),R\Gamma_{\mathrm{Hdg}}(U^\fra/\Delta,\bbLog_\fra^N)\bigr)
		&\coloneqq
		H^m\bigl(R\Hom_{\MHS_\bbR}(\bbR(0),R\Gamma_{\mathrm{Hdg}}
		(U^\fra/\Delta,\bbLog_\fra^N))\bigr)\\
		&=
		\Hom_{D^+(\MHS_\bbR)}\bigl(\bbR(0),R\Gamma_{\mathrm{Hdg}}
		(U^\fra/\Delta,\bbLog_\fra^N)[m]\bigr). 
	\end{align*}
\end{definition}

The equivariant 
Deligne-Beilinson cohomology $H_{\sD}^m(U^\fra/\Delta,\bbLog_\fra^N)$ may be calculated explicitly as the cohomology of 
the simple complex associated to the double complex
\begin{equation}\label{eq: M}
	W_0 M^\bullet_\bbR \oplus (F^0\cap W_0) M^\bullet_\bbC\rightarrow  W_0M^\bullet_\bbC, 
    \qquad (x,y)\mapsto x-y
\end{equation}
for $M^\bullet\coloneqq R\Gamma_{\mathrm{Hdg}}(U^\fra/\Delta,\bbLog_\fra^N)$
(see \cite{Bei86}*{\S5} or \cite{Hub95}*{Proposition 8.4.2}).

\begin{proposition}\label{prop: A11}
\begin{enumerate}
\item 
We have a spectral sequence 
\[E^{p,q}_2=H^p(\Delta, H^q(U^\fra,\bbLog^N_\fra))\Rightarrow H^{p+q}(U^\fra/\Delta,\bbLog^N_\fra)\]
in the category of mixed $\bbR$-Hodge structures. 
\item 
We have a spectral sequence 
\[E_2^{p,q}=\Ext^p_{\MHS_{\bbR}}\bigl(\bbR(0), H^{q}(U^\fra/\Delta,\bbLog_\fra^N)\bigr)
\Rightarrow H^{p+q}_{\sD}(U^\fra/\Delta,\bbLog_\fra^N) \]
in the category of $\bbR$-vector spaces. 
\end{enumerate}
\end{proposition}

\begin{proof}
	The spectral sequences are simply the spectral sequences associated to the double complex
	$C^\bullet$ of Lemma \ref{lem: A9}
	and the double complex given in \eqref{eq: M}.  See \cite{Tohoku}*{\S2.4} or 
	\cite{Mil80}*{Appendix B Example 3} for the existence of such spectral sequences
	associated to double complexes.
\end{proof}

\subsection*{Acknowledgements}
The authors express their sincere gratitude to Seidai Yasuda for helpful suggestions 
on the construction of the complex $R\Gamma_{\mathrm{Hdg}}(U^\fra/\Delta,\bbLog_\fra^N)$.

\end{document}